\documentclass[12pt,reqno]{amsart}
\usepackage{amsfonts, amssymb, latexsym, graphics, color}

\usepackage{stmaryrd}		

\newcommand{\codim}{\operatorname*{codim}}
\newcommand{\Spec}{\operatorname*{Spec}}

\newcommand{\Wpred}{\mathsf{Wpred}}
\newcommand{\Spred}{\mathsf{Spred}}

\newcommand{\Schub}{\mathfrak{S}}
\newcommand{\Groth}{\mathfrak{G}}

\newcommand{\rmi}{\mathrm{i}}
\newcommand{\rmrc}{\mathrm{rc}}
\newcommand{\rmirc}{\mathrm{irc}}

\theoremstyle{plain}
\newtheorem{theorem}{Theorem}[section]
\newtheorem{proposition}[theorem]{Proposition}
\newtheorem{corollary}[theorem]{Corollary}
\newtheorem{lemma}[theorem]{Lemma}

\theoremstyle{definition}

\newtheorem{conjecture}[theorem]{Conjecture}
\newtheorem{problem}[theorem]{Question}
\newtheorem{example}[theorem]{Example}

\newtheorem{warning}[theorem]{Notational Warning}

\numberwithin{equation}{section}

\usepackage{tikz}
\usetikzlibrary{matrix,arrows}
\usetikzlibrary{positioning}
\usetikzlibrary{fit}
\usetikzlibrary{patterns}

\newcommand{\id}{\text{id}}

\newcommand{\brur}[2]{(#1\! \leftrightarrow\! #2)}
\newcommand{\cG}{{\mathcal G}}

\newcommand{\dbrac}[1]{{\llbracket #1 \rrbracket}} 	

\newcommand{\pattern}[4]{										
  \raisebox{0.6ex}{
  \begin{tikzpicture}[scale=0.35, baseline=(current bounding box.center), #1]
  \useasboundingbox (0.0,-0.1) rectangle (#2+1.4,#2+1.1);
    \foreach \x/\y in {#4}
      \fill[pattern=north east lines] (\x,\y) rectangle +(1,1);
    \draw (0.01,0.01) grid (#2+0.99,#2+0.99);
    \foreach \x/\y in {#3}
      \filldraw (\x,\y) circle (6pt);
  \end{tikzpicture}}
}

\newcommand{\impattern}[5]{									
  \raisebox{0.6ex}{
  \begin{tikzpicture}[scale=0.35, baseline=(current bounding box.center), #1]
  \useasboundingbox (0.0,-0.1) rectangle (#2+1.4,#2+1.1);
    \foreach \x/\y in {#5}
      \fill[pattern=north east lines] (\x,\y) rectangle +(1,1);
    \draw (0.01,0.01) grid (#2+0.99,#2+0.99);
    \foreach \x/\y in {#4}
      \draw[fill=white] (\x,\y) circle (6pt);
    \foreach \x/\y in {#3}
      \filldraw (\x,\y) circle (6pt);
  \end{tikzpicture}}
}

\newcommand{\imopattern}[6]{									
  \raisebox{0.6ex}{
  \begin{tikzpicture}[scale=0.35, baseline=(current bounding box.center), #1]
  \useasboundingbox (0.0,-0.1) rectangle (#2+1.4,#2+1.1);
    \foreach \x/\y in {#6}
      \fill[pattern=north east lines] (\x,\y) rectangle +(1,1);
    \draw (0.01,0.01) grid (#2+0.99,#2+0.99);
    \foreach \x/\y in {#4}
      \draw[fill=white] (\x,\y) circle (6pt);
    \foreach \x/\y in {#5}
      \draw[fill=white] (\x,\y) circle (10pt);
    \foreach \x/\y in {#3}
      \filldraw (\x,\y) circle (6pt);
  \end{tikzpicture}}
}

\newcommand{\patternsbm}[5]{									
  \raisebox{0.6ex}{
  \begin{tikzpicture}[scale=0.35, baseline=(current bounding box.center), #1]
  \useasboundingbox (0.0,-0.1) rectangle (#2+1.4,#2+1.1);
    \foreach \x/\y in {#4}
      \fill[pattern=north east lines] (\x,\y) rectangle +(1,1);
    \draw (0.01,0.01) grid (#2+0.99,#2+0.99);
    \foreach \x/\y/\z/\w/\A in {#5}
       \fill[color = white!100, opacity=1, rounded corners = 1.5pt] (\x+0.1,\y+0.1) rectangle (\z-0.1,\w-0.1);
    \foreach \x/\y/\z/\w/\A in {#5}
       \draw[color = black, rounded corners = 1.5pt] (\x+0.1,\y+0.1) rectangle (\z-0.1,\w-0.1);
    \foreach \x/\y/\z/\w/\A in {#5}
       \fill[black] (\x/2+\z/2,\y/2+\w/2) node {$\scriptstyle\A$};
    \foreach \x/\y in {#3}
      \filldraw (\x,\y) circle (6pt);

  \end{tikzpicture}}
}

\newcommand{\imopatternsbm}[7]{									
  \raisebox{0.6ex}{
  \begin{tikzpicture}[scale=0.35, baseline=(current bounding box.center), #1]
  \useasboundingbox (0.0,-0.1) rectangle (#2+1.4,#2+1.1);
    \foreach \x/\y in {#6}
      \fill[pattern=north east lines] (\x,\y) rectangle +(1,1);
    \draw (0.01,0.01) grid (#2+0.99,#2+0.99);
    \foreach \x/\y/\z/\w/\A in {#7}
       \fill[color = white!100, opacity=1, rounded corners = 1.5pt] (\x+0.1,\y+0.1) rectangle (\z-0.1,\w-0.1);
    \foreach \x/\y/\z/\w/\A in {#7}
       \draw[color = black, rounded corners = 1.5pt] (\x+0.1,\y+0.1) rectangle (\z-0.1,\w-0.1);
    \foreach \x/\y/\z/\w/\A in {#7}
       \fill[black] (\x/2+\z/2,\y/2+\w/2) node {$\scriptstyle\A$};
    \foreach \x/\y in {#4}
      \draw[fill=white] (\x,\y) circle (6pt);
    \foreach \x/\y in {#5}
      \draw[fill=white] (\x,\y) circle (10pt);
    \foreach \x/\y in {#3}
      \filldraw (\x,\y) circle (6pt);
  \end{tikzpicture}}
}

\newcommand{\patternsbmm}[6]{									
  \raisebox{0.6ex}{											
  \begin{tikzpicture}[scale=0.35, baseline=(current bounding box.center), #1]
  \useasboundingbox (0.0,-0.1) rectangle (#2+1.4,#2+1.1);
    \foreach \x/\y in {#4}
      \fill[pattern=north east lines] (\x,\y) rectangle +(1,1);
    \draw (0.01,0.01) grid (#2+0.99,#2+0.99);
    \foreach \x/\y/\z/\w/\A in {#5}
       \fill[color = white!100, opacity=1, rounded corners = 1.5pt] (\x+0.1,\y+0.1) rectangle (\z-0.1,\w-0.1);
    \foreach \x/\y/\z/\w/\A in {#5}
       \draw[color = black, rounded corners = 1.5pt] (\x+0.1,\y+0.1) rectangle (\z-0.1,\w-0.1);
    \foreach \x/\y/\z/\w/\A in {#6}
       \fill[black] (\x/2+\z/2,\y/2+\w/2) node {$\scriptstyle\A$};
    \foreach \x/\y in {#3}
      \filldraw (\x,\y) circle (6pt);

  \end{tikzpicture}}
}

\newcommand{\patternsbmTenner}[6]{						
  \raisebox{0.6ex}{
  \begin{tikzpicture}[scale=0.35, baseline=(current bounding box.center), #1]
  \useasboundingbox (0.0,-0.1) rectangle (#2+1.4,#2+1.1);
    \foreach \x/\y in {#4}
      \fill[pattern=north east lines] (\x,\y) rectangle +(1,1);
    \draw (0.01,0.01) grid (#2+0.99,#2+0.99);
    \foreach \x/\y/\z/\w/\A in {#5}
       \fill[color = white!100, opacity=1, rounded corners=1.5pt] (\x+0.1,\y+0.1) rectangle (\z-0.1,\w-0.1);
    \foreach \x/\y/\z/\w/\A in {#5}
       \draw[color = black, rounded corners = 1.5pt] (\x+0.1,\y+0.1) rectangle (\z-0.1,\w-0.1);
    \draw[color = black!100, rounded corners=1.5pt] (0.1,4) -- (0.1,3+0.1) -- (2+0.1,3+0.1) -- (2+0.1,1+0.1) -- (4+0.1,1+0.1) -- (4+0.1,0+0.1) -- (5-0.1,0+0.1) -- (5-0.1,2-0.1) -- (3-0.1,2-0.1) -- (3-0.1,4-0.1) -- (1-0.1,4-0.1) -- (1-0.1,5-0.1) -- (0.1,5-0.1) -- (0.1,4);
    \fill[color = white!100, opacity=1, rounded corners=1.5pt] (0.1,4) -- (0.1,3+0.1) -- (2+0.1,3+0.1) -- (2+0.1,1+0.1) -- (4+0.1,1+0.1) -- (4+0.1,0+0.1) -- (5-0.1,0+0.1) -- (5-0.1,2-0.1) -- (3-0.1,2-0.1) -- (3-0.1,4-0.1) -- (1-0.1,4-0.1) -- (1-0.1,5-0.1) -- (0.1,5-0.1) -- (0.1,4);
    \foreach \x/\y/\z/\w/\A in {#6}
       \fill[black] (\x/2+\z/2,\y/2+\w/2) node {$\scriptstyle\A$};
    \foreach \x/\y in {#3}
      \filldraw (\x,\y) circle (6pt);

  \end{tikzpicture}}
}

\newcommand{\vinc}[3]{
\begin{tikzpicture}[baseline, inner sep = 0mm]

	\begin{scope}[yshift = 3]
	
	\foreach \x/\y in {#2}
	{
		\node at (\x*0.2,-0.14)  [label=$\y$] {};  
	}
	
	\foreach \z in {#3}
	{
		\ifnum 0<\z
			\ifnum \z<#1
				\draw[thick] (\z*0.2-0.07,-0.17) -- (\z*0.2+0.27,-0.17);
			\fi
		\fi
		
		\ifnum 0=\z
			\draw[thick] (0.08,0.1) -- (0.08,-0.17) -- (0.26,-0.17);
		\fi
		
		\ifnum \z=#1
			\draw[thick] (\z*0.2+0.13,0.1) -- (\z*0.2+0.13,-0.17) -- (\z*0.2-0.05,-0.17);
		\fi
	}
	\end{scope}
\end{tikzpicture}
}

\newcommand{\vinci}[3]{
\begin{tikzpicture}[baseline, scale = 0.8, , inner sep = 0mm]

	\begin{scope}[yshift = 3]
	
	\foreach \x/\y in {#2}
	{
		\node at (\x*0.2,-0.14)  [label=$\scriptstyle{\y}$] {};  
	}
	
	\foreach \z in {#3}
	{
		\ifnum 0<\z
			\ifnum \z<#1
				\draw[thick] (\z*0.2-0.07,-0.17) -- (\z*0.2+0.27,-0.17);
			\fi
		\fi
		
		\ifnum 0=\z
			\draw[thick] (0.08,0.1) -- (0.08,-0.17) -- (0.26,-0.17);
		\fi
		
		\ifnum \z=#1
			\draw[thick] (\z*0.2+0.13,0.1) -- (\z*0.2+0.13,-0.17) -- (\z*0.2-0.05,-0.17);
		\fi
	}
	\end{scope}
\end{tikzpicture}
}

\newcommand{\bivinc}[4]{
\begin{tikzpicture}[baseline, inner sep = 0mm]

	\begin{scope}[yshift = -2]
	
	\foreach \x/\y in {#2}
	{
		\node at (\x*0.2,0.15)   [label=$\x$] {};  
		\node at (\x*0.2,-0.14)  [label=$\y$] {};  
	}
	
	\foreach \z in {#3}
	{
		\ifnum 0<\z
			\ifnum \z<#1
				\draw[thick] (\z*0.2-0.07,-0.17) -- (\z*0.2+0.27,-0.17);
			\fi
		\fi
		
		\ifnum 0=\z
			\draw[thick] (0.08,0.1) -- (0.08,-0.17) -- (0.21,-0.17);
		\fi
		
		\ifnum \z=#1
			\draw[thick] (\z*0.2+0.13,0.1) -- (\z*0.2+0.13,-0.17) -- (\z*0.2,-0.17);
		\fi
	}
	
	\foreach \z in {#4}
	{
		\ifnum 0<\z
			\ifnum \z<#1
				\draw[thick] (\z*0.2-0.07,0.47) -- (\z*0.2+0.27,0.47);
			\fi
		\fi
		
		\ifnum 0=\z
			\draw[thick] (0.08,0.21) -- (0.08,0.47) -- (0.21,0.47);
		\fi
		
		\ifnum \z=#1
			\draw[thick] (\z*0.2+0.13,0.21) -- (\z*0.2+0.13,0.47) -- (\z*0.2,0.47);
		\fi
	}
	\end{scope}
\end{tikzpicture}
}

\newcommand{\bivincs}[5]{
\begin{tikzpicture}[baseline, inner sep = 0mm]

	\begin{scope}[yshift = -2]

	\foreach \x/\w in {#2}
	{
		\node at (\x*0.2,-0.14)   [label=$\w$] {}; 
	}
	
	\foreach \x/\w in {#3}
	{
		\node at (\x*0.2,0.15) [label=$\w$] {}; 
	}
	
	\foreach \z in {#4}
	{
		\ifnum 0<\z
			\ifnum \z<#1
				\draw[thick] (\z*0.2-0.07,-0.17) -- (\z*0.2+0.27,-0.17);
			\fi
		\fi
		
		\ifnum 0=\z
			\draw[thick] (0.08,0.1) -- (0.08,-0.17) -- (0.26,-0.17);
		\fi
		
		\ifnum \z=#1
			\draw[thick] (\z*0.2+0.13,0.1) -- (\z*0.2+0.13,-0.17) -- (\z*0.2-0.05,-0.17);
		\fi
	}
	
	\foreach \z in {#5}
	{
		\ifnum 0<\z
			\ifnum \z<#1
				\draw[thick] (\z*0.2-0.07,0.47) -- (\z*0.2+0.27,0.47);
			\fi
		\fi
		
		\ifnum 0=\z
			\draw[thick] (0.08,0.21) -- (0.08,0.47) -- (0.26,0.47);
		\fi
		
		\ifnum \z=#1
			\draw[thick] (\z*0.2+0.13,0.21) -- (\z*0.2+0.13,0.47) -- (\z*0.2-0.05,0.47);
		\fi
	}
	\end{scope}
\end{tikzpicture}
}

\newcommand{\bivinci}[4]{
\begin{tikzpicture}[baseline, scale = 0.8, inner sep = 0mm]

	\begin{scope}[yshift = -2]
	
	\foreach \x/\y in {#2}
	{
		\node at (\x*0.2,0.15)   [label=$\scriptstyle{\x}$] {};  
		\node at (\x*0.2,-0.14)  [label=$\scriptstyle{\y}$] {};  
	}
	
	\foreach \z in {#3}
	{
		\ifnum 0<\z
			\ifnum \z<#1
				\draw[thick] (\z*0.2-0.07,-0.17) -- (\z*0.2+0.27,-0.17);
			\fi
		\fi
		
		\ifnum 0=\z
			\draw[thick] (0.08,0.1) -- (0.08,-0.17) -- (0.26,-0.17);
		\fi
		
		\ifnum \z=#1
			\draw[thick] (\z*0.2+0.13,0.1) -- (\z*0.2+0.13,-0.17) -- (\z*0.2-0.05,-0.17);
		\fi
	}
	
	\foreach \z in {#4}
	{
		\ifnum 0<\z
			\ifnum \z<#1
				\draw[thick] (\z*0.2-0.07,0.47) -- (\z*0.2+0.27,0.47);
			\fi
		\fi
		
		\ifnum 0=\z
			\draw[thick] (0.08,0.21) -- (0.08,0.47) -- (0.26,0.47);
		\fi
		
		\ifnum \z=#1
			\draw[thick] (\z*0.2+0.13,0.21) -- (\z*0.2+0.13,0.47) -- (\z*0.2-0.05,0.47);
		\fi
	}
	\end{scope}
\end{tikzpicture}
}

\begin{document}

\title{Which Schubert varieties are local complete intersections?}

\author[\'Ulfarsson]{Henning \'Ulfarsson}
\author[Woo]{Alexander Woo}

\address[\'Ulfarsson]{School of Computer Science, Reykjav\'ik University, Menntavegi 1, 101 Reykjav\'ik, Iceland}
\address[Woo]{Department of Mathematics, University of Idaho, PO Box 441103, Moscow, ID 83844-1103, United States of America}

\email{henningu@ru.is, awoo@uidaho.edu}

\date{\today}
\begin{abstract}

We characterize by pattern avoidance the Schubert varieties for
$\mathrm{GL}_n$ which are local complete intersections (lci).  For
those Schubert varieties which are local complete intersections, we
give an explicit minimal set of equations cutting out their
neighborhoods at the identity.  Although the statement of our
characterization only requires ordinary pattern avoidance, showing
that the Schubert varieties not satisfying our conditions are not lci
appears to require more general notions of pattern
avoidance.  The Schubert varieties defined by inclusions, originally
introduced by Gasharov and Reiner, turn out to be an important
subclass, and we further develop some of their combinatorics.
Applications include formulas for Kostant polynomials and
presentations of cohomology rings for lci Schubert varieties.

\end{abstract}
\maketitle

\setcounter{tocdepth}{1}
\tableofcontents

\section{Introduction}
The main purpose of this paper is to characterize the Schubert
varieties which are local complete intersections.

Let $G=\mathrm{GL}_n(\mathbb{C})$ and $B$ a Borel subgroup, which we
take to be the upper triangular matrices.  The quotient $G/B$ is a
projective variety known as the {\bf flag variety}; its points
correspond to {\bf complete flags}, which are chains of subspaces
$F_\bullet=\langle0\rangle\subsetneq F_1\subsetneq \cdots\subsetneq
F_{n-1}\subsetneq \mathbb{C}^n$ with $\dim F_i=i$ for all $i$.  The group
$G$, and hence its subgroup $B$, acts on $G/B$ by left multiplication.
Given a permutation $w$, the {\bf Schubert variety} $X_w$ is the closure of
the orbit $BwB/B$ of the permutation matrix for $w$ under the action
of $B$.

A local ring $R$ is a {\bf local complete intersection} ({\bf lci}) if
it is the quotient of some regular local ring by an ideal generated by
a regular sequence.  A variety (or, more generally, a scheme) is lci
if every local ring is lci.  Since regular local rings are
automatically lci, smooth varieties are automatically lci.
Furthermore, lci varieties are automatically Gorenstein and hence
Cohen--Macaulay.  Thus, being lci can be viewed as saying that the
singularities are in some sense mild.

The characterization of smooth Schubert varieties (for the flag
variety) began with a geometric characterization by Ryan~\cite{Ryan}
and a combinatorial characterization by Wolper~\cite{Wol}.  Lakshmibai
and Sandhya~\cite{LakSan} then discovered (independently of the work
of Ryan and Wolper) that the smooth Schubert varieties $X_w$ can be
characterized by the combinatorial notion of {\bf pattern avoidance}.
A permutation $v\in S_m$ {\bf embeds} in $w\in S_n$ if there are some
$m$ entries of $w$, say at indices $i_1<\cdots<i_m$, in the relative
order given by $v$, meaning that $w(i_j)<w(i_k)$ if and only if
$v(j)<v(k)$.  If $v$ does not embed in $w$, then $w$ is said to {\bf
avoid} $v$.  Lakshmibai and Sandhya showed that $X_w$ is smooth if and
only if $w$ avoids both of the permutations 3412 and 4231 (written in
$1$-line notation).

More recently, Yong and the second author characterized the
permutations $w$ for which $X_w$ is Gorenstein~\cite{WYGor}.  This
characterization cannot be given purely in terms of pattern avoidance
but requires a more complicated generalization, either {\bf interval
pattern avoidance} (called Bruhat-restricted pattern avoidance in the
original) or alternatively {\bf bivincular patterns} as explained
in \cite{Ulf}.  However, the lci Schubert varieties can be
characterized by ordinary pattern avoidance.  More precisely, we prove
the following theorem.

\begin{theorem}
\label{thm:main}
The Schubert variety $X_w$ is lci if and only if $w$ avoids the six
patterns $53241$, $52341$, $52431$, $35142$, $42513$, and $351624$.
\end{theorem}

For convenience we work over $\mathbb{C}$ in this paper, but our
results and proofs hold over $\mathbb{Z}$ and hence over any field.

Our proof for this theorem carries out the general strategy for any
local property suggested by the work of Yong and the second
author~\cite{WYGov}.  Let $\Omega_v$ denote the {\bf opposite Schubert
  cell}, which is the orbit $B_-vB/B$ of the permutation matrix $v$
under the opposite Borel group $B_-$ of lower triangular matrices.
Furthermore, let $\mathcal{N}_{v,w}$ denote $\Omega_v\cap X_w$.  It is
a lemma of Kazhdan and Lusztig~\cite[Lemma A.4]{KazLusLem} that the
point $vB/B$ given by a permutation matrix $v$ has a neighborhood in
$X_w$ which isomorphic to $\mathcal{N}_{v,w}\times
\mathbb{C}^{\ell(v)}$.

For a permutation $w$ avoiding the six patterns in Theorem~\ref{thm:main}, we study explicit
equations for $\mathcal{N}_{\id,w}$ as a subvariety of $\Omega_{\id}$
and explicitly find $\codim(X_w)$ generators for its defining ideal,
hence showing that $X_w$ is lci at the identity.  This suffices since
the locus of non-lci points on any scheme is closed; since this locus
on a Schubert variety is $B$-invariant, it must therefore be a union
of Schubert subvarieties and hence include the identity.  We identify
these generators based on the combinatorics of the {\bf essential set}
of $w$, which was originally defined by Fulton~\cite{FulMSV} to give a
minimal set of generators for the ideal defining a matrix Schubert
variety.  The combinatorics of the essential set were later further
studied by Eriksson and Linusson~\cite{ErikLin}.

To show that a permutation containing one of the six patterns is not
lci, we first identify two infinite families and eleven isolated pairs
$(u,v)$ such that $\mathcal{N}_{u,v}$ is not lci.  The two infinite
families define irreducible components of the singular locus (as shown
independently by Billey and Warrington~\cite{BilWar},
Cortez~\cite{Cor}, Kassel, Lascoux, and Reutenauer~\cite{KLR}, and
Manivel~\cite{ManLocus} following earlier work of
Gasharov~\cite{Gash}), and the equations for $\mathcal{N}_{u,v}$ in
these cases were identified independently by Manivel~\cite{Man} and by
Cortez~\cite{Cor}.  Now a theorem of Yong and the second
author~\cite[Cor.\ 2.7]{WYGov} implies that if $w$ fails to interval
avoid any of these pairs $[u,v]$, then $w$ will not be lci.

We then show that any permutation containing one of our six forbidden
patterns must actually contain an interval either from one of the
infinite families or from our list of eleven isolated cases.  To
formulate this proof, we must first translate both sets of avoidance
conditions into {\bf marked mesh patterns}, previously defined by the
first author~\cite{Ulf}.  Marked mesh patterns generalize mesh
patterns, which were introduced in full generality by Br\"and\'en and
Claesson~\cite{BraCla}, although special cases had previously been
implicitly used, for example in the aforementioned determination of
the singular locus of Schubert varieties.  The original motivation for
defining mesh patterns was to write various permutation statistics as
linear combinations of permutation patterns. They have since been
shown to characterize the permutations satisfying various properties.
For example, the {\bf simsun} permutations, introduced
in~\cite{SimSun} and later named after Simion and Sundaram, are
characterized by the avoidance of the mesh pattern
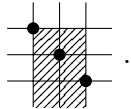
\[
\pattern{scale=1}{ 3 }{ 1/3, 2/2, 3/1 }{1/0, 1/1, 1/2, 2/0, 2/1, 2/2}.
\]
Furthermore, it is not
hard to see that every interval pattern can be described as a mesh pattern,
as shown by the first author~\cite[Lemma 22]{Ulf}. Marked mesh patterns are similar to mesh patterns but allow more control over the number of elements occupying a particular region in the graph of a permutation.

Another related result is the characterization of Schubert varieties
which are {\bf defined by inclusions}, due to Gasharov and
Reiner~\cite{GasRei}.  They show that $X_w$ is defined by inclusions
if $w$ avoids 4231, 35142, 42513, and 351624.  As one can tell from
the patterns involved, our theorem implies that Schubert varieties
defined by inclusions are lci, which was previously unknown.  Indeed,
the Schubert varieties defined by inclusions turn out to be an
important special case in proving the sufficiency of our pattern
avoidance conditions.  In particular, we use the essential set to
canonically associate a permutation defined by inclusions to any
permutation indexing an lci Schubert variety.  However, unlike the
property of being lci, which is entirely intrinsic to the Schubert
variety, it is possible, at least {\em a priori} from the definition,
for $X_v$ and $X_w$ to be abstractly isomorphic despite one being
defined by inclusions and the other not.  It is not known if this
actually occurs.

One obvious further problem is that of generalizing our theorem to
Schubert varieties for other semisimple Lie groups.  Unfortunately,
while there are generalizations of pattern avoidance~\cite{BilBra,
BilPos} and interval pattern avoidance~\cite{WooIPA} for these other
Schubert varieties, explicit equations for the analogues of
$\mathcal{N}_{u,v}$ are not generally known.  However, we have a
possible hint towards the answer.  Hultman, Linusson, Shareshian, and
Sj\"ostrand~\cite{HLSS} showed that the number of chambers in the
inversion arrangement for a permutation $w$ equals the number of
permutations less than or equal to $w$ in Bruhat order if and only if
$X_w$ is defined by inclusions.  Hultman~\cite{Hul} has extended this
result to other finite reflection groups, although his
characterization is not in terms of pattern avoidance.

One could also hope to determine explicitly the (non-)lci locus of any
Schubert variety.  We conjecture that our list of interval patterns
fully specifies the non-lci locus.  In principle, this conjecture (or
its correct version) can be proven by identifying explicit generators
for all of the lci Kazhdan--Lusztig varieties $\mathcal{N}_{x,w}$ just
as we do here for the case where $x$ is the identity and $w$ avoids
the given patterns.  However, at least at first glance, the amount of
combinatorial analysis required seems daunting.  Another possible
extension of our work would be to identify, for each $k>0$, those
Schubert varieties which fail to be lci by at most $k$ excess
generators (for the ideal generating $\mathcal{N}_{\id,w}$).  It would
be interesting to know if this property is characterized by pattern
avoidance, and furthermore avoidance of a finite number of patterns,
for all $k$.

We remark on several further applications of our results.  First we
point out how various implications between properties of singularities
can be derived purely combinatorially on Schubert varieties by
containment of patterns.  Also, we can recover results on lci matrix
Schubert varieties due to Hsiao~\cite{Hsiao}.  Furthermore, our proof
of sufficiency gives explicit equations for $\mathcal{N}_{\id,w}$ when
$w$ is lci.  We describe two applications of this result.  First, we
give explicit formulas of the Kostant--Kumar
polynomials~\cite{KosKumCoh,KosKumK} for both cohomology and
$K$-theory at the identity (which are equivalent to certain specific
specializations of the double Schubert and Grothendieck polynomials of
Lascoux and Sch\"utzenberger~\cite{LSSchub,LSGroth}) in the case where
$w$ is lci.  This calculation in the smooth case recovers a small part
of a result of Kumar~\cite{KumSmooth} characterizing smooth points on
Schubert varieties using the Kostant--Kumar polynomials and suggests a
possible similar characterization of lci points as well as a potential
local definition for being defined by inclusions.  Second, we use the
result of Aky{\i}ld{\i}z, Lascoux, and Pragacz~\cite{ALP} identifying the
cohomology ring $H^*(X_w)$ with a particular quotient of
$\mathcal{N}_{\id,w}$ to extend the presentation of $H^*(X_w)$ in the
case $X_w$ is defined by inclusions, due to Gasharov and
Reiner~\cite{GasRei}, to a presentation of $H^*(X_w)$ for all lci
Schubert varieties.

Furthermore, there has been some recent renewed interest in lci
varieties in the context of jet schemes spurred by Musta{\c{t}}{\u{a}}'s
theorem~\cite{Mus} that an lci variety has an irreducible jet scheme
if and only if the variety has canonical singularities.  The lci
Schubert varieties should provide a useful class of examples for
understanding and possibly extending this theorem since they have a
well understood resolution of singularities, the Bott--Samelson
resolution~\cite{BottSam}.  (This is only a resolution of
singularities in a weak sense since the image of the exceptional
locus contains nonsingular points.)  The jet schemes for the special
case of determinantal varieties have already been worked out by
Mustata~\cite{Mus}, Yuen~\cite{Yuen}, Kosir and
Setharuman~\cite{KosSet}, and Docampo~\cite{Doc}.

In addition, Anderson and Stapledon~\cite{AndSta} have recently
developed an interpretation of classes in equivariant cohomology of a
smooth variety $Y$ as representing subvarieties of the arc space
$J_\infty Y$.  When $X$ is an equivariant local complete intersection
in $Y$, they show that powers of the cohomology class $[X]\in
H^*_T(Y)$ represent certain subvarieties of $J_\infty X$.  Our
proof shows that the lci Schubert varieties are equivariant local
complete intersections; hence they provide an interesting class of
examples for their theory.  In addition, they relate $[X]^{m+1}$ to
the class of the jet scheme $J_m X$.  Hence our formula for the
Kostant--Kumar polynomials at the identity may help in determining
local equations for $J_m X_w$ in the case where $X_w$ is a local
complete intersection.

Our paper is organized as follows.  Section 2 gives definitions and
basic facts about lci varieties, Schubert varieties, equations
defining Schubert varieties, and various notions of pattern avoidance.
In Section 3, we prove some combinatorial results on the essential
sets of permutations which are defined by inclusions as well as
permutations avoiding our given patterns.  Some of these results may
be of independent interest.  Section 4 proves that permutations
avoiding our given patterns are lci using the combinatorics of Section
3.  Section 5 proves that permutations containing our given patterns
are not lci.  We describe various applications in Section 6 and pose
a number of open questions in Section 7.

The first author thanks Einar Steingr\'imsson for enthusiastic support and
discussions. He also thanks Emil{\'i}a Halld{\'o}rsd{\'o}ttir for always
believing in him.
For the duration of this project, the first author was employed by Reykjav\'ik
University and supported by grant no.\ 090038011--3 from the Icelandic Research Fund.

The second author thanks Alexander Yong for help with the computations
which led to this project, Vic Reiner for encouragement and helpful
suggestions, and Axel Hultman and Winfried Bruns for answers to
technical questions.  He also thanks Kaisa Taipale for company during
several ``antisocial writing'' sessions during which a significant
portion of this paper was written.  For most of the duration of this
project, the second author was employed by St.~Olaf College, which
provided travel funding and other support that made this work
possible.

\section{Definitions}
Throughout this paper we will use the notation $\dbrac{a,b}$ to denote the set
$$\dbrac{a,b}:=\{x\in\mathbb{Z}\mid a\leq x\leq b\}.$$ In particular,
$\dbrac{a,a}=\{a\}$, and $\dbrac{a,b}=\emptyset$ when $b<a$.  Also, $\#S$ will denote the number of elements in $S$.

\subsection{Local complete intersections} 

Let $R$ be a local ring.  The ring $R$ is called a {\bf local complete
  intersection} ({\bf lci}) if there exists a regular local ring $S$ (meaning one
where, letting $m$ be the maximal ideal of $S$, the dimension of
$m/m^2$ as an $S/m$ vector space is the same as the Krull dimension of
$S$) and an ideal $I$ generated by a regular sequence on $S$ such that $R=S/I$.  Regular rings
are trivially local complete intersections, and, by the Kozsul resolution, local complete
intersections are Gorenstein and hence Cohen-Macaulay.  It turns out
that the choice of regular ring $S$ is irrelevant; if $R\cong S/I$
with $R$ a local complete intersection and $S$ any regular local ring, then $I$ will
always be generated by a regular sequence.  Furthermore, whether $R$
is or is not a local complete intersection can be detected purely by
using the $\mathrm{Ext}$ functor on $R$ and its residue field $k$.  For details
and other basic facts about local complete intersections, see the book
by Bruns and Herzog~\cite{winfried1998cohen}.

An algebraic variety or scheme $X$ is called a {\bf local complete
  intersection} ({\bf lci}) if for each point $p\in X$, the local ring
$\mathcal{O}_{X,p}$ is a local complete intersection.  For any
variety, the locus $V$ of points $p$ for which $\mathcal{O}_{X,p}$ is
{\it not} lci is a {\it closed} set (in the Zariski topology).

Note that if $S=\mathbb{C}[x_1,\ldots,x_n]$, $I$ is an ideal of $S$
generated by $k$ elements, and $\dim \Spec S/I=n-k$, then $\Spec S/I$
is automatically lci, as localization can never increase the number of
generators needed for an ideal.

\subsection{Schubert varieties}

Let $G=\mathrm{GL}_n(\mathbb{C})$, which we think of explicitly as the group of
invertible $n\times n$ matrices, and let $B$, $B_-$, and $T$ denote
respectively the subgroups of invertible upper triangular, lower
triangular, and diagonal matrices.  The {\bf flag variety} is the
quotient space $G/B$; upon a choice of basis for $\mathbb{C}^n$, a
point $gB\in G/B$ is naturally identified with a {\bf flag} $F_\bullet:
\langle 0\rangle\subsetneq F_1\subsetneq\cdots\subsetneq
F_{n-1}\subsetneq\mathbb{C}^n$ by taking $F_i$ to be the span of the
first $i$ columns of any coset representative $g$ of $gB$.

Let $w\in S_n$ be a permutation.  We think of $w$ as a permutation
matrix with $1$'s at row $w(j)$ (counted from the top) and column $j$
for each $j$ and with $0$'s everywhere else.  We let $e_w$ denote the {\bf Schubert point} which is the coset $wB\in G/B$.  The orbit $BwB/B\subseteq G/B$
is known as a {\bf Schubert cell}, and its closure
$X_w=\overline{BwB/B}\subseteq G/B$ is a {\bf Schubert variety}.  The
orbit $B_-wB/B\subseteq G/B$ is known as an {\bf opposite Schubert cell}
and denoted $\Omega^\circ_w$.  Our conventions are such that $X_{\id}$
is a point and $X_{w_0}$ (where $w_0$ is the long permutation defined by
$w_0(i)=n+1-i$ for all $i$) is all of $G/B$.

\subsection{Rank conditions for Schubert varieties}

Schubert varieties can be alternatively defined as the set of points
representing flags satisfying certain intersection conditions with the
standard flag or equivalently as the set of $B$-cosets with representatives
satisfying certain rank conditions on southwest submatrices.  For a
permutation $w$, define the {\bf rank function} $r_w$ by
\[
r_w(p,q)=\#\{k\leq q\mid w(k)\geq p\}.
\]
Let $E_\bullet$ be the flag
where $E_p$ is the span of the first $p$ standard basis vectors; this
flag is known as the {\bf standard flag}. A flag $F_\bullet$ represents
a point $gB$ in the Schubert variety $X_w$ if and only if
$\dim(E_p\cap F_q)\geq q-r_w(p+1,q)$ for all $p,q\in \dbrac{1,n}$.
This is equivalent to the rank of the southwest $(n+1-p)\times q$
submatrix (consisting of the $n+1-p$ bottommost rows and $q$ leftmost
columns) of any coset representative $g$ of $gB$ being at most
$r_w(p,q)$ for all $p$ and $q$.

Many of these rank conditions are redundant, and Fulton~\cite{FulMSV}
showed that the minimal set of conditions defining any Schubert
variety are those from what he called the {\bf essential set}.  The
{\bf Rothe diagram} of $w$ is the set of boxes (which we can think of
as being drawn over the permutation matrix)
\[
\{(p,q)\in\dbrac{1,n}\times\dbrac{1,n}\mid w(q)<p, w^{-1}(p)>q\}.
\]
The diagram can
be described visually as follows.  For each $q\in\dbrac{1,n}$, draw
a dot $\bullet$ at $(w(q),q)$. (Coordinates
are given in matrix notation.) For each dot draw the ``hook'' that
extends north and east of that dot.  The boxes not in any hook are the
boxes of the diagram.  The {\bf essential set} $E(w)$ is the set of
boxes in $D(w)$ which are northeast corners in some connected
component of $D(w)$.  To be precise,
\[
E(w)=\{(p,q)\in D(w) \mid (p,q+1)\not\in D(w), (p-1,q)\not\in
D(w)\},
\]
and a matrix $g$ represents a point $gB\in X_w$ if and only
if the southwest $(n+1-p)\times q$ submatrices of $g$ have rank at most
$r_w(p,q)$ for all $(p,q)\in E(w)$.  Furthermore, $E(w)$ is the
minimal subset of $\dbrac{1,n}\times \dbrac{1,n}$ with this property; no subset of
$E(w)$ will correctly define $X_w$.

\begin{warning}
There are a number of different conventions for the essential set in
the literature.  In particular, our convention is different from the
original one used by Fulton~\cite{FulMSV} and is known in some sources
as the coessential set.
\end{warning}

\begin{example}
\label{ex:ess-819372564}
Let $w=819372564$.  Then the diagram and essential set of $w$ are as
in Figure~\ref{fig:ess-819372564}.  In particular, $E(w)=\{(2,2),
(4,4), (4,6), (6,7), (9,2)\}$.\qed

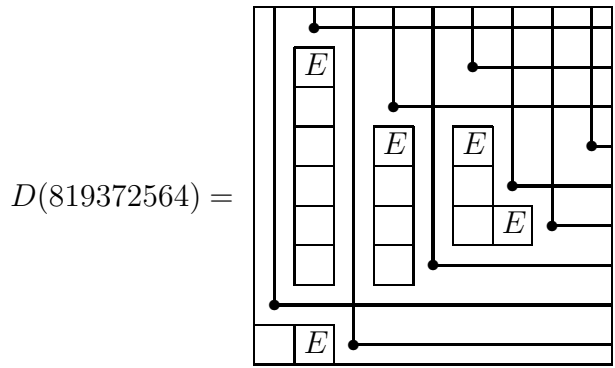
\begin{figure}[htbp]
\[\begin{picture}(200,140)
\put(-55,60){$D(819372564)=$}
\put(37.5,0){\framebox(135,135)}
\put(37.5,15){\line(1,0){30}}
\put(52.5,15){\line(0,-1){15}}
\put(67.5,15){\line(0,-1){15}}
\put(52.5,30){\line(0,1){90}}
\put(67.5,30){\line(0,1){90}}
\put(52.5,30){\line(1,0){15}}
\put(52.5,45){\line(1,0){15}}
\put(52.5,60){\line(1,0){15}}
\put(52.5,75){\line(1,0){15}}
\put(52.5,90){\line(1,0){15}}
\put(52.5,105){\line(1,0){15}}
\put(52.5,120){\line(1,0){15}}
\put(82.5,30){\line(0,1){60}}
\put(97.5,30){\line(0,1){60}}
\put(82.5,30){\line(1,0){15}}
\put(82.5,45){\line(1,0){15}}
\put(82.5,60){\line(1,0){15}}
\put(82.5,75){\line(1,0){15}}
\put(82.5,90){\line(1,0){15}}

\put(112.5,45){\line(0,1){45}}
\put(127.5,45){\line(0,1){45}}
\put(142.5,45){\line(0,1){15}}
\put(112.5,45){\line(1,0){30}}
\put(112.5,60){\line(1,0){30}}
\put(112.5,75){\line(1,0){15}}
\put(112.5,90){\line(1,0){15}}

\thicklines
\put(45,22.5){\circle*{4}}
\put(45,22.5){\line(1,0){127.5}}
\put(45,22.5){\line(0,1){112.5}}

\put(60,127.5){\circle*{4}}
\put(60,127.5){\line(1,0){112.5}}
\put(60,127.5){\line(0,1){7.5}}

\put(75,7.5){\circle*{4}}
\put(75,7.5){\line(1,0){97.5}}
\put(75,7.5){\line(0,1){127.5}}

\put(90,97.5){\circle*{4}}
\put(90,97.5){\line(1,0){82.5}}
\put(90,97.5){\line(0,1){37.5}}

\put(105,37.5){\circle*{4}}
\put(105,37.5){\line(1,0){67.5}}
\put(105,37.5){\line(0,1){97.5}}

\put(120,112.5){\circle*{4}}
\put(120,112.5){\line(1,0){52.5}}
\put(120,112.5){\line(0,1){22.5}}

\put(135,67.5){\circle*{4}}
\put(135,67.5){\line(1,0){37.5}}
\put(135,67.5){\line(0,1){67.5}}

\put(150,52.5){\circle*{4}}
\put(150,52.5){\line(1,0){22.5}}
\put(150,52.5){\line(0,1){82.5}}

\put(165,82.5){\circle*{4}}
\put(165,82.5){\line(1,0){7.5}}
\put(165,82.5){\line(0,1){52.5}}

\put(56,4.5){$E$}

\put(56,109.5){$E$}
\put(86,79.5){$E$}
\put(116,79.5){$E$}
\put(131,49.5){$E$}

\end{picture}\]
\caption{\label{fig:ess-819372564} Diagram and essential set for $w=819372564$.}
\end{figure}
\end{example}

\subsection{Local neighborhoods, Kazhdan--Lusztig varieties, and explicit equations} \label{ss:loc-KL-ee}

We now explain how local properties of $X_w$ can be explicitly
calculated.  The contents of this section can be found in greater
detail (and with proofs) in~\cite[Section 3.2]{WYGov}, and the ideas
behind it can partially be traced back to Fulton~\cite{FulMSV}.

The opposite Schubert cell $\Omega^\circ_{\id}\subseteq G/B$ is an open
affine neighborhood of $e_{\id}$, and, given any $v$,
$v\Omega^\circ_{\id}$ is an affine neighborhood of the Schubert point
$e_v$.  Since $B$ acts on any Schubert variety $X_w$, it suffices to
study one point in each $B$-orbit, which we take to be the Schubert
point $e_v$, so it suffices to study $X_w\cap v\Omega^\circ_{\id}$.
Moreover, by~\cite[Lemma A.4]{KazLusLem},
\[
X_w\cap v\Omega^\circ_{\id}\cong (X_w\cap\Omega^\circ_v)\times\mathbb{C}^{\ell(v)}.
\]
Hence, to check if $X_w$ is locally a complete intersection at $e_v$, it suffices to study the {\bf Kazhdan--Lusztig variety}
\[
\mathcal{N}_{v,w}:= X_w\cap\Omega^\circ_v.
\]

We now give equations which define $\mathcal{N}_{v,w}$
scheme-theoretically.  Given a permutation $v$, let $S_v$ be the
polynomial ring whose variables are labelled by the boxes in the
diagram of $v$, so $S_v=\mathbb{C}[z_{p,q}]_{(p,q)\in D(v)}$.
Furthermore, let $M_v$ be the matrix with a $1$ as the entry at
$(v(i),i)$ for each $i$, $z_{p,q}$ at $(p,q)\in D(v)$, and $0$'s
everywhere else.

For any subsets $A$ and $B$ of $\dbrac{1,n}$ such that both $A$ and $B$
have the same number of elements, let $d^{(v)}_{A,B}$ denote the
minor of $M_v$ which is the determinant of the square matrix whose
rows are the rows of $M_v$ indexed by elements of $A$ and whose
columns are the columns of $M_v$ indexed by elements of $B$.  We will
refer to $d^{(v)}_{A,B}$ as a {\bf generalized Pl\"ucker coordinate}.

The ring $S_v$ has a grading where $\deg z_{p,q}=p-v(q)$.  Note that
$$\deg d^{(v)}_{A,B}=\sum_{p\in A} p-\sum_{q\in B} v(q).$$  Furthermore,
$d^{(v)}_{A,B}=0$ if $\deg d^{(v)}_{A,B}<0$, and
$d^{(v)}_{A,B}\in\{-1,0,1\}$ if $\deg d^{(v)}_{A,B}=0$.

Given $p,q,r\in\dbrac{1,n}$, let $I^{(v)}_{(p,q,r)}$ be the
ideal of $S_v$ generated by all $d^{(v)}_{A,B}$ where
$A\subseteq\dbrac{p,n}$, $B\subseteq\dbrac{1,q}$, and
$\#A=\#B=r+1$; these are all the $r+1$ size minors of the rectangular
submatrix consisting of all entries (weakly) SW of $(p,q)$.  Given a
permutation $w$, let
\[
I_{v,w}=\sum_{(p,q)\in E(w)}
I^{(v)}_{(p,q,r_w(p,q))}.
\]
The following is a restatement
of~\cite[Prop.~3.1]{WYGov}; this Proposition was first stated in a
less concise form in \cite{FulMSV}.

\begin{proposition}
\label{prop:alleqns}
The Kazhdan--Lusztig variety
\[
\mathcal{N}_{v,w}\cong\Spec S_v/I_{v,w}.
\]
\end{proposition}

We will be particularly interested in the special case where $v=\id$.
Hence, in the remainder of this paper, we will omit $v$ from our
notation in this case, so $S=S_v$, $I_w=I_{v,w}$,
$d_{A,B}=d^{(v)}_{A,B}$, and $I_{(p,q,r)}=I^{(v)}_{(p,q,r)}$.  Note
that, in this case, $\deg z_{p,q}=p-q$, and
$$\deg d_{A,B}=\sum_{p\in A} p-\sum_{q\in B} q.$$

\begin{warning}
In~\cite{WYGov}, the variables in matrices are indexed with $z_{p,q}$
being the variable in the $p$-th row counting from the bottom.  This
was done for partial compatibility with the conventions for matrix
Schubert varieties.  Since matrix Schubert varieties play only a
marginal role in this paper, we have abandoned that convention and
index our matrix variables in the usual way, with $z_{p,q}$ being the
entry in row $p$ (counting from the top) and column $q$ (counting from
the left).
\end{warning}

\begin{example} \label{ex:kleqns526314}
Let $v=215436$ and $w=526314$.  Then
\[
M_v=\begin{pmatrix}
0 & 1 & 0 & 0 & 0 & 0 \\
1 & 0 & 0 & 0 & 0 & 0\\
z_{3,1} & z_{3,2} & 0 & 0 & 1 & 0 \\
z_{4,1} & z_{4,2} & 0 & 1 & 0 & 0\\
z_{5,1} & z_{5,2} & 1 & 0 & 0 & 0\\
z_{6,1} & z_{6,2} & z_{6,3} & z_{6,4} & z_{6,5} & 1
\end{pmatrix}
\mbox{ and $D(w)=$}
\begin{picture}(120,60)
\put(7.5,-37.5){\framebox(90,90)}
\put(7.5,-22.5){\line(1,0){30}}
\put(22.5,-37.5){\line(0,1){15}}
\put(37.5,-37.5){\line(0,1){15}}
\put(22.5,-7.5){\framebox(15,30)}
\put(22.5,7.5){\line(1,0){15}}
\put(52.5,-7.5){\framebox(30,15)}
\put(67.5,-7.5){\line(0,1){15}}
\thicklines
\put(15,-15){\circle*{4}}
\put(15,-15){\line(1,0){82.5}}
\put(15,-15){\line(0,1){67.5}}
\put(30,30){\circle*{4}}
\put(30,30){\line(1,0){67.5}}
\put(30,30){\line(0,1){22.5}}
\put(45,-30){\circle*{4}}
\put(45,-30){\line(1,0){52.5}}
\put(45,-30){\line(0,1){82.5}}
\put(60,15){\circle*{4}}
\put(60,15){\line(1,0){37.5}}
\put(60,15){\line(0,1){37.5}}
\put(75,45){\circle*{4}}
\put(75,45){\line(1,0){22.5}}
\put(75,45){\line(0,1){7.5}}
\put(90,0){\circle*{4}}
\put(90,0){\line(1,0){7.5}}
\put(90,0){\line(0,1){52.5}}
\end{picture}.
\]

Therefore, $S_v=\mathbb{C}[z_{3,1},z_{3,2},z_{4,1},z_{4,2},z_{5,1},z_{5,2},z_{6,1},z_{6,2},z_{6,3},z_{6,4},z_{6,5}]$.  Also
\[
I_{v,w}=I^{(v)}_{(6,2,0)}+I^{(v)}_{(3,2,1)}+I^{(v)}_{(4,5,2)}.
\]

Now, $I^{(v)}_{(6,2,0)}=\langle z_{6,1},z_{6,2}\rangle$.  The generalized Pl\"ucker coordinate $d_{\{i,6\},\{1,2\}}\in I^{(v)}_{(6,2,0)}$ for all $i$, so
\[
I^{(v)}_{(6,2,0)}+I^{(v)}_{(3,2,1)}=\langle
z_{6,1},z_{6,2},z_{3,1}z_{4,2}-z_{3,2}z_{4,1},z_{3,1}z_{5,2}-z_{3,2}z_{5,1},z_{4,1}z_{5,2}-z_{4,2}z_{5,1}\rangle.
\]

Since $d_{\{4,5,6\},\{3,4,5\}}=-z_{6,5}$, and the only generalized
Pl\"ucker coordinates among the generators of $I^{(v)}_{(4,5,2)}$ that
are not either a multiple of $-z_{6,5}$ or in
$I^{(v)}_{(6,2,0)}+I^{(v)}_{(3,2,1)}$ are $d_{\{4,5,6\},\{1,3,4\}}$
and $d_{\{4,5,6\},\{2,3,4\}}$,
\begin{align*}
I_{v,w}=\langle z_{6,1}, z_{6,2}&, z_{3,1}z_{4,2}-z_{3,2}z_{4,1}, z_{3,1}z_{5,2}-z_{3,2}z_{5,1}, z_{4,1}z_{5,2}-z_{4,2}z_{5,1}, \\
& z_{6,5}, z_{4,1}z_{6,4}+z_{5,1}z_{6,3}(-z_{6,1}), z_{4,2}z_{6,4}+z_{5,2}z_{6,3}(-z_{6,2})\rangle.
\end{align*}

If we let $a=z_{3,1}$, $b=z_{3,2}$, $c=z_{4,1}$, $d=z_{4,2}$,
$e=-z_{6,3}$, $f=z_{5,1}$, $g=z_{5,2}$, and $h=z_{6,4}$, then
\[
\mathcal{N}_{v,w}\cong\Spec\mathbb{C}[a,b,c,d,e,f,g,h]/\langle
ad-bc, ag-bf, cg-df, ch-ef, dh-eg\rangle,
\]
which we can think of as the variety of rank 1 ``matrices''
\[
\begin{pmatrix}
a & b & \\
c & d & e \\
f & g & h
\end{pmatrix}
\]

This is not a local complete intersection, as it is a codimension 3
variety whose ideal requires 5 generators.\qed
\end{example}

\begin{example}
Let $w=819372564$, as in Example~\ref{ex:ess-819372564}, and let
$v=\id$.
Then
\[
I_w=I_{(2,2,1)}+I_{(4,4,2)}+I_{(4,6,3)}+I_{(6,7,3)}+I_{(9,2,0)}.
\]
A priori, $I_{(2,2,1)}$ is generated by the $\binom{8}{2}\binom{2}{2}$
generalized Pl\"ucker coordinates $d_{A,\{1,2\}}$ where $A$ is a
2-element subset of $\{2,\ldots,9\}$.  Also, $I_{(4,4,2)}$ is a priori
generated by $\binom{6}{3}\binom{4}{3}$ generalized Pl\"ucker
coordinates, $I_{(4,6,3)}$ is generated by $\binom{6}{4}\binom{6}{4}$,
$I_{(6,7,3)}$ by $\binom{7}{4}\binom{4}{4}$ (some of which are shared
with $I_{(4,6,3)}$), and $I_{(9,2,0)}=\langle d_{\{9\},\{1\}}=z_{9,1},
d_{\{9\},\{2\}}=z_{9,2}\rangle$.  However, in our proof that $X_w$ is
lci, we will see that only $\binom{n}{2}-\ell(w)=\#D(w)=16$ of these
generators are needed.\qed
\end{example}

\subsection{Pattern avoidance and generalizations}

As mentioned in the introduction we say that a permutation $v\in S_m$
{\bf embeds} in $w\in S_n$ (or $w$ {\bf contains} $v$) if there are
some $m$ entries of $w$, say at indices $i_1<\cdots<i_m$, in the
relative order given by $v$, meaning that $w(i_j)<w(i_k)$ if and only
if $v(j)<v(k)$.  If $v$ does not embed in $w$, then $w$ is said to
{\bf avoid} $v$. When we discuss a permutation as being embedded or
being avoided by another permutation we usually call it a ({\bf
classical}) {\bf pattern}. For example, the pattern $132$ has three
embeddings in the permutation $526413$, namely in the entries marked by boldface in
$5\mathbf{264}13$, $5\mathbf{26}41\mathbf{3}$ and
$5\mathbf{2}6\mathbf{4}1\mathbf{3}$. Notice in particular that the
indices of an embedding need not be adjacent in the permutation. One
of the original motivations of studying permutations as patterns is
their relation to sorting algorithms in computer science. Probably the
earliest example of such an application is the characterization by
Knuth~\cite{Knu} of stack-sortable permutations as the ones avoiding
$231$. Since then, many extensions of classical patterns have been
introduced; the ones relevant here are those of interval patterns and
(marked) mesh patterns.

Interval patterns were introduced by Yong and the second author
in~\cite{WYGov} and, we now recall their definition. First recall that
the {\bf Bruhat order} on the symmetric group is the reflexive
transitive closure of the partial order defined by declaring $u$ to be
less than or equal to $v$ if $v = u s_{ij}$ and $\ell(v)
> \ell(u)$. Here $s_{ij}$ is the transposition that switches the
(not necessarily adjacent) positions $i$ and $j$, and $\ell(v)$ is the number of inversions in
the permutation $v$, or equivalently, the length of any reduced
expression for $v$ as a product of simple reflections $s_{i(i+1)}$,
called the {\bf Coxeter length} of $v$. Now, if $[u,v]$ and
$[x,w]$ are intervals in the Bruhat orders on $S_m$ and $S_n$
respectively, we say that $[u,v]$ ({\bf interval}) {\bf pattern embeds
in} $[x,w]$ if there is a common embedding consisting of indices $i_1
< \cdots < i_m$ of $u$ in $x$ and $v$ in $w$, such that the entries of
$x$ and $w$ outside of these indices agree, and additionally, the
intervals $[u,v]$ and $[x,w]$ are isomorphic posets. The motivation
for these patterns is that they govern any ``reasonable'' local
property, as shown by Yong and the second author~\cite{WYGov}.  Since,
given $u,v,w$, and the indices of the embedding, the permutation $x$
is automatically determined, we can omit $x$ in the notation.  Hence
we will abuse terminology to say that $[u,v]$ embeds in $w$ or that
$w$ avoids $[u,v]$ as appropriate.

Interval patterns are a special case of mesh patterns, which we now define. A {\bf mesh pattern} is a pair $(v,R)$ where $v$ is a permutation (classical pattern) from $S_m$ and $R$ is a subset of the square $\dbrac{0,m} \times \dbrac{0,m}$. An embedding of $(v,R)$ in a permutation $w$ is first of all an embedding of $v$ in $w$ in the usual sense, meaning indices $i_1 < \cdots < i_m$ such that the relative order of $w(i_1), \dotsc, w(i_m)$ is given by $v$. Equivalently, we have order-preserving bijections $\alpha, \beta \colon \dbrac{1,m} \to \dbrac{1,n}$ such that
\[\{ (\alpha(i), \beta(j)) \mid (i,j) \in G(v) \} \subseteq G(w),\]
where for any permutation $u$, $G(u)$ is defined to be the graph
\[G(u)=\{(i,u(i)): i\in\dbrac{1,n}\}\]
of $u$.  In addition, to be an embedding of $(v,R)$, we further require the following:
$$\mbox{If }(i,j) \in R\mbox{ then } R_{ij} \cap G(u)=\emptyset.$$
Here $R_{ij}$ is defined as the rectangle $\dbrac{\alpha(i)+1, \alpha(i+1)-1} \times
\dbrac{\beta(j)+1, \beta(j+1)-1}$, where, as a convention, we set $\alpha(0) = 0
= \beta(0)$ and $\alpha(m+1) = n+1 = \beta(m+1)$.

As a simple example,
consider the mesh pattern $(12, \{ (1,0),(1,1),(1,2) \})$ which can be
depicted as follows:
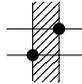
\[
\pattern{scale=1}{ 2 }{ 1/1, 2/2 }{1/0, 1/1, 1/2 }
\]
An occurrence of this mesh pattern in a permutation is a non-inversion (an occurrence of the classical pattern $12$) with the additional requirement that there is nothing in between the two elements in the occurrence. This is also called a {\bf vincular pattern} and denoted $\vinc{2}{1/1,2/2}{1}$. As one additional example we show how~\cite[Lemma 22]{Ulf} can be used to translate interval patterns into mesh patterns: The interval pattern $[14235, 45123]$ can be translated into the mesh pattern
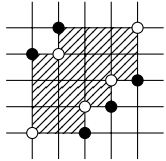
\[
\impattern{scale=1}{ 5 }{ 1/4, 2/5, 3/1, 4/2, 5/3 }
{ 1/1, 2/4, 3/2, 4/3, 5/5 }{ 4/3, 4/4, 3/2, 3/3, 3/4, 2/1, 2/2, 2/3, 2/4, 1/1, 1/2, 1/3 }
\]
(The white dots are not part of the mesh pattern; they only indicate the permutation $14235$ from the interval.) See~\cite[Subsec.~4.1]{Ulf} for more examples.

The definition of marked mesh patterns given by the first author~\cite[Def.~24]{Ulf} extends the definition of mesh patterns and allows another kind of designated regions where a certain number of elements is required to be present.
We review their definition only via an example:
\begin{example}
To show that the marked mesh pattern
$\patternsbmm{scale=1}{ 3 }{ 1/1, 2/3, 3/2}{1/0,1/1,1/2,1/3}{2/2/4/3/\scriptscriptstyle{1}}{3/2/4/3/\scriptscriptstyle{1}}$
occurs in the permutation $526413$, we first need to find an occurrence of the
underlying classical pattern $132$. There are three such occurrences, as shown
below.
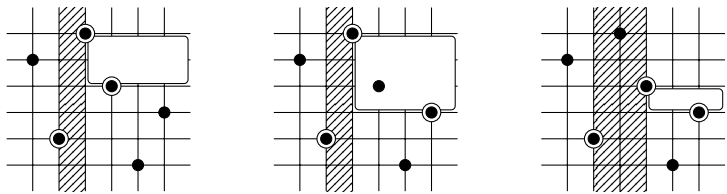
\[
\imopatternsbm{scale=1}{ 6 }{ 1/5, 2/2, 3/6, 4/4, 5/1, 6/3 }
{}{ 2/2, 3/6, 4/4 }{2/0,2/1,2/2,2/3,2/4,2/5,2/6}{3/4/7/6/{}}\qquad
\imopatternsbm{scale=1}{ 6 }{ 1/5, 2/2, 3/6, 4/4, 5/1, 6/3 }
{}{ 2/2, 3/6, 6/3 }{2/0,2/1,2/2,2/3,2/4,2/5,2/6}{3/3/7/6/{}}\qquad
\imopatternsbm{scale=1}{ 6 }{ 1/5, 2/2, 3/6, 4/4, 5/1, 6/3 }
{}{ 2/2, 4/4, 6/3 }{2/0,2/1,2/2,2/3,2/4,2/5,2/6, 3/0,3/1,3/2,3/3,3/4,3/5,3/6}{4/3/7/4/{}}
\]
However, only the middle occurrence of $132$ is an occurrence of the marked
mesh pattern since it is the only occurrence having \emph{at least} one dot in the
box marked with "$1$" in the pattern, as well as having no dots in the shaded vertical strip.\qed
\end{example}

\begin{warning}
Unfortunately, the customary conventions for writing a permutation in
a matrix (stemming ultimately from the conventions for matrix
multiplication) is upside down from the customary conventions for the
graph of a permutation (stemming ultimately from the conventions for
drawing the graph of a function).  Furthermore, the convention for
indexing entries in a matrix disagree with the cartesian convention
for indexing points on a graph.  Hence our conventions for drawing
mesh patterns are upside down from our conventions for writing
matrices and for drawing Rothe diagrams and essential sets.  Our
conventions for indexing regions in mesh patterns and regions in Rothe
diagrams and essential sets also disagree.
\end{warning}

\section{Rothe diagrams of lci permutations}
Prior to proving the sufficiency of our pattern avoidance conditions,
we need detailed information on the diagrams and essential sets
of permutations avoiding the given patterns.  This information may be
of independent combinatorial interest.  We begin by studying the
special case of Schubert varieties {\bf defined by inclusions}, which
were introduced in a different context by Gasharov and
Reiner~\cite{GasRei}.  In particular, the essential sets for
permutations indexing these Schubert varieties satisfy certain
combinatorial conditions.  We weaken these conditions to define what
it means for a permutation to be {\bf almost defined by inclusions}
and show that every permutation avoiding the given patterns is almost
defined by inclusions.  To each permutation almost defined by
inclusions, we will associate by modifying the diagram a permutation
honestly defined by inclusions.

\subsection{Permutations defined by inclusions}

Let $w\in S_n$ be a permutation.  We say that $w$ is {\bf defined by
  inclusions} ({\bf dbi}) if, for each box $(p,q)\in E(w)$,
$q-r_w(p,q)=\min\{p-1,q\}$.  To explain the terminology, note that
this condition on the essential set is equivalent to the statement
that the intersection conditions defining the Schubert variety are all
of the form $E_{p-1}\subset F_q$ or $F_q\subset E_{p-1}$.  Gasharov
and Reiner proved the following theorem.

\begin{theorem}{\cite[Thm. 4.2]{GasRei}}
\label{thm:dbidiagram}
The following are equivalent:
\begin{enumerate}
\item The Schubert variety $X_w$ is defined by inclusions.
\item For every box $(p,q)\in E(w)$, either
\begin{itemize}
\item[A:] there are no $1$'s in the permutation matrix $w$ weakly SW of $(p,q)$
  (In other words, there is no $k$ such that $k\leq q$ and
  $w(k)\geq p$.); or
\item[B:] there are no $1$'s in the permutation matrix $w$ strictly NE of $(p,q)$
  (In other words, there is no $k$ such that $k>q$ and $w(k)<p$.)
\end{itemize}
\item The permutation $w$ avoids $4231$, $35142$, $42513$, and $351624$.
\end{enumerate}
\end{theorem}

For an essential set box $(p,q)\in E(w)$ satisfying condition A,
$r_w(p,q)=0$.  This portion of the essential set and of the diagram
will not require further combinatorial attention.  We now focus on the
remainder of the diagram and essential set.  Let $D^\prime(w)\subset
D(w)$ be the subset consisting of diagram boxes $(x,y)$ where
$r_w(x,y)>0$, and let $E^\prime(w)\subset E(w)$ be the subset of essential
set boxes $(p,q)$ satisfying $r_w(p,q)>0$.  We begin with a lemma
relating the position $(p,q)$ to $r_w(p,q)$ for $(p,q)\in
E^\prime(w)$.

\begin{lemma}
\label{lem:dbilocation}
Let $w$ be a permutation defined by inclusions and
$(p,q)\in E^\prime(w)$.  Then $p\leq q$ and $r_w(p,q)=q-p+1$.
\end{lemma}

Visually, this says that the rank associated to the essential set box
$(p,q)$ is its Manhattan distance above the main diagonal plus 1.

\begin{proof}
Note that $r_w(p,q)$ is the number of $1$'s strictly SW of $(p,q)$ in
the permutation matrix $w$.  Since there is a $1$ in every row and
column of $w$ and in particular in every row below row $p$ and every
column to the left of column $q$,
$$(n-p)+(q-1)=s+r_w(p,q)+t+r_w(p,q),$$
where $s$ and $t$ are respectively the number of $1$'s strictly SE and
NW of $(p,q)$.  Since $(p,q)\in E^\prime(w)$, there are no $1$'s
strictly to the NE of $(p,q)$.  There is one $1$ in row $p$, one $1$
in column $q$, and the remainder are counted once in $s$, $t$, or
$r_w(p,q)$, so the total number of $1$'s is $n=s+t+r_w(p,q)+2$.
Therefore, $s+t+r_w(p,q)+2-p+q-1=s+t+2r_w(p,q)$, so $q-p+1=r_w(p,q)$,
as required.
\end{proof}

Now we state a lemma on the relative positions of boxes in $E^\prime(w)$.

\begin{lemma}
\label{lem:dbinwse}
Let $w$ be defined by inclusions, and let $(p,q)$ and
$(p^\prime,q^\prime)$ be two distinct elements of $E^\prime(w)$.  Then
either $p<p^\prime$ and $q<q^\prime$, or $p>p^\prime$ and
$q>q^\prime$.
\end{lemma}

In other words, any two boxes in $E^\prime(w)$ are strictly NW and SE
of each other.  Indeed, we can define a partition $\lambda$ (drawn in
the French manner) whose outer corners are the boxes of $E^\prime(w)$.
(Actually, $\lambda$ should also include any boxes $(p,q)\in
E(w)\setminus E^\prime(w)$ which nevertheless satisfy condition B.)
We can also define a partition $\mu$ whose diagram is $D(w)\setminus
D^\prime(w)$; its outer corners are the essential set boxes of rank 0.
The skew partition $\lambda/\mu$ is not exactly the right hull of the
permutation $w$ as studied by Sj\"ostrand~\cite{SjosRH}, but our
partition $\mu$ is the same as his, and our partition $\lambda$ is
slightly smaller but closely related.  It appears that further
combinatorial considerations on the diagrams of the permutations may
elucidate the connection between the results of Sj\"ostrand and of
Gasharov and Reiner.

\begin{proof}
If $p=p^\prime$, then we can assume without loss of generality that
$q>q^\prime$.  Since $(p,q)$ is in the diagram, $w(q)<p$.  Then
$w(q)<p^\prime$ and $q>q^\prime$, so $(p^\prime,q^\prime)$ does not
satisfy Condition B.

Otherwise, assume without loss of generality that $p<p^\prime$.  If
$q\geq q^\prime$, then since $(p,q)$ is in the diagram, $w^{-1}(p)>q$.
Therefore, $p<p^\prime$ and $w^{-1}(p)>q^\prime$, so
$(p^\prime,q^\prime)$ does not satisfy Condition B.
\end{proof}

We now describe a partition of $D^\prime(w)$ into rectangular regions,
one associated to each box of $E^\prime(w)$.  Let $k$ be the number of
elements of $E^\prime(w)$.  We order the boxes in $E^\prime(w)$ by
rank (with ties broken arbitrarily) and label them
$(p_1,q_1),\ldots,(p_k,q_k)$, so $r_w(p_1,q_1)\leq\cdots\leq
r_w(p_k,q_k)$.  Let $R_m\subset D^\prime(w)$ be the set of diagram boxes
which are weakly SW of $(p_m,q_m)$ but not weakly SW of
$(p_{m^\prime},q_{m^\prime})$ for any $m^\prime<m$.  For convenience, let
$r_m=r_w(p_m,q_m)$.

\begin{lemma}
\label{lem:dbipartition}
Each region $R_m$ is a rectangle consisting of boxes all from the
connected component of $D(w)$ containing $(p_m,q_m)$.
\end{lemma}

\begin{proof}
Suppose $(x,y)\in D(w)$ is weakly SW of $(p_m,q_m)$ but not in the
same component as $(p_m,q_m)$.  Since there is one or more crossed out
lines between $(x,y)$ and $(p_m,q_m)$ in the drawing of the diagram,
$r_w(x,y)<r_w(p_m,q_m)$.  If $r_w(x,y)=0$, then $(x,y)\not\in
D^\prime(w)$.  Otherwise, $(x,y)$ is SW of some essential set box
$(p_{m^\prime},q_{m^\prime})$ in its own connected component.  Since
$r_w(p_{m^\prime},q_{m^\prime})=r_w(x,y)<r_w(p_m,q_m)$, we must have
that $m^\prime<m$ by the requirements on our ordering of
$E^\prime(w)$.  Therefore, by our definition of the region $R_m$,
$(x,y)\not\in R_m$.

Because all connected components of $D(w)$ form the diagram of a
partition shape, and the essential set boxes in any connected
component are the outer corners of the partition, each $R_m$ must be a
rectangle.
\end{proof}

For each integer $m\in\dbrac{1,k}$, we define $\Wpred(m)$ as
the index of the first region (other than $R_m$) with a box directly W
of $(p_m,q_m)$ and $\Spred(m)$ as the index of the first region (other
than $R_m$) with a box directly S of $(p_m,q_m)$.  If no such region
exists, we accordingly let $\Wpred(m)=0$ or $\Spred(m)=0$.  By our
definition of the regions $R_m$, $\Wpred(m)<m$, and $\Spred(m)<m$.

We remark that, in the remainder of the paper, it is not absolutely
necessary that our essential set boxes be ordered strictly in
increasing rank.  Rather, any ordering for which
Lemma~\ref{lem:dbipartition} holds will suffice.  Explicitly, this
means that if $m^\prime<m$, then either $r_{m^\prime}\leq r_m$, or
there exists $m^{\prime\prime}<m^\prime$ where
$(p_{m^{\prime\prime}},q_{m^{\prime\prime}})$ is visually between
$(p_m,q_m)$ and $(p_{m^\prime},q_{m^\prime})$ in the NW--SE ordering
of $E^{\prime}(w)$ on the drawing of the diagram.

\begin{example}
\label{ex:ess-819732654}
The permutation $w=819732654$ is defined by inclusions.  The
diagram and essential set of $w$ are as in
Figure~\ref{fig:ess-819732654}.  In particular, $E(w)=\{(2,2), (4,6),
(9,2)\}$.  Furthermore, $E^{\prime}(w)=\{(2,2),(4,6)\}$,
$(p_1,q_1)=(2,2)$ with $r_1=1$, and $(p_2,q_2)=(4,6)$ with $r_2=3$.
In this case, both $R_1$ and $R_2$ are entire connected components
of $D^{\prime}(w)$, but it is possible for a connected component of
$D^\prime(w)$ to be partitioned into several regions.\qed

\begin{figure}[htbp]
\[\begin{picture}(200,140)
\put(-55,60){$D(819732654)=$}
\put(37.5,0){\framebox(135,135)}
\put(37.5,15){\line(1,0){30}}
\put(52.5,15){\line(0,-1){15}}
\put(67.5,15){\line(0,-1){15}}

\put(52.5,30){\line(0,1){90}}
\put(67.5,30){\line(0,1){90}}
\put(52.5,30){\line(1,0){15}}
\put(52.5,45){\line(1,0){15}}
\put(52.5,60){\line(1,0){15}}
\put(52.5,75){\line(1,0){15}}
\put(52.5,90){\line(1,0){15}}
\put(52.5,105){\line(1,0){15}}
\put(52.5,120){\line(1,0){15}}

\put(97.5,45){\line(0,1){45}}
\put(112.5,45){\line(0,1){45}}
\put(127.5,45){\line(0,1){45}}
\put(97.5,45){\line(1,0){30}}
\put(97.5,60){\line(1,0){30}}
\put(97.5,75){\line(1,0){30}}
\put(97.5,90){\line(1,0){30}}

\thicklines
\put(45,22.5){\circle*{4}}
\put(45,22.5){\line(1,0){127.5}}
\put(45,22.5){\line(0,1){112.5}}

\put(60,127.5){\circle*{4}}
\put(60,127.5){\line(1,0){112.5}}
\put(60,127.5){\line(0,1){7.5}}

\put(75,7.5){\circle*{4}}
\put(75,7.5){\line(1,0){97.5}}
\put(75,7.5){\line(0,1){127.5}}

\put(90,37.5){\circle*{4}}
\put(90,37.5){\line(1,0){82.5}}
\put(90,37.5){\line(0,1){97.5}}

\put(105,97.5){\circle*{4}}
\put(105,97.5){\line(1,0){67.5}}
\put(105,97.5){\line(0,1){37.5}}

\put(120,112.5){\circle*{4}}
\put(120,112.5){\line(1,0){52.5}}
\put(120,112.5){\line(0,1){22.5}}

\put(135,52.5){\circle*{4}}
\put(135,52.5){\line(1,0){37.5}}
\put(135,52.5){\line(0,1){82.5}}

\put(150,67.5){\circle*{4}}
\put(150,67.5){\line(1,0){22.5}}
\put(150,67.5){\line(0,1){67.5}}

\put(165,82.5){\circle*{4}}
\put(165,82.5){\line(1,0){7.5}}
\put(165,82.5){\line(0,1){52.5}}

\put(56,4.5){$E$}

\put(56,109.5){$E$}
\put(116,79.5){$E$}

\end{picture}\]
\caption{\label{fig:ess-819732654} Diagram and essential set for $w=819732654$.}
\end{figure}
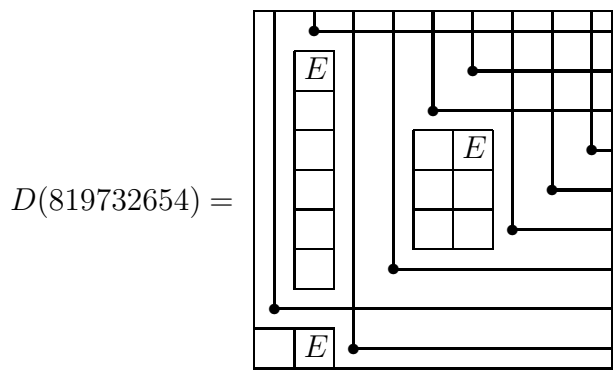
\end{example}

\subsection{Permutations almost defined by inclusions}

Now we define conditions on the diagram of a permutation $w$ which are
a weakening of the conditions of Gasharov and Reiner.  We say a
permutation $w$ is {\bf almost defined by inclusions} if, for all
$(p,q)\in E(w)$, either $(p,q)$ satisfies one of the Conditions A and
B defined by Gasharov and Reiner, or $(p,q)$ satisfies one of the
following Conditions W and X as well as one of the following conditions Y and
Z.  Note Conditions W and X are respectively the mirror images of
Conditions Y and Z under reflection across the main antidiagonal
(which in terms of permutations takes $w$ to $w_0w^{-1}w_0$).

\begin{itemize}
\item[W:] 
For all $p^\prime<p$, $(p^\prime,q)\not\in E(w)$ (and hence, for all
$p^\prime<p$, $(p^\prime,q)\not\in D(w)$).  Furthermore, either
$(p,q-1)\not\in D(w)$, or there exists $p^\prime<p$ such that
$(p^\prime,q-1)\in E(w)$, $(p^\prime,q-1)$ satisfies Condition B, and
$r_w(p^\prime,q-1)=r_w(p,q)$.  (The last part of the previous
condition is equivalent to $(p^\prime,q-1)$ and $(p,q)$ being in the same
connected component of $D(w)$.)

\item[X:]
There exists a unique $p^\prime<p$ such that $(p^\prime,q)\in E(w)$.
Furthermore, $(p^\prime,q)$ satisfies Condition B, and
$r_w(p^\prime,q)=r_w(p,q)+1$.  Finally, if $q^\prime$ is the smallest
integer such that $(p^\prime,b)\in D(w)$ for all
$b\in\dbrac{q^\prime,q}$, then $(p,q^\prime-1)\in D(w)$.  (Given the
first two conditions, the combinatorics of diagrams always implies
that $q^\prime>1$ and that $(p,b)\in D(w)$ for all $b\in\dbrac{q^\prime,q}$.)

\item[Y:]
For all $q^\prime>q$, $(i,q^\prime)\not\in E(w)$.  Furthermore, either
$(p+1,q)\not\in D(w)$, or there exists $q^\prime>q$ such that
$(p+1,q^\prime)\in E(w)$, $(p+1,q^\prime)$ satisfies Condition B, and
$r_w(p+1,q^\prime)=r_w(p,q)$.

\item[Z:]
There exists a unique $q^\prime>q$ such that $(p,q^\prime)\in E(w)$.
Furthermore, $(p,q^\prime)$ satisfies Condition B and
$r_w(p,q^\prime)=r_w(p,q)+1$.  Finally, if $p^\prime$ is the greatest
integer such that $(a,q^\prime)\in D(w)$ for all
$a\in\dbrac{p,p^\prime}$, then $(p^\prime+1,q)\in D(w)$.
\end{itemize}

Note Conditions W and X are mutually exclusive, as are Y and Z.  By
the {\bf type} of an essential set box failing conditions A and B we
mean the pair of conditions among W, X, Y, and Z it satisfies, so a
Type WZ essential set box is one that satisfies Conditions W and Z.

\begin{example}
Let $w=819372564$, as in Example~\ref{ex:ess-819372564} and
Figure~\ref{fig:ess-819372564}.  The essential set boxes at $(2,2)$,
$(4,6)$, and $(9,2)$ each satisfy either condition A or B.  The
essential set box at $(4,4)$ is of Type WZ, and the essential set box
at $(6,7)$ is of Type WY.\qed
\end{example}

It will turn out that the permutations that are almost defined by
inclusions are precisely the ones indexing lci Schubert varieties.
In the interest of keeping the logic of the proof clear, we will
maintain a distinction between the two notions until we have proved
their equivalence.

We now show that avoiding the six given patterns implies that a
permutation is almost defined by inclusions.

\begin{theorem}
\label{thm:adbipatterns}
If a permutation is not almost defined by inclusions, then it contains one of
the patterns $53241$, $52341$, $52431$, $35142$, $42513$, and $351624$.
\end{theorem}

Our proof follows Gasharov and Reiner's proof for
Theorem~\ref{thm:dbidiagram} with some additional complications
required in our case.

\begin{proof}
Suppose our permutation $w$ is not almost defined by inclusions.
Therefore, either there is some box $(p,q)\in E(w)$ that satisfies
none of Conditions A, B, W, and X, or there is some box $(p,q)\in E(w)$ that
satisfies none of Conditions A, B, Y, and Z.  We will give the details
of our proof only in the former case; the latter case can be proved in an
entirely identical fashion, except that we switch N and E and switch W and S
on diagrams, which corresponds to changing $w$ to $w_0w^{-1}w_0$.

Our proof strategy will be to split into a number of cases depending
on the features of $D(w)$ and the placement of the $1$'s relative to
$D(w)$.  In each case we will find one of the stated patterns in $w$.

Violation of A requires a $1$ in the permutation matrix SW of $(p,q)$;
we choose $c$ so that $(w(c),c)$ is SW of $(p,q)$.  Similarly, we
choose $c^\prime$ such that $(w(c^\prime),c^\prime)$ is NE of
$(p,q)$; the existence of $c^\prime$ is guaranteed by violation of B.

We now split into two cases depending on whether there is a diagram
box directly N of $(p,q)$.

\begin{enumerate}
\item For all $p^\prime<p$, the box $(p^\prime,q)\not\in D(w)$.

Note that $w(j)>w(c^\prime)$, as otherwise $(w(c^\prime),j)$ would be
in the diagram, contradicting our assumption for this case.  Then,
since $(p,q)$ violates Condition W, the box $(p,q-1)$ is in the
diagram, and it is not S of a box in the same component satisfying
Condition B.  We now have two cases depending on whether $(p-1,q-1)$
is in the diagram.

\begin{enumerate}
\item The box $(p-1,q-1)\not\in D(w)$.

In this case, $w^{-1}(p-1)<q$.  Let $B$ be the rectangular box bounded
by rows $p-1$ and $w(c)$ and columns $q$ and $c^\prime$, and let
$B^\prime$ be the box bounded by rows $w(c^\prime)$ and $p$ and
columns $c$ and $q$.

If $B$ does not contain a $1$ in its interior, then $w(q+1)>w(c)$ and
$w^{-1}(p)>c^\prime$.  Therefore, the $1$'s in columns $c,q,q+1,c^\prime$,
and $w^{-1}(p)$ form a $42513$ pattern.

If $B$ contains a $1$ in its interior, there are two cases.  If
$B^\prime$ contains a $1$ in its interior, then the $1$'s in $B$ and
$B^\prime$ along with the $1$'s in columns $c$, $q$, and $c^\prime$
produce a $52341$ or $53241$ pattern, depending on whether the $1$ in
$B^\prime$ is above or below row $w(q)$.  If $B^\prime$ does not
contain a $1$ in its interior, then $w^{-1}(p-1)<c$ and
$w(q-1)<w(c^\prime)$.  Therefore, the $1$'s in columns
$w^{-1}(p-1),c,q-1$ and $c^\prime$ along with the $1$ in $B$ produce a
$35142$ pattern (where the $1$ in $B$ represents the $4$).

\item The box $(p-1,q-1)\in D(w)$.

In this case, $w(q)=p-1$.  Furthermore, there exists a unique
$p^\prime<p$ such that $(p^\prime,q-1)\in E(w)$ and
$r_w(p^\prime,q-1)=r_w(p,q)$.  (This is the essential set box in
column $q-1$ in the same connected component of $D(w)$.)  Since
$(p,q)$ does not satisfy Condition W, $(p^\prime,q-1)$ does not
satisfy condition B.  We change our choice of $c^\prime$ if necessary
so that $(w(c^\prime),c^\prime)$ is NE of both $(p,q)$ and
$(p^\prime,q-1)$.  Now let $B$ be the rectangular box bounded by rows
$p-1$ and $w(c)$ and columns $q$ and $c^\prime$, as in Case 1a, and
let $B^\prime$ be the box bounded by rows $w(c^\prime)$ and $p^\prime$
and columns $c$ and $q$.

The argument is now the same as in Case 1a, with two minor
differences.  In the case both $B$ and $B^\prime$ contain a $1$, the
pattern $52341$ is the only one that can be produced.  In the case
where $B$ contains a $1$ but $B^\prime$ does not, $p^\prime$ must be
used instead of $p$.
\end{enumerate}
\item There exists $p^\prime<p$ with $(p^\prime,q)\in D(w)$.

This implies that there exists $p^\prime<p$ with $(p^\prime,q)\in
E(w)$.  We break into the cases where there exists such a
$(p^\prime,q)\in E(w)$ violating Condition B and where every
$(p^\prime,q)\in E(w)$ with $p^\prime<p$ satisfies Condition B (which
implies that there is only one $(p^\prime,q)\in E(w)$ with
$p^\prime<p$).

\begin{enumerate}
\item There exists $p^\prime<p$ such that $(p^\prime,q)\in E(w)$ and $(p^\prime,q)$ violates Condition B.

We may assume, changing our choice of $c^\prime$ if necessary, that
$(w(c^\prime),c^\prime)$ is NE of $(p^\prime,q)$.  Now let $B$ be the
rectangular box bounded by rows $p-1$ and $w(c)$ and columns $q$ and
$c^\prime$, as in Case 1a, and let $B^\prime$ the box bounded by rows
$w(c^\prime)$ and $p$ and columns $c$ and $q+1$ (and not $q$ as
in Case 1a).

If $B$ does not contain a $1$ in its interior, then we have a $42513$
pattern or a $351624$ pattern depending on whether or not $B^\prime$
contains a $1$ in its interior.  If $B$ contains a $1$ in its
interior, then if $B^\prime$ contains {\it two} $1$'s in its
interior, we have a $53241$ or a $52341$ pattern depending on the
arrangement of these two $1$'s.  If $B$ contains a $1$ in its interior
and $B^\prime$ contains fewer than two $1$'s in its interior, either
$w^{-1}(p-1)<c$, or $(p-1,w^{-1}(p-1))$ is the only $1$ in the interior
of $B$.  If $w^{-1}(p-1)<c$, then either $w^{-1}(p^\prime)$ is E of
the $1$ in $B$, in which case $w^{-1}(p-1)$, $c$, $q$, the $1$ in $B$,
and $w^{-1}(p^\prime)$ form a $35142$ pattern, or $w^{-1}(p^\prime)$
is W of the $1$ in $B$.  In this latter case, $w(q+1)>c$, so
$w^{-1}(p-1)$, $c$, $q$, $q+1$, $w^{-1}(p^\prime)$, and the $1$ in $B$
form a $351624$ pattern.  Otherwise, when $(p-1,w^{-1}(p-1))$ is the
only $1$ in the interior of $B$, $w^{-1}(p^\prime-1)<c$ and
$w(q)<w(c^\prime)$, so $w^{-1}(p^\prime-1)$, $c$, $q$, the $1$ in
$B$, and $c^\prime$ form a $35142$ pattern.

\item For the only $p^\prime<p$ such that $(p^\prime,q)\in E(w)$, the essential set box $(p^\prime,q)$ satisfies Condition B.

Since $(p,q)$ violates Condition X, either
$r_w(p^\prime,q)<r_w(p,q)-1$, or, letting $q^\prime$ denote the
smallest integer such that $(p^\prime,b)\in D(w)$ for all $b$ with
$j^\prime\leq b\leq j$, the box $(p,q^\prime-1)\not\in D(w)$.  We
further split into cases.

\begin{enumerate}
\item The rank conditions satisfy $r_w(p^\prime,q)<r_w(p,q)-1$.

In this case, neither $(p-1,q)$ nor $(p-2,q)$ is in $D(w)$.  Let $B$
and $B^\prime$ be as in Case 2a.  If $B$ does not have a $1$ in the
interior, we have a $42513$ or $351624$ pattern as in Case 2a.  Also
as in Case 2a, if $B^\prime$ has two $1$'s in its interior, then we
have a $53241$ or $52341$ pattern.  Otherwise, either $w^{-1}(p-1)<c$
or both $w^{-1}(p-2)<c$ and $w(q)<w(c^\prime)$, and the remainder of
the proof in these cases follows the remaining parts of Case 2a.

\item The box $(p,q^\prime-1)\not\in D(w)$.

We let $c=j^\prime-1$; note that, by our assumptions for this case,
$(w(c),c)$ is SW of $(p,q)$, as before.  In particular, our choice of
$c$ now forces $w^{-1}(p-1)<c$.  Let $B$ be the rectangular box
bounded by rows $p-1$ and $w(c)$ and columns $q$ and $c^\prime$ as
before.  We now have a $35142$ or $351624$ pattern depending on
whether $B$ contains a $1$ in its interior or not. \qedhere
\end{enumerate}
\end{enumerate}
\end{enumerate}
\end{proof}

We now canonically associate a permutation defined by inclusions to
every permutation which is almost defined by inclusions.  For a
permutation $w$ almost defined by inclusions, let
$E^{\prime\prime}(w)$ denote the subset of the essential set which
satisfies neither Condition A nor Condition B.  As before, let
$E^\prime(w)$ be the subset of the essential set which satisfies
Condition B but not Condition A.

\begin{theorem}
\label{thm:dbifromadbi}
Let $w$ be a permutation almost defined by inclusions.  Then there
exists a permutation $v$ such that
\begin{enumerate}
\item The essential set $E(v)=E(w)\setminus E^{\prime\prime}(w)$, and
\item The ranks $r_v(p,q)=r_w(p,q)$ for all $(p,q)\in E(v)$.
\end{enumerate}
These conditions define a unique permutation $v$ which is defined by
inclusions.  Furthermore, $\ell(v)-\ell(w)$ is the number of boxes in
$E^{\prime\prime}(w)$.
\end{theorem}

It would be interesting to know if the existence of such a $v$ in some
way gives an alternative characterization, independent of conditions
on the diagram, of being almost defined by inclusions.  We discuss
this further as Question~\ref{prob:dbifromadbiconv} in Section 7.

\begin{proof}
Formally, we will structure the proof as induction on the number of
boxes in $E^{\prime\prime}(w)$.  Informally, one should think of this
proof as giving a way to construct $v$ by eliminating boxes of
$E^{\prime\prime}(w)$ from the essential set one at a time,
increasing the length of the permutation by $1$ at each step.  It turns
out that the boxes of $E^{\prime\prime}(w)$ can be removed from the
essential set in any order.

The base case is where $E^{\prime\prime}(w)$ is empty.  In this case,
$w$ is defined by inclusions, so $v=w$ is such a permutation, and
$\ell(v)-\ell(w)=0$, which is the number of boxes in
$E^{\prime\prime}(w)$.

Let $(p,q)$ be a box of $E^{\prime\prime}(w)$.  We now divide the
proof into cases depending on the type of $(p,q)$.

Suppose $(p,q)$ is of type WY.  Let $w^\prime=wt$, where $t$ is the
transposition switching $q$ and $w^{-1}(p)$.  Since there are no
essential set boxes directly N of $(p,q)$, $w^{-1}(a)<q$ for all
$a\in\dbrac{w(q),p}$.  Since there are no essential set boxes directly
E of $(p,q)$, $w(b)>i$ for all $b \in\dbrac{j,w^{-1}(i)}$ with $j<b<w^{-1}(i)$.  This implies
that $\ell(w^\prime)=\ell(w)+1$ and
$D(w^\prime)=D(w)\setminus\{(p,q)\}$.  Conditions W and Y ensure that
$E(w^\prime)=E(w)\setminus\{(p,q)\}$ since any diagram box in
$(p,q-1)$ or $(p+1,q)$ is respectively S or W of an essential set box
of the same rank.  Since there are no boxes of $D(w)$ other than
$(p,q)$ in the rectangle bounded by rows $w(q)$ and $p$ and columns
$q$ and $w^{-1}(p)$, every box of $E(w^\prime)$ satisfies exactly the
same conditions that it satisfies as an element of $E(w)$.  Therefore,
$w^\prime$ is almost defined by inclusions, and
$E^\prime(w^\prime)=E^\prime(w)$.  Since $E^{\prime\prime}(w^\prime)$
has one fewer box than $E^{\prime\prime}(w)$, by the inductive
hypothesis, the theorem holds for $w^\prime$, so it holds for $w$.

Now suppose $(p,q)$ is of type WZ.  Let $p^\prime$ and $q^\prime$ be
as in Condition Z.  It follows that $w(q+1)=p^\prime+1$.  Let
$w^\prime=wt$, where $t$ switches $q$ and $q+1$.  Since $w(q)<p\leq
p^\prime+1=w(q+1)$, $\ell(w^\prime)=\ell(w)+1$.  Furthermore,
\[
D(w^\prime)=D(w)\cup\{(p,q+1),\ldots,(p^\prime,q+1)\}\setminus\{(p,q),\ldots,(p^\prime+1,q)\}.
\]
Since $r_w(p,q^\prime)=r_w(p,q)+1$ by Condition Z, $(p,q+2)\in D(w)$,
so $(p,q+2)\in D(w^\prime)$.  Also, by Condition W, if $(p,q-1)\in
D(w)$, then $(p-1,q-1)\in D(w)$, so if $(p,q-1)\in D(w^\prime)$, then
$(p-1,q-1)\in D(w^\prime)$.  Therefore,
$E(w^\prime)=E(w)\setminus\{(p,q)\}$.  Every box of $E(w^\prime)$
satisfies exactly the same conditions that it satisfies as an element
of $E(w)$, so $w^\prime$ is almost defined by inclusions and
$E^\prime(w^\prime)=E^\prime(w)$.  By the inductive hypothesis, the
theorem holds for $w^\prime$, so it holds for $w$.

Now suppose $(p,q)$ is of type XY.  Let $q^\prime$ be as in Condition
X; then $w(q^\prime-1)=p-1$.  Let $w^\prime=wt$ where $t$ switches
$q^\prime+1$ and $w^{-1}(p)$.  The argument in this case is entirely
analogous to the one in the case where $(p,q)$ is of type WZ.

Finally, suppose that $(p,q)$ is of type XZ.  Let $p_X$ and $q_X$ be
the $p^\prime$ and $q^\prime$ of Condition X, and $p_Z$ and $q_Z$ the
$p^\prime$ and $q^\prime$ of Condition Z.  Note that $w(q+1)=p_Z+1$
and $w(q_X-1)=p-1$.  Let $w^\prime=wt$ where $t$ switches $q_X-1$ and
$q+1$.  Since the interior of the rectangle bounded by columns $q_X-1$
and $q+1$ and rows $p-1$ and $p_Z+1$ consists entirely of boxes in
$D(w)$, there is no $a\in\dbrac{q_X-1,q+1}$ with
$w(a)\in\dbrac{p-1,p_Z+1}$, so $\ell(w^\prime)=\ell(w)+1$.
Furthermore,
\begin{align*}
D(w^\prime) = &D(w)\cup\{(p-1,q_X),\ldots,(p-1,q)\} \cup \{(p,q+1),\ldots,(p_Z,q+1)\} \\
&\setminus\{(p_Z+1,q_X-1),\ldots,(p_Z+1,q)\}\setminus\{(p,q_X-1),\ldots,(p_Z+1,q_X-1)\}.
\end{align*}
By Conditions Z and X respectively, $(p,q+2)$ and $(p-2,q)$ are both
in $D(w^\prime)$, so $E(w^\prime)=E(w)\setminus\{(p,q)\}$.  Every box
of $E(w^\prime)$ satisfies exactly the same conditions that it
satisfies as an element of $E(w)$, so $w^\prime$ is almost defined by
inclusions, and $E^\prime(w^\prime)=E^\prime(w)$.  By the inductive
hypothesis, the theorem holds for $w^\prime$, so it holds for $w$.
\end{proof}

\begin{example}
If $w=819372564$, as in Example~\ref{ex:ess-819372564}, then the
permutation $v$ which is associated to $w$ is the permutation
$v=819732654$ given in Example~\ref{ex:ess-819732654}.\qed
\end{example}

\section{Sufficiency} \label{sect:dbisufficiency}
We now proceed to prove that, if $w$ avoids the six stated patterns,
then $X_w$ is lci.  First note that the non-lci locus of $X_w$ is
closed and invariant under the action of the Borel group $B$, so the
non-lci locus must be a union of Schubert subvarieties of $X_w$.  In
particular, the non-lci locus of $X_w$ must contain the point
$e_{\id}$.  In the case where $w$ avoids the six stated patterns, we
show that $X_w$ is lci by showing that it is lci at $e_{\id}$.  By
Proposition~\ref{prop:alleqns}, we only need to show that the ideal
$I_w$ (defined in the paragraph before the proposition) is generated
by $\binom{n}{2}-\ell(w)$ polynomials.

Note that $\binom{n}{2}-\ell(w)$ is precisely the number of boxes in
the diagram $D(w)$.  Therefore, given a permutation $w$ which is
almost defined by inclusions, we will define one polynomial in $I_w$
for each box of $D(w)$.  Letting $J_w$ be the ideal generated by these
polynomials, we explicitly show that every other generator of $I_w$ is
in the ideal $J_w$ and hence that $I_w=J_w$ is a complete
intersection.

We begin first with the case where $w$ is defined by inclusions.
Following that case, we treat the general case by showing that, if $w$
is almost defined by inclusions, and $v$ is the defined by inclusions
permutation associated to $w$ by Theorem~\ref{thm:dbifromadbi}, then
the ideal $I_w$ is generated by $I_v$ plus one polynomial for each box
of $E^{\prime\prime}(w)$.

\subsection{The defined by inclusions case}
\label{sec:dbisufficiency}

Let $w$ be a permutation defined by inclusions.  Fix a total ordering
of the essential set $E^\prime(w)$ in which smaller rank boxes come before
larger rank boxes as in the discussion prior to
Lemma~\ref{lem:dbipartition}.  Let $k$ be the number of boxes in
$E^\prime(w)$, $(p_1,q_1),\ldots,(p_k,q_k)$ be the boxes of the essential set
in order, $r_m=r_w(p_m,q_m)$ for each $m\in\dbrac{1,k}$, and
$R_1,\ldots,R_k$ the rectangular regions of $D^\prime(w)$ defined prior to
Lemma~\ref{lem:dbipartition}.

For each box $(x,y)\in D(w)$, we define a polynomial $f_{(x,y)}$ in
$S$ (which will be a generalized Pl\"ucker coordinate) as follows.  If $r_w(x,y)=0$, then let
$$A(x,y)=\{x\}\mbox{ and } B(x,y)=\{y\}.$$
Otherwise, the box $(x,y)$ is in some rectangle $R_m$.  Let
\[
A(x,y)=\dbrac{p_m,p_m+r_m-1}\cup\{x+r_m\},
\]
and let
\[
B(x,y)=\{y-r_m\}\cup\dbrac{q_m-r_m+1,q_m}.
\]
Now let $f_{(x,y)}=d_{A(x,y),B(x,y)}$, and $J_w=\langle f_{(x,y)}\rangle_{(x,y)\in D(w)}$.

\begin{example}
Let $w=819732654$ as in Example~\ref{ex:ess-819732654} and
Figure~\ref{fig:ess-819732654}.  Then
\begin{align*}
J_w=\langle d&_{\{2,3\},\{1,2\}}, d_{\{2,4\},\{1,2\}},
d_{\{2,5\},\{1,2\}}, d_{\{2,6\},\{1,2\}}, d_{\{2,7\},\{1,2\}},
d_{\{4,5,6,7\},\{3,4,5,6\}}, \\ & d_{\{4,5,6,8\},\{3,4,5,6\}},
d_{\{4,5,6,9\},\{3,4,5,6\}}, d_{\{4,5,6,7\},\{2,4,5,6\}},
d_{\{4,5,6,8\},\{2,4,5,6\}}, \\ & d_{\{4,5,6,9\},\{2,4,5,6\}},
d_{\{9\},\{1\}}, d_{\{9\},\{2\}}\rangle. \tag*{\qed}
\end{align*}
\end{example}

Our aim is to prove the following theorem.

\begin{theorem}
\label{thm:dbilci}
Suppose $w$ is defined by inclusions.  Then the ideals $I_w$ and $J_w$
are equal.  Hence $I_w$ defines a local complete intersection.
\end{theorem}

We will first make a reduction showing $I_w$ is generated by a subset
of the original stated generators.  Recall that, for
$p,q,r\in\dbrac{1,n}$, $I_{(p,q,r)}$ is defined as the ideal
generated by the generalized Pl\"ucker coordinates $d_{A,B}$ for all
$A\subseteq\dbrac{p,n}$ and $B\subseteq\dbrac{1,q}$ where both
$A$ and $B$ have $r+1$ elements.  Furthermore, $I_w$ is the ideal
\[
I_w=\sum_{(p,q)\in E(w)} I_{(p,q,r_w(p,q))}.
\]
Now define
$I^\prime_{(p,q,r)}$ to be the ideal generated by
$d_{\dbrac{p,p+r-1}\cup\{x+r\},\{y-r\}\cup\dbrac{q-r+1, q}}$ for all $x$ with
$x\in\dbrac{p, n-r}$ and all $y\in\dbrac{1+r,q}$.  Let
\[
I^\prime_w=\sum_{(p,q)\in E(w)} I^\prime_{(p,q,r_w(p,q))}.
\]

\begin{lemma}
\label{lem:dbinejustify}
Suppose $w$ is defined by inclusions.  Then
$I^\prime_{(p,q,r_w(p,q))}=I_{(p,q,r_w(p,q))}$ for all $(p,q)\in E(w)$, and hence
$I^\prime_w=I_w$.
\end{lemma}

\begin{proof}
For all $x\in\dbrac{p, n-r_w(p,q)}$ and all $y\in\dbrac{1+r_w(p,q),q}$, we have
\[
\dbrac{p,p+r_w(p,q)-1}\cup\{x+r_w(p,q)\}\subseteq \dbrac{p,n}
\]
and
\[
\{y-r_w(p,q)\}\cup\dbrac{q-r_w(p,q)+1,q}\subseteq \dbrac{1,q}.
\]
Furthermore, the size of both these sets is $r_w(p,q)+1$.  Therefore,
\[
I^\prime_{(p,q,r_w(p,q))}\subseteq I_{(p,q,r_w(p,q))}.
\]

Fix $(p,q)\in E(w)$, and fix $r=r_w(p,q)$.  If $r=0$, then the
two generating sets are the same, so we need to prove the reverse
direction only in the case where $(p,q)\in E^\prime(w)$.  Given $(p,q)\in
E^\prime(w)$, we need to show that $d_{A,B}\in I^\prime_{(p,q,r)}$ for
all $A$ and $B$ satisfying the conditions that
$A\subseteq\dbrac{p,n}$, $B\subseteq\dbrac{1,q}$, and both $A$
and $B$ have $r+1$ elements.  We do so by induction on the
degree of $d_{A,B}$.  Recall that our polynomial ring $S$ is graded so
that $\deg d_{A,B}=\sum_{p\in A} p-\sum_{q\in B} q$.

Fix $A$ and $B$ such that $d_{A,B}$ is one of the defined generators
of $I_{(p,q,r)}$.  If $\dbrac{p,p+r-1}\subseteq A$ and
$\dbrac{q-r+1,q}\subseteq B$, then $d_{A,B}\in
I^\prime_{(p,q,r)}$ by definition.  Otherwise, either the difference
$\dbrac{p,p+r-1}\setminus A$ or
$\dbrac{q-r+1,q}\setminus B$ is nonempty; we treat the case
where $\dbrac{q-r+1,q}\setminus B$ is nonempty and leave the
entirely analogous argument in the other case to the reader.  Let
$b\in\dbrac{q-r+1,q}\setminus B$.

Our proof is by expanding the determinant $d_{A\cup\{b\},B\cup\{b\}}$
in two different ways.  (If $b\in A$, we mean to consider the
determinant of the matrix where row $b$ occurs twice.  This
determinant is identically $0$, but we can still consider its
expansions formally.)  Let $b^\prime$ be the smallest element of $B$,
and note that $b^\prime<q-r+1$.  Now consider the Laplace expansion of
$d_{A\cup\{b\},B\cup\{b\}}$ using column $b^\prime$, which is given by
\[
d_{A\cup\{b\},B\cup\{b\}}=\sum_{a\in A\cup\{b\}} \pm z_{a,b^\prime}
d_{A\cup\{b\}\setminus\{a\},B\cup\{b\}\setminus\{b^\prime\}}.
\]

By Lemma~\ref{lem:dbilocation}, $p=q-r+1$.  Therefore,
$p>b^\prime$, and $a>b^\prime$ for any $a\in A\cup\{b\}$.  This
implies
\begin{align*}
\deg d_{A\cup\{b\}\setminus\{a\},B\cup\{b\}\setminus\{b^\prime\}} =\deg
d_{A,B}-a+b^\prime < \deg d_{A,B}.
\end{align*}
Furthermore, $p\leq b$, so $A\cup\{b\}\setminus\{a\}\subseteq \dbrac{p,n}$.
Now, by the inductive hypothesis,
\[
d_{A\cup\{b\}\setminus\{a\},B\cup\{b\}\setminus\{b^\prime\}}\in
I^\prime_{(p,q,r)}
\]
for all $a\in A\cup\{b\}$ since they are all of
smaller degree.  Therefore, $d_{A\cup\{b\},B\cup\{b\}} \in I^\prime_{(p,q,r)}$.

Now we expand $d_{A\cup\{b\},B\cup\{b\}}$ along row $b$.  This
expansion is given by
\[
d_{A\cup\{b\},B\cup\{b\}}=\pm d_{A,B}+\sum_{b^\prime\in B}
z_{b,b^\prime}d_{A,B\cup\{b\}\setminus\{b^\prime\}}.
\]
If $b^\prime>b$, then $z_{b,b^\prime}=0$.  If $b^\prime<b$, then
\begin{align*}
\deg d_{A,B\cup\{b\}\setminus\{b^\prime\}} = \deg d_{A,B}-b+b^\prime < \deg d_{A,B},
\end{align*}
so, by the inductive hypothesis,
$d_{A,B\cup\{b\}\setminus\{b^\prime\}} \in I^\prime_{(p,q,r)}$.
Therefore,
$d_{A,B}\in I^\prime_{(p,q,r)}$,
as desired.
\end{proof}

\begin{example}
We illustrate our proof for the case where $p=4$, $q=6$, and $r=3$.
(This case comes up for $w=819732654$ as in
Example~\ref{ex:ess-819732654}.)  Here $I_{(p,q,r)}$ is the ideal of
all size $4$ minors of the matrix
\[
\begin{pmatrix}
z_{4,1} & z_{4,2} & z_{4,3} & 1 & 0 & 0 \\
z_{5,1} & z_{5,2} & z_{5,3} & z_{5,4} & 1 & 0 \\
z_{6,1} & z_{6,2} & z_{6,3} & z_{6,4} & z_{6,5} & 1 \\
z_{7,1} & z_{7,2} & z_{7,3} & z_{7,4} & z_{7,5} & z_{7,6} \\
z_{8,1} & z_{8,2} & z_{8,3} & z_{8,4} & z_{8,5} & z_{8,6} \\
z_{9,1} & z_{9,2} & z_{9,3} & z_{9,4} & z_{9,5} & z_{9,6}\end{pmatrix},
\]
while $I^\prime_{(p,q,r)}$ is the ideal generated by the size 4
minors which use the 3 topmost rows (plus some other row) and the 3 rightmost
columns (plus some other column).

Consider 
\[
d_{\{4,6,7,9\},\{1,2,4,6\}}=\begin{vmatrix}
z_{4,1} & z_{4,2} & 1 & 0 \\
z_{6,1} & z_{6,2} & z_{6,4} & 1 \\
z_{7,1} & z_{7,2} & z_{7,4} & z_{7,6} \\
z_{9,1} & z_{9,2} & z_{9,4} & z_{9,6}\end{vmatrix},
\]
which has degree $13$ in our grading.  Our proof writes
$d_{\{4,6,7,9\},\{1,2,4,6\}}$ in terms of generalized Pl\"ucker
coordinates of smaller degree in $I_{(p,q,r)}$, which by induction are
in $I^\prime_{(p,q,r)}$, as follows.

In this case, $b=5$.  Hence we consider
\[
d_{\{4,5,6,7,9\},\{1,2,4,5,6\}}=\begin{vmatrix}
z_{4,1} & z_{4,2} & 1 & 0 & 0 \\
z_{5,1} & z_{5,2} & z_{5,4} & 1 & 0 \\
z_{6,1} & z_{6,2} & z_{6,4} & z_{6,5} & 1 \\
z_{7,1} & z_{7,2} & z_{7,4} & z_{7,5} & z_{7,6} \\
z_{9,1} & z_{9,2} & z_{9,4} & z_{9,5} & z_{9,6}\end{vmatrix}.
\]
Since $b^\prime=1$, we consider the expansion
\begin{align*}
d&_{\{4,5,6,7,9\},\{1,2,4,5,6\}}=
z_{4,1}\begin{vmatrix}
z_{5,2} & z_{5,4} & 1 & 0 \\
z_{6,2} & z_{6,4} & z_{6,5} & 1 \\
z_{7,2} & z_{7,4} & z_{7,5} & z_{7,6} \\
z_{9,2} & z_{9,4} & z_{9,5} & z_{9,6}\end{vmatrix}
-z_{5,1}\begin{vmatrix}
z_{4,2} & 1 & 0 & 0 \\
z_{6,2} & z_{6,4} & z_{6,5} & 1 \\
z_{7,2} & z_{7,4} & z_{7,5} & z_{7,6} \\
z_{9,2} & z_{9,4} & z_{9,5} & z_{9,6}\end{vmatrix} \\
&+z_{6,1}\begin{vmatrix}
z_{4,2} & 1 & 0 & 0 \\
z_{5,2} & z_{5,4} & 1 & 0 \\
z_{7,2} & z_{7,4} & z_{7,5} & z_{7,6} \\
z_{9,2} & z_{9,4} & z_{9,5} & z_{9,6}\end{vmatrix}
-z_{7,1}\begin{vmatrix}
z_{4,2} & 1 & 0 & 0 \\
z_{5,2} & z_{5,4} & 1 & 0 \\
z_{6,2} & z_{6,4} & z_{6,5} & 1 \\
z_{9,2} & z_{9,4} & z_{9,5} & z_{9,6}\end{vmatrix}
+z_{9,1}\begin{vmatrix}
z_{4,2} & 1 & 0 & 0 \\
z_{5,2} & z_{5,4} & 1 & 0 \\
z_{6,2} & z_{6,4} & z_{6,5} & 1 \\
z_{7,2} & z_{7,4} & z_{7,5} & z_{7,6}\end{vmatrix}.
\end{align*}
The determinants in the expansion have degrees 10, 9, 8, 7, and 5
respectively, so by induction $d_{\{4,5,6,7,9\},\{1,2,4,5,6\}}\in
I^\prime_{(p,q,r)}$.

Now we expand $d_{\{4,5,6,7,9\},\{1,2,4,5,6\}}$ along row $b=5$.  We get
\begin{align*}
d&_{\{4,5,6,7,9\},\{1,2,4,5,6\}}=
z_{5,1}\begin{vmatrix}
z_{4,2} & 1 & 0 & 0 \\
z_{6,2} & z_{6,4} & z_{6,5} & 1 \\
z_{7,2} & z_{7,4} & z_{7,5} & z_{7,6} \\
z_{9,2} & z_{9,4} & z_{9,5} & z_{9,6}\end{vmatrix}
-z_{5,2}\begin{vmatrix}
z_{4,1} & 1 & 0 & 0 \\
z_{6,1} & z_{6,4} & z_{6,5} & 1 \\
z_{7,1} & z_{7,4} & z_{7,5} & z_{7,6} \\
z_{9,1} & z_{9,4} & z_{9,5} & z_{9,6}\end{vmatrix} \\
& +z_{5,4}\begin{vmatrix}
z_{4,1} & z_{4,2} & 0 & 0 \\
z_{6,1} & z_{6,2} & z_{6,5} & 1 \\
z_{7,1} & z_{7,2} & z_{7,5} & z_{7,6} \\
z_{9,1} & z_{9,2} & z_{9,5} & z_{9,6}\end{vmatrix}
-1\begin{vmatrix}
z_{4,1} & z_{4,2} & 1 & 0 \\
z_{6,1} & z_{6,2} & z_{6,4} & 1 \\
z_{7,1} & z_{7,2} & z_{7,4} & z_{7,6} \\
z_{9,1} & z_{9,2} & z_{9,4} & z_{9,6}\end{vmatrix}
+0\begin{vmatrix}
z_{4,1} & z_{4,2} & 1 & 0 \\
z_{6,1} & z_{6,2} & z_{6,4} & z_{6,5} \\
z_{7,1} & z_{7,2} & z_{7,4} & z_{7,5} \\
z_{9,1} & z_{9,2} & z_{9,4} & z_{9,5}\end{vmatrix}.\end{align*}
The last term is 0, the next to last term is
$d_{\{4,6,7,9\},\{1,2,4,6\}}$, and the determinants in the first three
terms have degrees 9, 10, and 12 respectively.  Since the first three
terms have determinants of degree less than 13, they are in
$I^\prime_{(p,q,r)}$.  As $d_{\{4,5,6,7,9\},\{1,2,4,5,6\}}\in
I^\prime_{(p,q,r)}$,
\begin{equation}
d_{\{4,6,7,9\},\{1,2,4,6\}} \in I^\prime_{(p,q,r)}.\tag*{\qed}
\end{equation}
\end{example}

Our proof for the theorem requires an ordering of $E^\prime(w)$ and
the associated partition of $D^\prime(w)$ as described before
Lemma~\ref{lem:dbipartition}.  We fix here our notation for this
partition.  Label $E^\prime(w)$ as $(p_1,q_1),\ldots,(p_k,q_k)$ in
such a way that $r_w(p_m,q_m)\leq r_w(p_{m^\prime},q_{m^\prime})$
whenever $m<m^\prime$.  Let $R_m$ be the subset of $D^\prime(w)$
defined as those $(x,y)\in D^\prime(w)$ which are (weakly) SW of
$(p_m,q_m)$ but not (weakly) SW of $(p_{m^\prime},q_{m^\prime})$ for
any $m^\prime<m$.  Let $r_m=r_w(p_m,q_m)$ for all $m\in\dbrac{1,k}$.

Our proof will be by induction and we fix notation for various
subideals of $I_w$ and $J_w$.  Let $I_0\subseteq I_w$ to be the ideal
generated by $z_{x,y}=f_{(x,y)}$ for $(x,y)\in D(w)$ such that
$r_w(x,y)=0$.  Let $I_m\subseteq I_w$ be the ideal
\[
I_m:=I_0+\sum_{m^\prime=1}^mI_{(p_{m^\prime},q_{m^\prime},r_{m^\prime})}.
\]
Similarly, define $J_0\subseteq J_w$ to be the ideal generated by
$z_{x,y}$ for $(x,y)\in D(w)$ such that $r_w(x,y)=0$, and
let
\[
J_m:=J_0+\langle f_{(x,y)} \rangle_{(x,y)\in R_{m^\prime},
m^\prime\leq m}.
\]

\medskip
\noindent
\emph{Proof of Theorem~\ref{thm:dbilci}.}
We show by induction that $I_m=J_m$ for each $m$.  This suffices since
$I_w=I_k$ and $J_w=J_k$ by definition.

The ideals $I_0$ and $J_0$ are equal by definition, since both are
generated by $f_{(x,y)}=d_{\{x\},\{y\}}=z_{x,y}$ for all $(x,y)$ where
$r_w(x,y)=0$.

Assume by induction that $I_{m-1}=J_{m-1}$.  By definition,
$J_{m}\subseteq I_m$, so we need to show $I_m\subseteq J_m$.  The
ideal $I_m=I_{m-1}+I_{(p_m,q_m,r_m)}$, and by
Lemma~\ref{lem:dbinejustify},
$I_{(p_m,q_m,r_m)}=I^\prime_{(p_m,q_m,r_m)}$.  Therefore, we only need
to show $I^\prime_{(p_m,q_m,r_m)}\subseteq J_m$.  

Given $x\in\dbrac{p_m,n-r_m}$ and $y\in\dbrac{1+r_m,q_m}$, let
\[
A(x,y)=\dbrac{p_m,p_m+r_m-1}\cup\{x+r_m\}
\]
and
\[
B(x,y)=\{y-r_m\}\cup\dbrac{q_m-r_m+1,q_m}.
\]
To show that
$I^\prime_{(p_m,q_m,r_m)}\subseteq J_m$, we need to show that
$d_{A(x,y),B(x,y)}\in J_m$ for all $x$ and $y$ satisfying the above
conditions.  We do so by further induction on $\deg d_{A(x,y),B(x,y)}$
(or equivalently on $x-y$).

Let $a$ and $b$ denote the height and width (in terms of the number of
boxes) of the rectangular region $R_m$, so that $(p_m+a-1,q_m-b+1)$ is
the SW-most box in $R_m$.  If $(x,y)\in R_m$, then
$d_{A(x,y),B(x,y)}\in J_m$ by definition.  Otherwise, $x\geq p_m+a$, or
$y\leq q_m-b$.  We give the remaining details of this proof only for
latter case where $y\leq q_m-b$ and leave the completely analogous
argument in the former case to the reader.

By Lemma~\ref{lem:dbipartition}, $r_w(x,y)=r_m$ for all $(x,y)\in
R_m$, so, in particular, $r_w(p_m,q_m-b+1)=r_m$.  Note
$r_w(p_m,q_m-b+1)$ is exactly the number of non-diagram boxes directly
W of $(p_m,q_m-b+1)$, and there are exactly $q_m-b$ columns to the
left of $(p_m,q_m-b+1)$.  Therefore, there exist $q_m-b-r_m$ diagram
boxes directly W of $(p_m,q_m-b+1)$.  In particular, since $1+r_m\leq
y\leq q_m-b$, $q_m-b-r_m\geq 1$, there exists a diagram box W of
$(p_m,q_m-b+1)$.  Let $q^\prime$ be the largest index such that
$q^\prime\leq q_m-b$ and $(p_m,q^\prime)\in D(w)$.  Note that
$r_m-r_w(p_m,q^\prime)$ is the number of non-diagram columns between
$q_m$ and $q^\prime$, so $q_m-b-q^\prime=r_m-r_w(p_m,q^\prime)$.
Since $y\leq q_m-b$, this implies $y-r_m\leq
q^\prime-r_w(p_m,q^\prime)$.

Since $(p_m,q^\prime)$ is the first diagram box directly W of $R_m$,
either $\Wpred(m)=0$, in which case $r_w(p_m,q^\prime)=0$, or
$(p_m,q^\prime)\in R_{\Wpred(m)}$.  If $r_w(p_m,q^\prime)=0$, then
$r_w(p_m,y-r_m)=0$, and $z_{p,y-r_m}\in I_0$ for all $p\geq p_m$.
Hence, all the entries in an entire column of the submatrix defining
$d_{A(x,y),B(x,y)}$ are in $I_0$, and $d_{A(x,y),B(x,y)}\in I_0$.

In the other case where $\Wpred(m)>0$, let $m^\prime=\Wpred(m)$,
noting that $q^\prime=q_{m^\prime}$ and
$r_w(p_m,q^\prime)=r_{m^\prime}$, so $y-r_m\leq
q_{m^\prime}-r_{m^\prime}$.  Now define sets $F$ and $G$
as follows.  Let
\[
F=\dbrac{p_{m^\prime},p_{m^\prime}+r_{m^\prime}-1}\cup\dbrac{p_m,p_m+r_m-1}\cup\{x+r_m\},
\]
and let
\[
G=\dbrac{q_{m^\prime}-r_{m^\prime}+1,q_{m^\prime}}\cup\dbrac{q_m-r_m+1,q_m}\cup\{y-r_m\}.
\]
Note that
\[
x+r_m> p_m+r_m-1> p_{m^\prime}+r_{m^\prime}-1
\]
and
\[
y-r_m<q_{m^\prime}-r_{m^\prime}+1<q_m-r_m+1;
\]
hence
\[
\#F=r_{m^\prime}+r_m+1-\max\{0,p_{m^\prime}+r_{m^\prime}-p_m\},
\]
while
\[
\#G=r_{m^\prime}+r_m+1-\max\{0,q_{m^\prime}-(q_m-r_m)\}.
\]
By Lemma~\ref{lem:dbilocation}, $r_m=q_m-p_m+1$ while
$r_{m^\prime}=q_{m^\prime}-p_{m^\prime}+1$.  This implies
\[
p_{m^\prime}+r_{m^\prime}-p_m=q_{m^\prime}-(q_m-r_m),
\]
so $F$ and $G$
have the same number of elements.  We let
\begin{align*}
c &=r_{m^\prime}+r_m+1-\#F \\
&=\max\{0,q_{m^\prime}-(q_m-r_m)\} \\
&=\#\left(\dbrac{q_{m^\prime}-r_{m^\prime}+1,q_{m^\prime}}\cap\dbrac{q_m-r_m+1,q_m}\right).
\end{align*}

Our proof is by expanding the determinant $d_{F,G}$ in two different
ways.  Let
\[
G^\prime=\{y-r_m\}\cup\dbrac{q_{m^\prime}-r_{m^\prime}+1,q_{m^\prime}},
\]
note that $G^\prime$ has $r_{m^\prime}$ elements, and consider the
Laplace expansion of $d_{F,G}$ using the columns in $G^\prime$, which
is given by
\[
d_{F,G}=\sum_{F^\prime\subseteq F} \pm
d_{F^\prime,G^\prime}d_{F\setminus F^\prime, G\setminus G^\prime},
\]
where the sum is over all subsets $F^\prime\subseteq F$ of
size $r_{m^\prime}$.  Since $y-r_m\leq q_{m^\prime}$,
$G^\prime\subseteq\dbrac{1,q_{m^\prime}}$, and
$F\subseteq\dbrac{q_{m^\prime},n}$, so $d_{F^\prime,G^\prime}$ is
in the ideal $I_{m^\prime}$ and hence by the inductive hypothesis in
$J_{m^\prime}\subseteq J_m$.  Therefore, $d_{F,G}\in
J_m$.

Now let
\[
F^{\prime\prime}=F\setminus
A(x,y)=\dbrac{p_{m^\prime},p_{m^\prime}+r_{m^\prime}-c-1},
\]
and
consider the Laplace expansion of $d_{F,G}$ using the rows of
$F^{\prime\prime}$, which is given by
\[
d_{F,G}=\sum_{G^{\prime\prime}\subseteq G} \pm
d_{F^{\prime\prime},G^{\prime\prime}}d_{A(x,y), G\setminus
G^{\prime\prime}}.
\]
Consider the term where
\[
G^{\prime\prime}=G\setminus
B(x,y)=\dbrac{q_{m^\prime}-r_{m^\prime}+1, q_{m^\prime}-c}.
\]
By Lemma~\ref{lem:dbilocation}, $r_{m^\prime}=q_{m^\prime}-p_{m^\prime}+1$, so
$q_{m^\prime}-r_{m^\prime}+1=p_{m^\prime}$, and
\[
q_{m^\prime}-c=p_{m^\prime}+r_{m^\prime}-c-1.
\]
Therefore, $F^{\prime\prime}=G^{\prime\prime}$, so
$d_{F^{\prime\prime},G^{\prime\prime}}=1$ for this choice of
$G^{\prime\prime}$.  Our Laplace expansion is therefore
\[
d_{F,G}=\pm
d_{A(x,y),B(x,y)} + \sum_{G^{\prime\prime}} \pm
d_{F^{\prime\prime},G^{\prime\prime}}d_{A(x,y), G\setminus
G^{\prime\prime}},
\]
where the sum is now over all
$G^{\prime\prime}\subseteq G$ of size $r_{m^\prime}-c$ other than
$G\setminus B(x,y)$.

If
\[
\sum_{u\in G^{\prime\prime}}u \geq \sum_{u\in G\setminus B(x,y)}
u,
\]
then there exists $u^\prime\in G^{\prime\prime}$ with
$u^\prime>q_{m^\prime}-c=p_{m^\prime}+r_{m^\prime}-c-1$.
Note
$u^\prime>p$ for all $p\in F^{\prime\prime}$, which implies
$d_{F^{\prime\prime},G^{\prime\prime}}=0$ as an entire column of the submatrix is $0$.

On the other hand, if
\[
\sum_{u\in G^{\prime\prime}}u <\sum_{u\in G\setminus B(x,y)} u,
\]
then
\begin{align*}
\deg d_{A(x,y), G\setminus G^{\prime\prime}} &= \deg d_{A(x,y),B(x,y)}-
\sum_{u\in G\setminus B(x,y)} u + \sum_{u\in G^{\prime\prime}}u\\ &<\deg d_{A(x,y),B(x,y)}.
\end{align*}
Since $d_{A(x,y), G\setminus G^{\prime\prime}}\in I_{(p_m,q_m,r_m)}=I^\prime_{(p_m,q_m,r_m)}$, by the inductive hypothesis for our induction on the degree of
$d_{A(x,y),B(x,y)}$,
\[
d_{A(x,y), G\setminus G^{\prime\prime}}\in J_m.
\]

Since $d_{F,G}\in J_m$ and
\[
d_{F^{\prime\prime},G^{\prime\prime}}d_{A(x,y), G\setminus
G^{\prime\prime}}\in J_m
\]
for all $G^{\prime\prime}\subseteq G$ of size
$r_{m^\prime}-c$ not including $G\setminus B(x,y)$, we have proven
$d_{A(x,y),B(x,y)}\in J_m$.\qed

\begin{example}
Consider $w=819732654$ as in Example~\ref{ex:ess-819732654}.  For
$m=2$, $(p_2,q_2)=(4,6)$, and $r_2=3$, our proof shows that
$d_{\{4,5,6,x+3\},\{y-3,4,5,6\}}\in J_m$ for all $x$ and $y$ with
$4\leq x\leq 6$ and $4\leq y\leq 6$.  The only cases where this is not
true by definition are those where $y=4$.  Note $m^\prime=\Wpred(2)=1$.

We illustrate the proof in the case where $x=5$ and $y=4$.  Consider the determinant
\[
d_{\{4,5,6,8\},\{1,4,5,6\}}=\begin{vmatrix}
z_{4,1} & 1 & 0 & 0 \\
z_{5,1} & z_{5,4} & 1 & 0 \\
z_{6,1} & z_{6,4} & z_{6,5} & 1 \\
z_{8,1} & z_{8,4} & z_{8,5} & z_{8,6}\end{vmatrix},
\]
which has degree $7$ in our grading.  Our proof writes
$d_{\{4,6,7,9\},\{1,2,4,6\}}$ in terms of determinants either of
smaller degree in $I_{(p_2,q_2,r_2)}$ or known to be in $I_{m^\prime}$ as
follows.

Here, $F=\{2,4,5,6,8\}$, $G=\{1,2,4,5,6\}$, and $c=0$.  Hence we consider
\[
d_{\{2,4,5,6,8\},\{1,2,4,5,6\}}=\begin{vmatrix}
z_{2,1} & 1 & 0 & 0 & 0\\
z_{4,1} & z_{4,2} & 1 & 0 & 0 \\
z_{5,1} & z_{5,2} & z_{5,4} & 1 & 0 \\
z_{6,1} & z_{6,2} & z_{6,4} & z_{6,5} & 1 \\
z_{8,1} & z_{8,2} & z_{8,4} & z_{8,5} & z_{8,6}\end{vmatrix}.
\]

For this example, $G^\prime=\{1,2\}$.  Note that the size $2$ minors
involving the columns 1 and 2 are in $I_{(2,2,1)}\subseteq I_1$.
Therefore, $d_{\{2,4,5,6,8\},\{1,2,4,5,6\}}\in I_1$.

Now, $F^{\prime\prime}=\{2\}$.  Therefore, we expand
$d_{\{2,4,5,6,8\},\{1,2,4,5,6\}}$ along row 2 to get
\[
d_{\{2,4,5,6,8\},\{1,2,4,5,6\}}=
z_{2,1}\begin{vmatrix}
z_{4,2} & 1 & 0 & 0 \\
z_{5,2} & z_{5,4} & 1 & 0 \\
z_{6,2} & z_{6,4} & z_{6,5} & 1 \\
z_{8,2} & z_{8,4} & z_{8,5} & z_{8,6}\end{vmatrix}
-1\begin{vmatrix}
z_{4,1} & 1 & 0 & 0 \\
z_{5,1} & z_{5,4} & 1 & 0 \\
z_{6,1} & z_{6,4} & z_{6,5} & 1 \\
z_{8,1} & z_{8,4} & z_{8,5} & z_{8,6}\end{vmatrix}+0-0+0.
\]
The second term is $-d_{\{4,5,6,8\}\{1,4,5,6\}}$, and the determinant
in the first term is in $I_2$ and has degree $6<7$.  Hence, by induction,
it is in $J_2$ and $d_{\{4,5,6,8\}\{1,4,5,6\}}\in J_2$ as desired. \qed
\end{example}

\subsection{The general case}
\label{sect:nondbisufficiency}

Now we treat the general case.  Let $w$ be almost defined by
inclusions, and let $(p,q)\in E^{\prime\prime}(w)$.  Let
\[
A^\prime(p,q)=\dbrac{p,p+r_w(p,q)}
\]
and
\[
B^\prime(p,q)=\dbrac{q-r_w(p,q),q}.
\]
Define
$f_{(p,q)}=d_{A^\prime(p,q),B^\prime(p,q)}$.
We show that $I_w$ is
generated by $I_v$ and $f_{(p,q)}$ for $(p,q)\in E^{\prime\prime}(w)$,
where $v$ is the defined by inclusions permutation associated to $w$
by Theorem~\ref{thm:dbifromadbi}.  We do so by showing the following
Lemma.

\begin{lemma}
\label{lem:adbigenerators}
Let $w$ be almost defined by inclusions, $v$ the defined by inclusions
permutation associated to $w$, and let $(p,q)\in E^{\prime\prime}(w)$.
Then
\[
I_{(p,q,r_w(p,q))}\subseteq I_v+\langle f_{(p,q)}\rangle.
\]
\end{lemma}

\begin{proof}
Let $r=r_w(p,q)$.  The ideal $I_{(p,q,r)}$ is generated by all
$d_{A,B}$ where $A\subseteq \dbrac{p,n}$, $B\subseteq
\dbrac{1,q}$, and $\#A=\#B=r+1$.  We need to show that
\[
d_{A,B}\in I_w+\langle f_{(p,q)}\rangle
\]
for all such $A$ and $B$.  We do so by
induction on the degree of $d_{A,B}$.  If $A=A^\prime(p,q)$ and
$B=B^\prime(p,q)$, then $d_{A,B}=f_{(p,q)}$ and we are done.  We are
left with the cases where $A\neq A^\prime(p,q)$ and where $B\neq
B^\prime(p,q)$.  We treat only the case where $B\neq B^\prime(p,q)$
and leave the entirely analogous argument in the other case to the
reader.

Since $w$ is almost defined by inclusions, $(p,q)$ satisfies either
condition W or condition X.  We treat these two cases separately.

Suppose $(p,q)$ satisfies condition X.  Then there exists
$(p^\prime,q)\in E(v)$ with $p^\prime<p$ and
\[
r_v(p^\prime,q)=r_w(p^\prime,q)=r+1.
\]
Let $b\in B^\prime(p,q)\setminus B$.  Consider the determinant
$d_{A\cup\{b\},B\cup\{b\}}$.  (If $b\in A$, we mean to consider the determinant of the matrix where row $b$ occurs twice.  This determinant is identically $0$, but we can still consider its expansions
formally.)  Note that
\[
b\geq q-r=q-r_v(p^\prime,q)+1.
\]
Since $(p^\prime,q)\in E(v)$ and $r_v(p^\prime,q)\neq 0$, by
Lemma~\ref{lem:dbilocation},
$p^\prime=q-r_w(p^\prime,q)+1$,
so $b\geq p^\prime$.  Therefore,
$A\cup\{b\}\subseteq\dbrac{p^\prime,n}$.
Also,
$B\cup\{b\}\subseteq\dbrac{1,q}$,
and
\[
\#A=\#B=r+1=r_v(p^\prime,q)+1,
\]
so
$d_{A\cup\{b\},B\cup\{b\}} \in I_v$.

Now we expand $d_{A\cup\{b\},B\cup\{b\}}$ along row $b$.  This
expansion is given by
\[
d_{A\cup\{b\},B\cup\{b\}}=\pm d_{A,B}+\sum_{b^\prime\in B}
z_{b,b^\prime}d_{A,B\cup\{b\}\setminus\{b^\prime\}}.
\]
If $b^\prime>b$, then $z_{b,b^\prime}=0$.  If $b^\prime<b$, then
\[
\deg d_{A,B\cup\{b\}\setminus\{b^\prime\}}=\deg d_{A,B}-b+b^\prime<\deg d_{A,B},
\]
so by the inductive hypothesis,
\[
d_{A,B\cup\{b\}\setminus\{b^\prime\}} \in I_w+\langle
f_{(p,q)}\rangle.
\]
Therefore, since every term of its expansion is in
$I_v+\langle f_{(p,q)}\rangle$,
\[
d_{A,B}\in I_v+\langle f_{(p,q)}\rangle.
\]

Next suppose $(p,q)$ satisfies Condition W.  Since $B\neq
B^\prime(p,q)$, $q>r+1$, so there exists $s<q$ with $(p,s)\in D(w)$.
Let $q^\prime$ be the largest such $s$.  If $q^\prime=q-1$, Condition
W states that there exists $p^\prime<p$ where $(p^\prime,q^\prime)\in
E(v)$.

If $q^\prime<q-1$, then $(p,q^\prime)$ is not in $E(w)$, because
otherwise $(p,q^\prime)$ would violate both Conditions Y and Z.  Since
$(p,q^\prime+1)\not\in D(w)$, this implies that there exists
$p^\prime<p$ with $(p^\prime,q^\prime)\in E(w)$.  Now
$(p^\prime+1,q^\prime)\in D(v)$ but $(p^\prime+1,q^\prime+1)\not\in
D(w)$, so $(p^\prime,q^\prime)$ does not satisfy Condition Y.

Furthermore, since $q^\prime<w^{-1}(p-1)\leq q$ (as $(p-1,q^\prime)\in
D(w)$ but $(p-1,q)\not\in D(w)$) but $w(s)>p$ for all $s$ with
$q^\prime<s<q$ (as $(p,s)\in D(w)$), it must be that $w(q)=p-1$.  Now
if there exists a unique $q^{\prime\prime}>q^\prime$ with
$(p^\prime,q^{\prime\prime})\in E(w)$, then $q^{\prime\prime}>q$
(since $w(s)\geq p-1\geq p^\prime$ for all $s$ with $q^\prime<s\leq
q$).  Note that $(p^\prime,q^\prime+1)$ and $(p^\prime,q)$ are two
distinct boxes not in $D(w)$ but between $(p^\prime,q^\prime)$ and
$(p^\prime,q^{\prime\prime})$, so
$r_w(p^\prime,q^{\prime\prime})-r_w(p^\prime,q^\prime)\geq 2$.
Therefore, $(p^\prime,q^\prime)$ does not satisfy Condition Z.  Since
$w$ is almost defined by inclusions, it must be the case that
$(p^\prime,q^\prime)$ satisfies Condition A or B, and hence
$(p^\prime,q^\prime)\in E(v)$.  Let $r^\prime=r_v(p^\prime,q^\prime)=r_w(p^\prime,q^\prime)$.

Let $G=B\cap\dbrac{1,q^\prime}$.  Note that, by the definition of
$q^\prime$, every box between $(p,q^\prime)$ and $(p,q)$ is not in
$D(v)$, and $r-r^\prime$ is the number of such boxes, so
$r-r^\prime=q-q^\prime-1$.  Since $B\subseteq\dbrac{1,q}$, this
implies
\[
\#G\geq r+1-(q-q^\prime)=r^\prime.
\]
If $\#G\geq r^\prime+1$,
then $d_{F,G}\in I_v$ for any $F\subseteq A$ with $\#G$ elements.  Since,
by Laplace expansion using the columns of $G$,
\[
d_{A,B}=\sum_{F\subseteq A} \pm d_{F,G}d_{A\setminus F,B\setminus G},
\]
it must follow that $d_{A,B}\in I_v$.

Otherwise, $G$ has exactly $r^\prime$ elements.  Let
\[
b\in \dbrac{q^\prime-r^\prime, q^\prime}\setminus G.
\]
We complete the
proof by expanding the determinant $d_{A\cup\{b\},B\cup\{b\}}$ in two
ways.  First let $G^\prime=G\cup\{b\}$, and consider the Laplace
expansion of $d_{A\cup\{b\},B\cup\{b\}}$ using the columns of
$G^\prime$, which is given by
\[
d_{A\cup\{b\},B\cup\{b\}}=\sum_{F^\prime\subseteq A\cup\{b\}} \pm
d_{F^\prime,G^\prime}d_{A\cup\{b\}\setminus F^\prime,
  B\cup\{b\}\setminus G^\prime},
\]
where the sum is over all subsets
$F^\prime\subseteq A\cup\{b\}$ of size $r^\prime$.  By
Lemma~\ref{lem:dbilocation}, $p^\prime=q^\prime-r^\prime+1$, so $b\geq
p^\prime$.  Therefore, $A\cup\{b\}\subseteq\dbrac{p^\prime,n}$, and
\[
d_{F^\prime,G^\prime}\in I_{(p^\prime,q^\prime,r^\prime)}\subseteq I_v
\]
for all choices of $F^\prime$.  Hence
$d_{A\cup\{b\},B\cup\{b\}}\in I_v$.

Now we expand $d_{A\cup\{b\},B\cup\{b\}}$ along row $b$ and use the
last part of the argument in the case where $(p,q)$ satisfies Condition X
to show that
$d_{A,B}\in I_v+\langle f_{(p,q)}\rangle$.
\end{proof}

\begin{example}
Consider $w=819372564$ as in Example~\ref{ex:ess-819372564}.  In this
case, $v=819732654$ is the permutation associated to $w$ by Theorem~\ref{thm:dbifromadbi}.  Lemma~\ref{lem:adbigenerators} states that
\[
I_{(4,4,2)}\subseteq
I_v+\langle d_{\{4,5,6\},\{2,3,4\}}\rangle
\]
and that
\[
I_{(6,7,3)}\subseteq I_v+\langle d_{\{6,7,8,9\},\{4,5,6,7\}}\rangle.
\]

To illustrate the proof, consider $d_{\{4,5,7\},\{1,3,4\}}\in I_{(4,4,2)}$, so
$(p,q)=(4,4)$ and $B^\prime(p,q)=\{2,3,4\}$.  Note the essential set
box at $(4,4)$ is of type WZ.  Here $(p^\prime,q^\prime)=(2,2)$, so
$r^\prime=1$.  Hence $G=\{1\}$, and $b=2$.

Therefore we consider
\[
d_{\{2,4,5,7\},\{1,2,3,4\}}=\begin{vmatrix}
z_{2,1} & 1 & 0 & 0 \\
z_{4,1} & z_{4,2} & z_{4,3} & 1 \\
z_{5,1} & z_{5,2} & z_{5,3} & z_{5,4} \\
z_{7,1} & z_{7,2} & z_{7,3} & z_{7,4}\end{vmatrix}.
\]
Since the size $2$ minors
involving the columns 1 and 2 are in $I_{(2,2,1)}\subseteq I_v$,
$d_{\{2,4,5,7\},\{1,2,3,4\}}\in I_v$.

Then we expand $d_{\{2,4,5,7\},\{1,2,3,4\}}$ along row $b=2$, getting
\[
d_{\{2,4,5,7\},\{1,2,3,4\}}=
z_{2,1}\begin{vmatrix}
z_{4,2} & z_{4,3} & 1 \\
z_{5,2} & z_{5,3} & z_{5,4} \\
z_{7,2} & z_{7,3} & z_{7,4}\end{vmatrix}
-1\begin{vmatrix}
z_{4,1} & z_{4,3} & 1 \\
z_{5,1} & z_{5,3} & z_{5,4} \\
z_{7,1} & z_{7,3} & z_{7,4}\end{vmatrix}+0-0.
\]
The second term is $-d_{\{4,5,7\},\{1,3,4\}}$, and the first term
involves $d_{\{4,5,7\},\{2,3,4\}}$, which is in $I_{(4,4,2)}$ and has
smaller degree.  Hence by induction $d_{\{4,5,7\},\{1,3,4\}}\in
I_v+\langle d_{\{4,5,6\},\{2,3,4\}}\rangle$.\qed
\end{example}

Finally we finish the proof of half of our theorem.

\begin{theorem}
Suppose $w$ avoids $52431$, $52341$, $53241$, $35142$, $42513$, and
$351624$.  Then $X_w$ is a local complete intersection.
\end{theorem}

\begin{proof}
By Theorem~\ref{thm:adbipatterns}, $w$ is almost defined by
inclusions.  Now by Theorem~\ref{thm:dbifromadbi}, there exists a
permutation $v$ which is defined by inclusions such that
$\ell(v)-\ell(w)$ is the number of boxes in $E^{\prime\prime}(w)$.  By
Theorem~\ref{thm:dbilci}, $I_v$ is generated by $\binom{n}{2}-\ell(v)$
polynomials.  Furthermore, by Lemma~\ref{lem:adbigenerators}, $I_w$ is
generated by $I_v$ and $\ell(v)-\ell(w)$ polynomials, so $I_w$ is
generated by $\binom{n}{2}-\ell(w)$ polynomials.  Hence, by
Proposition~\ref{prop:alleqns}, $X_w$ is a local complete intersection
in a neighborhood of $e_\id$.  Hence $X_w$ is a local complete
intersection.
\end{proof}

\section{Necessity}
Our strategy for the proof of the reverse direction is as follows.  We
identify two infinite families and eleven isolated intervals $[u,v]$
such that the Kazhdan--Lusztig variety $\mathcal{N}_{u,v}$ is not lci,
and hence $X_v$ is not lci at $e_u$.  It follows from~\cite[Thm.\
2.6]{WYGov} that $X_w$ is not lci if $w$ interval contains $[u,v]$.
We show that, if $w$ contains one of the six given patterns, then $w$
will interval contain either one of the eleven intervals or an
interval from one of the two infinite families.  This is accomplished
using marked mesh patterns, which were previously defined by the first author in~\cite[Def.~24]{Ulf}.

We now list our two infinite families and eleven isolated intervals,
along with drawings as mesh patterns.

\textbf{Family $\mathbf A$} consists of intervals of the form
\[
[(a+1)a\cdots1(a+b+2)\cdots(a+2), (a+b+2)(a+1)a\cdots2(a+b+1)\cdots(a+2)1],
\]
where $a,b>0$ and $a>1$ or $b>1$. We list the first few members of the family in Figure~\ref{fig:famA}.
\begin{figure}[htbp]
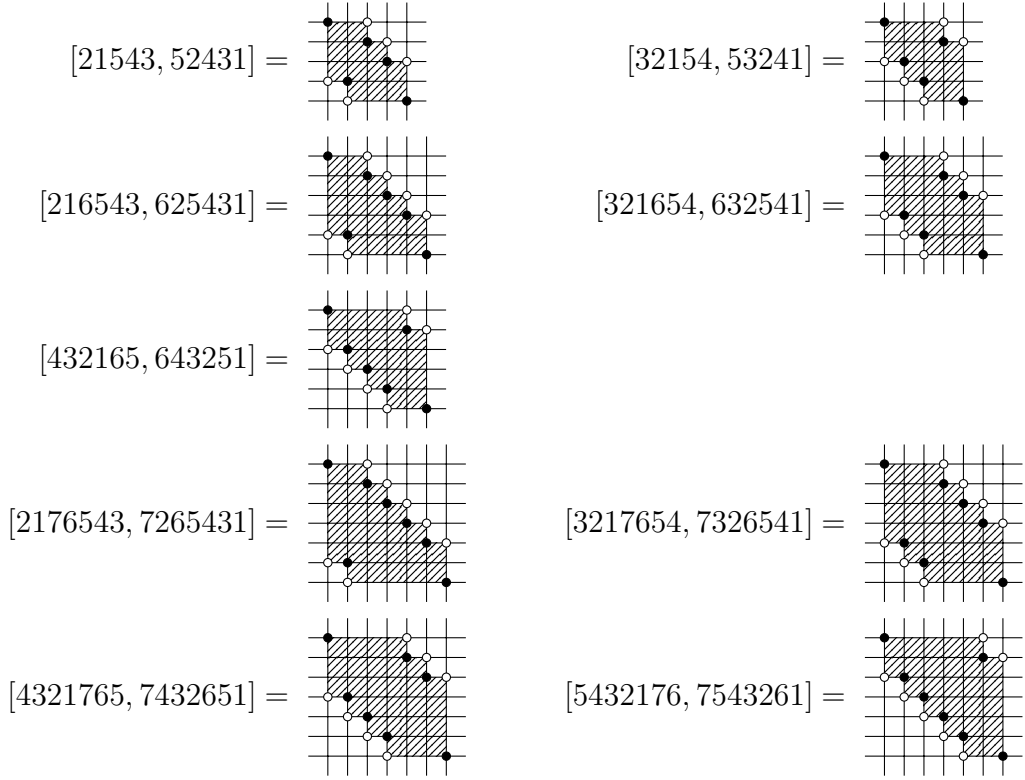

\begin{center}

\begin{align*}
[21543,52431] &=
\impattern{scale=0.75}{ 5 }{ 1/5, 2/2, 3/4, 4/3, 5/1 }
{ 1/2, 2/1, 3/5, 4/4, 5/3 }{ 1/2, 1/3, 1/4, 2/1, 2/2, 2/3, 2/4, 3/1, 3/2, 3/3, 4/1, 4/2 }
&[32154, 53241] &=
\impattern{scale=0.75}{ 5 }{ 1/5, 2/3, 3/2, 4/4, 5/1 }
{ 1/3, 2/2, 3/1, 4/5, 5/4 }{ 1/3, 1/4, 2/2, 2/3, 2/4, 3/1, 3/2, 3/3, 3/4, 4/1, 4/2, 4/3 } \\
[216543, 625431] &= 
\impattern{scale=0.75}{ 6 }{ 1/6, 2/2, 3/5, 4/4, 5/3, 6/1 }
{ 1/2, 2/1, 3/6, 4/5, 5/4, 6/3 }{ 1/2, 1/3, 1/4, 1/5, 2/1, 2/2, 2/3, 2/4, 2/5, 3/1, 3/2, 3/3, 3/4, 4/1, 4/2, 4/3, 5/1, 5/2 }
&[321654, 632541] &=
\impattern{scale=0.75}{ 6 }{ 1/6, 2/3, 3/2, 4/5, 5/4, 6/1 }
{ 1/3, 2/2, 3/1, 4/6, 5/5, 6/4 }{ 1/3, 1/4, 1/5, 2/2, 2/3, 2/4, 2/5, 3/1, 3/2, 3/3, 3/4, 3/5, 4/1, 4/2, 4/3, 4/4, 5/1, 5/2, 5/3 } \\
[432165, 643251] &=
\impattern{scale=0.75}{ 6 }{ 1/6, 2/4, 3/3, 4/2, 5/5, 6/1 }
{ 1/4, 2/3, 3/2, 4/1, 5/6, 6/5 }{ 1/4, 1/5, 2/3, 2/4, 2/5, 3/2, 3/3, 3/4, 3/5, 4/1, 4/2, 4/3, 4/4, 4/5, 5/1, 5/2, 5/3, 5/4 } \\
[2176543, 7265431] &= 
\impattern{scale=0.75}{ 7 }{ 1/7, 2/2, 3/6, 4/5, 5/4, 6/3, 7/1 }
{ 1/2, 2/1, 3/7, 4/6, 5/5, 6/4, 7/3 }{ 1/2, 1/3, 1/4, 1/5, 1/6, 2/1, 2/2, 2/3, 2/4, 2/5, 2/6, 3/1, 3/2, 3/3, 3/4, 3/5, 4/1, 4/2, 4/3, 4/4, 5/1, 5/2, 5/3, 6/1, 6/2 }
&[3217654, 7326541] &=
\impattern{scale=0.75}{ 7 }{ 1/7, 2/3, 3/2, 4/6, 5/5, 6/4, 7/1 }
{ 1/3, 2/2, 3/1, 4/7, 5/6, 6/5, 7/4 }{ 1/3, 1/4, 1/5, 1/6, 2/2, 2/3, 2/4, 2/5, 2/6, 3/1, 3/2, 3/3, 3/4, 3/5, 3/6, 4/1, 4/2, 4/3, 4/4, 4/5, 5/1, 5/2, 5/3, 5/4, 6/1, 6/2, 6/3 } \\
[4321765, 7432651] &=
\impattern{scale=0.75}{ 7 }{ 1/7, 2/4, 3/3, 4/2, 5/6, 6/5, 7/1 }
{ 1/4, 2/3, 3/2, 4/1, 5/7, 6/6, 7/5 }{ 1/4, 1/5, 1/6, 2/3, 2/4, 2/5, 2/6, 3/2, 3/3, 3/4, 3/5, 3/6, 4/1, 4/2, 4/3, 4/4, 4/5, 4/6, 5/1, 5/2, 5/3, 5/4, 5/5, 6/1, 6/2, 6/3, 6/4 }
&[5432176, 7543261] &=
\impattern{scale=0.75}{ 7 }{ 1/7, 2/5, 3/4, 4/3, 5/2, 6/6, 7/1 }
{ 1/5, 2/4, 3/3, 4/2, 5/1, 6/7, 7/6 }{ 1/5, 1/6, 2/4, 2/5, 2/6, 3/3, 3/4, 3/5, 3/6, 4/2, 4/3, 4/4, 4/5, 4/6, 5/1, 5/2, 5/3, 5/4, 5/5, 5/6, 6/1, 6/2, 6/3, 6/4, 6/5 }
\end{align*}

\caption{The first few members of the family $A$}
\label{fig:famA}
\end{center}
\end{figure}

\textbf{Family $\mathbf B$} consists of intervals of the form
\begin{align*}
[&(a+1)\cdots1(a+3)(a+2)(a+b+4)\cdots(a+4), \\
&(a+3)(a+1)\cdots2(a+b+4)1(a+b+3)\cdots(a+4)(a+2)],
\end{align*}
where $a,b\geq 0$ and $a+b\geq 1$. We list the first few members of the family in Figure~\ref{fig:famB}.
\begin{figure}[htbp]
\begin{center}
\begin{align*}
[13254, 35142] &=
\impattern{scale=0.75}{ 5 }{ 1/3, 2/5, 3/1, 4/4, 5/2 }
{ 1/1, 2/3, 3/2, 4/5, 5/4 }{ 1/1, 1/2, 2/1, 2/2, 2/3, 2/4, 3/2, 3/3, 3/4, 4/2, 4/3 }
&[21435, 42513] &=
\impattern{scale=0.75}{ 5 }{ 1/4, 2/2, 3/5, 4/1, 5/3 }
{ 1/2, 2/1, 3/4, 4/3, 5/5 }{ 1/2, 1/3, 2/1, 2/2, 2/3, 3/1, 3/2, 3/3, 3/4, 4/3, 4/4 } \\
[132654, 361542] &=
\impattern{scale=0.75}{ 6 }{ 1/3, 2/6, 3/1, 4/5, 5/4, 6/2 }
{ 1/1, 2/3, 3/2, 4/6, 5/5, 6/4 }{ 1/1, 1/2, 2/1, 2/2, 2/3, 2/4, 2/5, 3/2, 3/3, 3/4, 3/5, 4/2, 4/3, 4/4, 5/2, 5/3 }
&[214365, 426153] &=
\impattern{scale=0.75}{ 6 }{ 1/4, 2/2, 3/6, 4/1, 5/5, 6/3 }
{ 1/2, 2/1, 3/4, 4/3, 5/6, 6/5 }{ 1/2, 1/3, 2/1, 2/2, 2/3, 3/1, 3/2, 3/3, 3/4, 3/5, 4/3, 4/4, 4/5, 5/3, 5/4 } \\
[321546, 532614] &=
\impattern{scale=0.75}{ 6 }{ 1/5, 2/3, 3/2, 4/6, 5/1, 6/4 }
{ 1/3, 2/2, 3/1, 4/5, 5/4, 6/6 }{ 1/3, 1/4, 2/2, 2/3, 2/4, 3/1, 3/2, 3/3, 3/4, 4/1, 4/2, 4/3, 4/4, 4/5, 5/4, 5/5 } \\
[1327654, 3716542] &= 
\impattern{scale=0.75}{ 7 }{ 1/3, 2/7, 3/1, 4/6, 5/5, 6/4, 7/2 }
{ 1/1, 2/3, 3/2, 4/7, 5/6, 6/5, 7/4 }{ 1/1, 1/2, 2/1, 2/2, 2/3, 2/4, 2/5, 2/6, 3/2, 3/3, 3/4, 3/5, 3/6, 4/2, 4/3, 4/4, 4/5, 5/2, 5/3, 5/4, 6/2, 6/3}
&[2143765, 4271653] &=
\impattern{scale=0.75}{ 7 }{ 1/4, 2/2, 3/7, 4/1, 5/6, 6/5, 7/3 }
{ 1/2, 2/1, 3/4, 4/3, 5/7, 6/6, 7/5 }{ 1/2, 1/3, 2/1, 2/2, 2/3, 3/1, 3/2, 3/3, 3/4, 3/5, 3/6, 4/3, 4/4, 4/5, 4/6, 5/3, 5/4, 5/5, 6/3, 6/4 } \\
[3215476, 5327164] &=
\impattern{scale=0.75}{ 7 }{ 1/5, 2/3, 3/2, 4/7, 5/1, 6/6, 7/4 }
{ 1/3, 2/2, 3/1, 4/5, 5/4, 6/7, 7/6 }{ 1/3, 1/4, 2/2, 2/3, 2/4, 3/1, 3/2, 3/3, 3/4, 4/1, 4/2, 4/3, 4/4, 4/5, 4/6, 5/4, 5/5, 5/6, 6/4, 6/5 }
&[4321657, 6432715] &=
\impattern{scale=0.75}{ 7 }{ 1/6, 2/4, 3/3, 4/2, 5/7, 6/1, 7/5 }
{ 1/4, 2/3, 3/2, 4/1, 5/6, 6/5, 7/7 }{ 1/4, 1/5, 2/3, 2/4, 2/5, 3/2, 3/3, 3/4, 3/5, 4/1, 4/2, 4/3, 4/4, 4/5, 5/1, 5/2, 5/3, 5/4, 5/5, 5/6, 6/5, 6/6 }
\end{align*}

\caption{The first few members of the family $B$}
\label{fig:famB}
\end{center}
\end{figure}

We list the exceptional intervals in Figure~\ref{fig:except_ints}.
\begin{figure}[htbp]
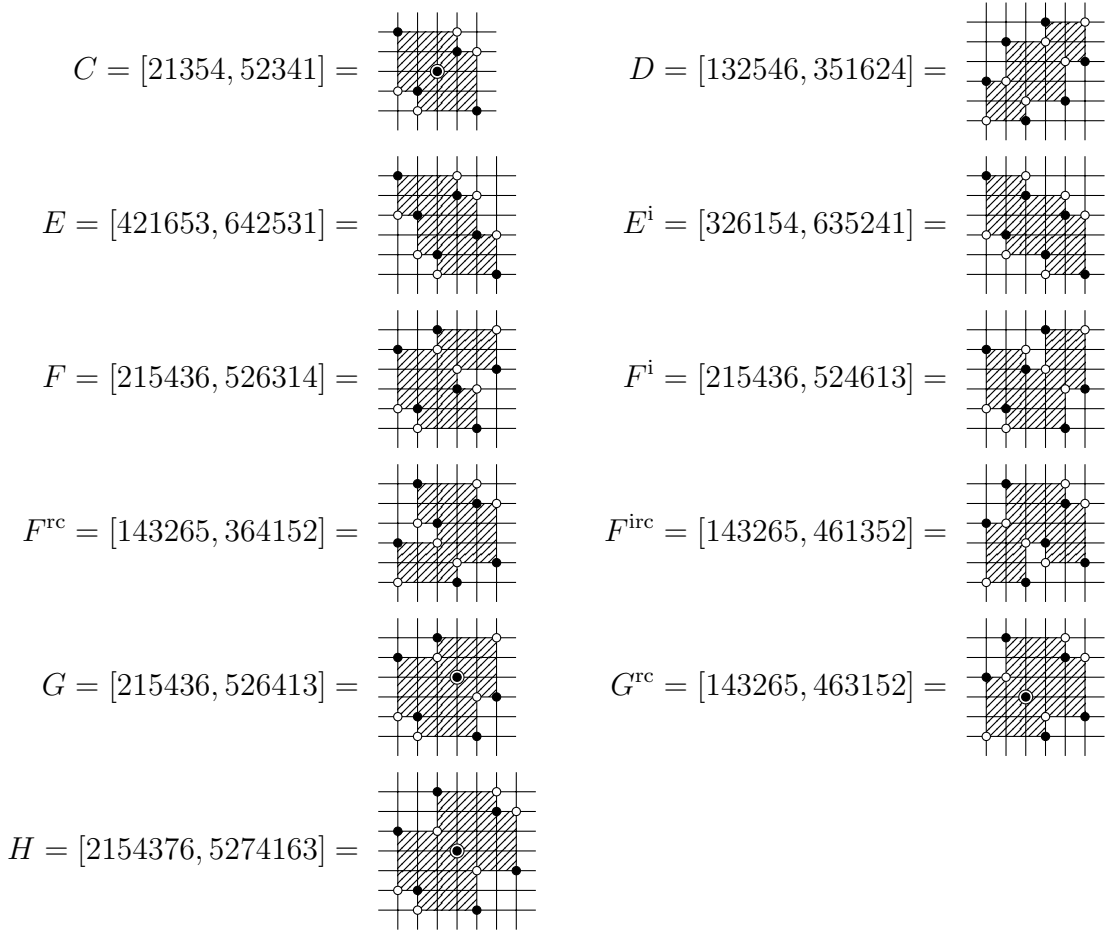

\begin{center}
\begin{align*}
C = [21354, 52341] &=
\imopattern{scale=0.75}{ 5 }{ 1/5, 2/2, 3/3, 4/4, 5/1 }
{ 1/2, 2/1, 4/5, 5/4 }{ 3/3 }{ 1/2, 1/3, 1/4, 2/1, 2/2, 2/3, 2/4, 3/1, 3/2, 3/3, 3/4, 4/1, 4/2, 4/3 }
&D = [132546, 351624] &=
\impattern{scale=0.75}{ 6 }{ 1/3, 2/5, 3/1, 4/6, 5/2, 6/4 }
{ 1/1, 2/3, 3/2, 4/5, 5/4, 6/6 }{ 1/1, 1/2, 2/1, 2/2, 2/3, 2/4, 3/2, 3/3, 3/4, 4/2, 4/3, 4/4, 4/5, 5/4, 5/5 } \\
E = [421653, 642531] &=
\impattern{scale=0.75}{ 6 }{ 1/6, 2/4, 3/2, 4/5, 5/3, 6/1 }
{ 1/4, 2/2, 3/1, 4/6, 5/5, 6/3 }{ 1/4, 1/5, 2/2, 2/3, 2/4, 2/5, 3/1, 3/2, 3/3, 3/4, 3/5, 4/1, 4/2, 4/3, 4/4, 5/1, 5/2 }
&E^{\rmi} = [326154, 635241] &=
\impattern{scale=0.75}{ 6 }{ 1/6, 2/3, 3/5, 4/2, 5/4, 6/1 }
{ 1/3, 2/2, 3/6, 4/1, 5/5, 6/4 }{ 1/3, 1/4, 1/5, 2/2, 2/3, 2/4, 2/5, 3/2, 3/3, 3/4, 4/1, 4/2, 4/3, 4/4, 5/1, 5/2, 5/3 } \\
F = [215436, 526314] &=
\impattern{scale=0.75}{ 6 }{ 1/5, 2/2, 3/6, 4/3, 5/1, 6/4 }
{ 1/2, 2/1, 3/5, 4/4, 5/3, 6/6 }{ 1/2, 1/3, 1/4, 2/1, 2/2, 2/3, 2/4, 3/1, 3/2, 3/3, 3/4, 3/5, 4/1, 4/2, 4/4, 4/5, 5/4, 5/5 }
&F^{\rmi} = [215436, 524613] &=
\impattern{scale=0.75}{ 6 }{ 1/5, 2/2, 3/4, 4/6, 5/1, 6/3 }
{ 1/2, 2/1, 3/5, 4/4, 5/3, 6/6 }{ 1/2, 1/3, 1/4, 2/1, 2/2, 2/3, 2/4, 3/1, 3/2, 3/3, 4/1, 4/2, 4/3, 4/4, 4/5, 5/3, 5/4, 5/5 } \\
F^{\rmrc} = [143265, 364152] &=
\impattern{scale=0.75}{ 6 }{ 1/3, 2/6, 3/4, 4/1, 5/5, 6/2 }
{ 1/1, 2/4, 3/3, 4/2, 5/6, 6/5 }{ 1/1, 1/2, 2/1, 2/2, 2/4, 2/5, 3/1, 3/2, 3/3, 3/4, 3/5, 4/2, 4/3, 4/4, 4/5, 5/2, 5/3, 5/4 }
&F^{\rmirc} = [143265, 461352] &=
\impattern{scale=0.75}{ 6 }{ 1/4, 2/6, 3/1, 4/3, 5/5, 6/2 }
{ 1/1, 2/4, 3/3, 4/2, 5/6, 6/5 }{ 1/1, 1/2, 1/3, 2/1, 2/2, 2/3, 2/4, 2/5, 3/3, 3/4, 3/5, 4/2, 4/3, 4/4, 4/5, 5/2, 5/3, 5/4 } \\
G = [215436, 526413] &=
\imopattern{scale=0.75}{ 6 }{ 1/5, 2/2, 3/6, 4/4, 5/1, 6/3 }
{ 1/2, 2/1, 3/5, 5/3, 6/6 }{ 4/4 }{ 1/2, 1/3, 1/4, 2/1, 2/2, 2/3, 2/4, 3/1, 3/2, 3/3, 3/4, 3/5, 4/1, 4/2, 4/3, 4/4, 4/5, 5/3, 5/4, 5/5 } \quad
&G^{\rmrc} = [143265, 463152] &=
\imopattern{scale=0.75}{ 6 }{ 1/4, 2/6, 3/3, 4/1, 5/5, 6/2 }
{ 1/1, 2/4, 4/2, 5/6, 6/5 }{ 3/3 }{ 1/1, 1/2, 1/3, 2/1, 2/2, 2/3, 2/4, 2/5, 3/1, 3/2, 3/3, 3/4, 3/5, 4/2, 4/3, 4/4, 4/5, 5/2, 5/3, 5/4 } \\
H = [2154376, 5274163] &=
\imopattern{scale=0.75}{ 7 }{ 1/5, 2/2, 3/7, 4/4, 5/1, 6/6, 7/3 }
{ 1/2, 2/1, 3/5, 5/3, 6/7, 7/6 }{ 4/4 }{ 1/2, 1/3, 1/4, 2/1, 2/2, 2/3, 2/4, 3/1, 3/2, 3/3, 3/4, 3/5, 3/6, 4/1, 4/2, 4/3, 4/4, 4/5, 4/6, 5/3, 5/4, 5/5, 5/6, 6/3, 6/4, 6/5 } 
\end{align*}
\caption{The exceptional intervals}
\label{fig:except_ints}
\end{center}
\end{figure}

Note the set of intervals listed is invariant under two symmetries.
The {\bf reverse complement} of a permutation $v$, an interval pattern
$[u,v]$, or a mesh pattern $(v,R)$ is obtained by rotating the graph
$G(v)$ along with the conditions from $R$ (or from $u$) by 180
degrees, or equivalently by conjugating both $u$ and $v$ by $w_0$ and
moving $R$ accordingly.  The {\bf inverse} of a permutation $v$ is the
permutation $v^{-1}$; this is equivalent to reflecting the graph
$G(v)$ along with the conditions across the line $y=x$.  One can of
course also apply both symmetries.  Geometrically, one can explain
this invariance by noting that
$\mathcal{N}_{u,v}\cong\mathcal{N}_{w_0uw_0,w_0vw_0}\cong\mathcal{N}_{u^{-1},v^{-1}}\cong\mathcal{N}_{w_0u^{-1}w_0,w_0v^{-1}w_0}.$
Note that $w$ avoids $v$ (or $[u,v]$ or $(v,R)$) if and only if the
reverse complement of $w$ avoids the reverse complement of $v$ (or of
$[u,v]$ or of $(v,R)$), and the same holds for the inverse and also for the
inverse of the reverse complement.  We will appeal to these symmetries
in the remainder of this section.

\begin{proposition} \label{prop:nonlciinterv}
For each of the intervals $[u,v]$ in Family A, Family B, and listed in
Figure~\ref{fig:except_ints}, the variety $\mathcal{N}_{u,v}$ is not
lci (precisely only at the origin $\mathbf{0}$).
\end{proposition}

\begin{proof}
The varieties $\mathcal{N}_{u,v}$ for the infinite families A and B
were determined independently by Cortez~\cite{Cor} and
Manivel~\cite{Man}.  For Family A, they are the varieties of
$(a+1)\times (b+1)$ matrices of rank at most 1; these are not lci except
in the case where $a=b=1$.  For Family B, they are the varieties of
$(a+b+2)\times 2$ matrices of rank at most 1; these are not lci unless
$a+b+2=2$.  In the case of Family A, this is immediate from the
definition of the ideal $I_{v,w}$ given in Subsection~\ref{ss:loc-KL-ee}.  For Family B,
this description requires a change in coordinates changing half the
variables by a sign.  Note that for both of these families,
$\mathcal{N}_{u,v}$ is only singular at the origin, and $X_u$ is in
fact an irreducible component of the singular locus of $X_v$, so some
generic singularity of $X_v$ is not lci.

For the intervals in C, D, E, and F, the variety $\mathcal{N}_{u,v}$
is isomorphic (after respectively 0, 2, 0, and 1 change in sign) to
the variety of $3\times 3$ ``matrices'' with a corner cut out of
``rank'' 1.  In other words, their coordinate ring is generated by the
$2\times 2$ ``determinants'' of the ``matrix''
\[
\begin{pmatrix}
a & b & \\
c & d & e \\
f & g & h
\end{pmatrix},
\]
namely $ad-bc$, $ag-bf$, $cg-df$, $ch-ef$, and $dh-eg$.  This is a
codimension 3 subvariety which is the vanishing locus of 5 linearly
independent degree 2 polynomials and hence is not lci.  Note that
this variety has two components in its singular locus which meet at
$\mathbf{0}$.

For the two intervals in G, the variety $\mathcal{N}_{u,v}$ is
isomorphic to one defined by equations $ae-bd$, $af-cd$, $bf-ce$,
$ag+bh+ci$, and $dg+eh+fi$ over the variables
\[\begin{pmatrix}
a & b & c\\
d & e & f \\
g & h & i
\end{pmatrix}.
\]
The equations state that the first two rows are dependent and both
orthogonal to the third row.  This is a codimension 3 subvariety which
is the vanishing locus of 5 polynomials and hence not lci.

For the interval in H, the variety $\mathcal{N}_{u,v}$ is isomorphic
to one defined by equations $ae-bd$, $af-cd$, $bf-ce$, $gk-hj$,
$gl-ij$, $hl-ik$, $ag+bh+ci$, $dg+eh+fi$, $aj+bk+cl$, and $dj+ek+fl$
over the variables
\[
\begin{pmatrix}
a & b & c\\
d & e & f \\
g & h & i \\
j & k & l
\end{pmatrix}.
\]
The equations say that the first two rows are dependent, the last two rows are dependent, and the first two rows are orthogonal to the last two rows.  This is a codimension 5 subvariety which is the vanishing locus of 10
polynomials and hence not lci.
\end{proof}
The calculations in the proof can be verified by hand as in Example~\ref{ex:kleqns526314} or by using the Macaulay 2
code accompanying~\cite{WYGov} currently available at the both authors' websites.

We believe that this list is not merely helpful in proving our
theorem but is actually the complete list specifying exactly the non-lci
loci of all Schubert varieties in the sense of~\cite[Thm. 2.6]{WYGov}.

\begin{theorem} \label{thm:nec}
If the permutation $w$ contains one of the patterns
$53241$, $52341$, $52431$, $35142$, $42513$, and $351624$,
then the Schubert variety $X_w$ is not a local complete intersection.
\end{theorem}

In Lemmas~\ref{lem:lcilem1}, \ref{lem:lcilem2}, and \ref{lem:lcilem3}, we show that, if a permutation $w$ contains one of the six classical patterns in the theorem, then it also interval contains one of the intervals in family $A$ or $B$ or one of the exceptional intervals in Figure~\ref{fig:except_ints}. Then the Schubert variety $X_w$ is not lci by Proposition~\ref{prop:nonlciinterv} and~\cite[Thm. 2.6]{WYGov}.

Note that a permutation $w$ contains at least one of the
first three patterns in the theorem if and only if it contains the marked mesh pattern
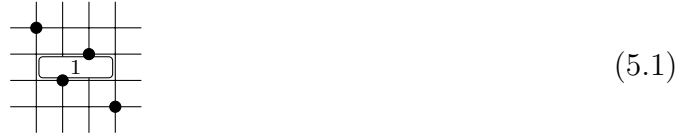
\begin{equation} \label{lcimeshpatt1}
\patternsbm{scale=1}{ 4 }{ 1/4, 2/2, 3/3, 4/1 }{}{1/2/4/3/1}
\end{equation}
where the area marked with a $1$ must contain at least one element.

\begin{lemma} \label{lem:lcilem1}
A permutation $w$ contains the marked mesh pattern \eqref{lcimeshpatt1}
if and only if it contains a mesh pattern from family $A$, the
exceptional $C$, or one of the exceptionals $E$, $E^{\rmi}$.
\end{lemma}

In the proof below we will talk about ``staircasing'' boxes.  Depending on context, this will mean picking either all of the NE-most or all of the SW-most elements in the box.  A NE-most element $i$ is one which has no elements to its NE in the box; to be specific, this means for all $j$ in the box, either $j<i$ or $w(j)<w(i)$.  Similarly, a SW-most element $i$ is one with no elements to its SW, so for all $j$ in the box, either $j>i$ or $w(j)>w(i)$.  For example, suppose that we have an occurrence of the mesh pattern
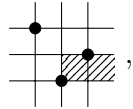
\[
\pattern{scale=1}{ 3 }{ 1/3, 2/1, 3/2 }{2/1,3/1}
\]
in a permutation. Assume that for this particular occurrence there are exactly four elements in the box $(1,1)$, and furthermore, that these four elements have the pattern $1423$. By ``staircasing'' this box we mean adding the elements corresponding to $3$ and $4$ (in the box) and disregarding the other elements. This produces the following pattern.
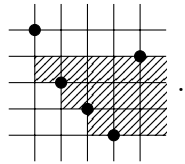
\[
\pattern{scale=1}{ 5 }{ 1/5, 2/3, 3/2, 4/1, 5/4 }{4/1,5/1,4/2,5/2,4/3,5/3, 3/1,3/2,3/3, 2/2,2/3, 1/3}
\]

In the following three lemmas we will need to use the ``shading lemma'' of Hilmarsson, J{\'o}nsd{\'o}ttir, Sigur{\dh}ard{\'o}ttir, {\'Ulfarsson} and Vi{\dh}arsd{\'o}ttir~\cite{HJSV11}, which gives conditions under which we can shade extra boxes in a particular mesh pattern without changing the set of permutations that avoid the pattern. We will use this lemma so frequently in the proofs below that when we add more shading to a pattern without giving a particular reason, the reader can assume this lemma is being used.  The arguments for this lemma are elementary.

\begin{proof}
The ``if'' part is easily verified. We consider two cases for the ``only if'' part.
\begin{enumerate}

\item \label{meshproof1-case1} There exists an occurrence of \eqref{lcimeshpatt1} in which the second and third boxes of the marked region are empty or in which the first and second boxes of the marked region are empty.

The case where the first and second boxes of the marked region are empty follows from the other case by the reverse complement symmetry, since the set of stated mesh patterns is invariant under reverse complement.

Therefore we can assume that we have an occurrence of
the pattern on the left below, which implies an occurrence of the pattern on the right.
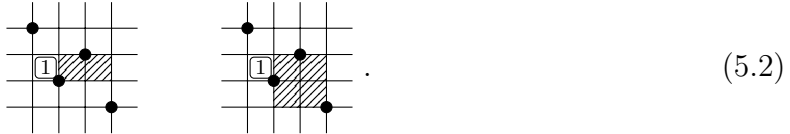
\begin{equation} \label{lcimeshpatt2}
\patternsbm{scale=1}{ 4 }{ 1/4, 2/2, 3/3, 4/1 }{2/2,3/2}{1/2/2/3/1} \qquad
\patternsbm{scale=1}{ 4 }{ 1/4, 2/2, 3/3, 4/1 }{2/1,3/1,2/2,3/2}{1/2/2/3/1}
\end{equation}
If we choose the right-most element in the marked box in the occurrence of the pattern on the right, we have an occurrence of the pattern below.
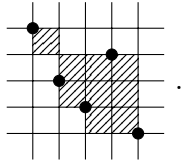
\begin{equation*} 
\patternsbm{scale=1}{ 5 }{ 1/5, 2/3, 3/2, 4/4, 5/1 }{1/4, 2/2, 2/3, 3/1, 3/2, 3/3, 4/1, 4/2, 4/3}{}
\end{equation*}
Consider the box with lower left corner with coordinates $(2,4)$. If this box is empty, we can staircase
boxes labelled $(1,3)$ and $(3,4)$, which will produce an occurrence of a member of family $A$. If box
$(2,4)$ is not empty, we choose the lowest element in that box. This will produce an occurrence of the pattern below.
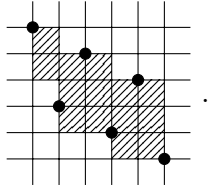
\begin{equation*} 
\patternsbm{scale=1}{ 6 }{ 1/6, 2/3, 3/5, 4/2, 5/4, 6/1 }{1/4, 1/5, 2/2, 2/3, 2/4, 3/2, 3/3, 3/4, 4/1, 4/2, 4/3, 5/1, 5/2, 5/3}{}
\end{equation*}
If all of the boxes $(1,3)$, $(2,5)$ and $(4,4)$ in this pattern are empty, we have an occurrence of the exceptional $E^\rmi$; otherwise we
staircase these boxes and produce a member of the family $A$.

\item \label{meshproof1-case3} In every occurrence of \eqref{lcimeshpatt1}, either the second or third box of the marked region is non-empty and either the left or right box of the marked region is non-empty.

Since the second or third box is non-empty, we contain the pattern below.
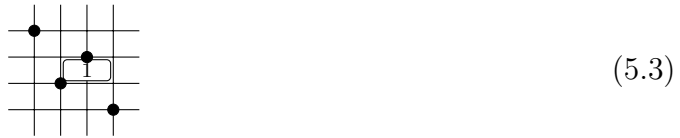
\begin{equation} \label{lcimeshpatt5}
\patternsbm{scale=1}{ 4 }{ 1/4, 2/2, 3/3, 4/1 }{}{2/2/4/3/1}
\end{equation}
We choose the lowest element in these two boxes and we consider two cases -- depending on whether this element is in the first or second box.

\begin{enumerate}

\item \label{meshproof1-case3a}
If the lowest element is in box $(2,2)$ of the pattern in \eqref{lcimeshpatt5} we produce an occurrence of the following pattern
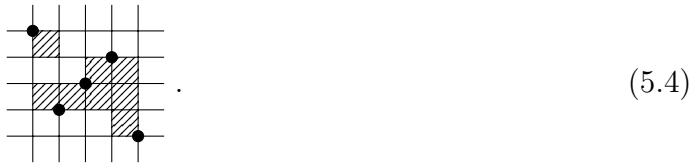
\begin{equation} \label{lcimeshpatt6}
\patternsbm{scale=1}{ 5 }{ 1/5, 2/2, 3/3, 4/4, 5/1 }{1/2, 1/4, 2/2, 3/2, 3/3, 4/1, 4/2, 4/3}{}
\end{equation}
Here we can assume the boxes $(1,3)$ and $(2,3)$ are empty since otherwise we have an occurrence of the pattern on the left in~\eqref{lcimeshpatt1},
which we handled above in case \eqref{meshproof1-case1}. Similarly, we can assume that the boxes $(2,1)$
and $(3,1)$ are empty. Finally the boxes $(2,4)$ and $(3,4)$ can be assumed to be empty; otherwise we
are back in case \eqref{meshproof1-case1} (with the first and second boxes empty). Having shaded these extra boxes in the pattern \eqref{lcimeshpatt6}, we
have produced the exceptional $C$.

\item \label{meshproof1-case3b}
If the lowest element is in box $(3,2)$ of the pattern \eqref{lcimeshpatt5} we produce an occurrence of the following pattern.
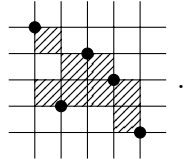
\begin{equation*}
\patternsbm{scale=1}{ 5 }{ 1/5, 2/2, 3/4, 4/3, 5/1 }{1/2, 1/4, 2/2, 2/3, 3/2, 3/3, 4/1, 4/2}{}
\end{equation*}
Now we can assume that the box $(3,1)$ is empty since otherwise we would have an occurrence of the pattern in \eqref{lcimeshpatt2}. We consider two cases, depending on whether the box $(1,3)$ is empty or not.

\begin{enumerate}

\item The box $(1,3)$ is not empty.
This produces an occurrence of the pattern below.
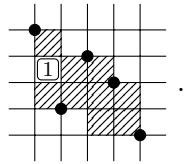
\begin{equation*}
\patternsbm{scale=1}{ 5 }{ 1/5, 2/2, 3/4, 4/3, 5/1 }{1/2, 1/4, 2/2, 2/3, 3/1, 3/2, 3/3, 4/1, 4/2}{1/3/2/4/{1}}
\end{equation*}
We can assume the marked box has a unique element that is both right-most and top-most because we would otherwise produce a member
of family $A$. We therefore have an occurrence of the following pattern.
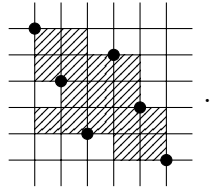
\begin{equation*}
\patternsbm{scale=1}{ 6 }{ 1/6, 2/4, 3/2, 4/5, 5/3, 6/1 }{1/2, 1/4, 1/5, 2/2, 2/3, 2/4, 2/5, 3/2, 3/3, 3/4, 4/1, 4/2, 4/3, 4/4, 5/1, 5/2}{}
\end{equation*}
If the box $(3,1)$ is non-empty, we can staircase it and produce a member of family $A$.
If this box is empty, but the box $(3,5)$ is non-empty, we can staircase it and produce a member of family $A$. Finally, if both of these boxes are empty we have the exceptional $E$.

\item Both of the boxes $(1,3)$ and $(3,1)$ are empty.
This produces an occurrence of the pattern below.
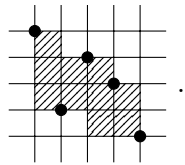
\begin{equation*} 
\patternsbm{scale=1}{ 5 }{ 1/5, 2/2, 3/4, 4/3, 5/1 }{1/2, 1/3, 1/4, 2/2, 2/3, 3/1, 3/2, 3/3, 4/1, 4/2}{}
\end{equation*}
If both of the boxes $(2,1)$ and $(2,4)$ are empty, we have a member
of family $A$. The same is true if one of the boxes are empty. If both
boxes contain elements, we staircase them (taking NE-most elements in
$(2,1)$ and SW-most elements in $(2,4)$). Two things can occur: either
the stair produced in $(2,1)$ is completely to the left of the stair
in $(2,4)$, which produces a pattern from family $A$, or the
stairs will cross at some point, in which case the first element in the stair
for $(2,4)$ that is to the left of an element in the lower stair can
be used to form a pattern from the family $A$. \qedhere

\end{enumerate}

\end{enumerate}

\end{enumerate}

\end{proof}

\begin{lemma} \label{lem:lcilem2}
\begin{enumerate}

\item \label{lem:lcilem2:part1}
If a permutation $w$ contains the pattern $35142$, then it contains at least one of the mesh patterns in family $A$ or $B$ or one of the exceptionals $C$, $E$, $E^\rmi$, $F^\rmrc$, $F^\rmirc$, $G$, or $H$.

\item
If a permutation $w$ contains the pattern $42513$, then it contains at least one of the mesh patterns in family $A$ or $B$ or one of the exceptionals $C$, $E$, $E^\rmi$, $F$, $F^\rmi$, $G^\rmrc$, or $H$.

\end{enumerate}

\end{lemma}

\begin{proof}
It suffices to prove part \eqref{lem:lcilem2:part1} since the other part
can be obtained by applying reverse-complement to everything. If a permutation
$w$ contains the pattern $35142$, then it will also contain the mesh pattern below.
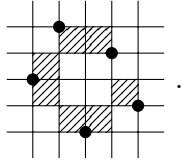
\[
\patternsbm{scale=1}{ 5 }{ 1/3, 2/5, 3/1, 4/4, 5/2 }{1/2,1/3,2/1,2/4,3/1,3/4,
4/2}{}
\]
Consider the region made up of the boxes $(2,2)$, $(2,3)$, $(3,2)$ and $(3,3)$. If this region contains more than one element, we will have an occurrence of the pattern in \eqref{lcimeshpatt1}. Then Lemma~\ref{lem:lcilem1} would imply what we are trying to prove. We can therefore assume that there is at most one element in this region.
\begin{enumerate}
\item If the region has no elements then by staircasing boxes $(1,1)$ and $(4,3)$ we produce an interval from family $B$.
\item If there is a single element in the region we need to look at four cases, depending on the box it belongs to. We can actually do just the cases where the element is in boxes $(2,2)$, $(2,3)$ and $(3,3)$ since the case of box $(3,2)$ follows from applying the inverse symmetry to the case involving box $(2,3)$.
\begin{enumerate}
\item The element is in box $(2,3)$. We get the following mesh pattern.
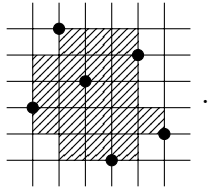
\[
\patternsbm{scale=1}{ 6 }{ 1/3, 2/6, 3/4, 4/1, 5/5, 6/2 }{1/2, 1/3, 1/4, 2/1, 2/2, 2/3, 2/4, 2/5, 3/1, 3/2, 3/3, 3/4, 3/5, 4/1, 4/2, 4/3, 4/4, 4/5, 5/2}{}
\]
If box $(5,4)$ is not empty, we get the marked mesh pattern
in~\eqref{lcimeshpatt1}.  If box $(5,3)$ or $(1,1)$ is not empty, we
get a mesh pattern from family $B$ by staircasing both boxes. If all
four boxes are empty, we get the exceptional $F^\rmrc$.
\item The element is in box $(2,2)$. We get the mesh pattern below.
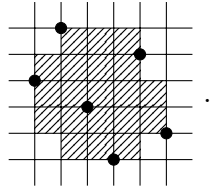
\[
\patternsbm{scale=1}{ 6 }{ 1/4, 2/6, 3/3, 4/1, 5/5, 6/2 }{1/2, 1/3, 1/4, 2/1, 2/2, 2/3, 2/4, 2/5, 3/1, 3/2, 3/3, 3/4, 3/5, 4/1, 4/2, 4/3, 4/4, 4/5, 5/2, 5/3}{}
\]
If box $(5,4)$ is not empty, we get the mesh pattern in~\eqref{lcimeshpatt1}. We can therefore assume that box
is empty. If box $(1,1)$ is not empty, and staircasing the box only adds one element to the pattern, we get
the exceptional $H$. If staircasing adds more than one element, then we get an interval in family $A$. If both boxes are empty, we get the exceptional $G$.
\item The element is in box $(3,3)$. We get the mesh pattern below.
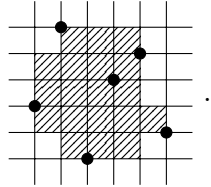
\[
\patternsbm{scale=1}{ 6 }{ 1/3, 2/6, 3/1, 4/4, 5/5, 6/2 }{1/2, 1/3, 1/4, 2/1, 2/2, 2/3, 2/4, 2/5, 3/1, 3/2, 3/3, 3/4, 3/5, 4/1, 4/2, 4/3, 4/4, 4/5, 5/2}{}
\]
If the boxes $(1,1)$ and $(5,3)$ are empty, we get an interval from family $B$. If these boxes are not empty, the elements occupying them can also be used to produce an interval from the same family. 
\qedhere
\end{enumerate}
\end{enumerate}
\end{proof}

\begin{lemma} \label{lem:lcilem3}
If a permutation $w$ contains the pattern $351624$, then
it contains at least one of the mesh patterns in family $A$ or
$B$ or one of the exceptionals $C$, $D$, $E$, $F$, $G$ or $H$ (or their symmetries under inverse and reverse-complement).
\end{lemma}

To prove this lemma, we need the following definition.
A {\bf slab permutation} is a permutation that avoids the patterns
$213, 123$, and $132$ or equivalently avoids the marked mesh pattern below.
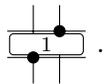
\[
\patternsbm{scale=1}{ 2 }{ 1/1, 2/2 }{}{0/1/3/2/1}
\]
A typical permutation of this sort is shown below.
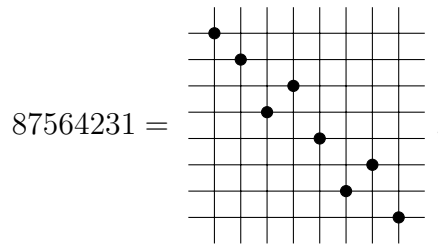
\[
87564231 = \pattern{scale=1}{ 8 }{ 1/8, 2/7, 3/5, 4/6, 5/4, 6/2, 7/3, 8/1 }{}
\]
These permutations are counted by the Fibonacci numbers, as was first shown by Simion and Schmidt~\cite[Prop.\ 15]{SimSchmi}.

\begin{proof}
If a permutation $w$ contains the pattern $351624$, then it will also
contain the following mesh pattern.
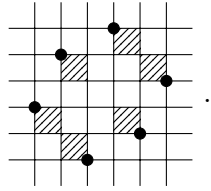
\[
\patternsbm{scale=1}{ 6 }{ 1/3, 2/5, 3/1, 4/6, 5/2, 6/4 }{1/2, 2/1, 2/4, 4/2, 4/5, 5/4}{}
\]
Let $\alpha$ be the region consisting of the boxes $(3,3)$, $(3,4)$, $(4,3)$, $(4,4)$, and
$\beta$ be the region consisting of the boxes $(2,2)$, $(2,3)$, $(3,2)$, $(3,3)$.
If $w$ contains the mesh pattern \eqref{lcimeshpatt1} we are done by Lemma \ref{lem:lcilem1}
above. We therefore assume $w$ does not contain that mesh pattern. This implies that the elements in $\alpha \cup \beta$
must form a slab permutation. If the slab lies entirely in $\alpha$ or
entirely in $\beta$, then we get a member of family $B$.  (The boxes $(1,1)$ and $(5,5)$ may need to staircased.)  Otherwise we
can assume that the slab starts in the upper left corner of $\alpha$
and ends in the lower right corner of $\beta$.  (For the other possibility
we can use the inverse symmetry.)  This will give us an occurrence of the
pattern $35142$ which implies what we want by Lemma \ref{lem:lcilem2}.
Finally $\alpha \cup \beta$ could be empty. If the boxes $(1,1)$ and $(5,5)$ are both empty, we get the exceptional $D$. If either of these boxes is non-empty, we staircase both and get an interval in family $B$.
\end{proof}

\section{Remarks and applications}
\subsection{Singularity implications with patterns}
In this section, we show that each of the pattern avoidance criteria
that determine the geometric properties we have discussed imply one
another.

We fix some notation for the patterns we will be using.
Define the {\bf smooth patterns} as
$s = 3412$, $s_c = 4231$;
the {\bf factorial patterns} as
$f = \vinc{ 4 }{ 1/3, 2/4, 3/1, 4/2 }{ 2 }$, $f_c = 4231$;
the {\bf dbi patterns} as
\begin{align*}
&d_1 = 35142,\quad  d_2 = 42513, \quad
d_3 = 351624, \\
&d_c = 4231;
\end{align*}
the {\bf lci patterns} as
\begin{align*}
&\ell_1 = 35142, \quad \ell_2 = 42513, \quad
\ell_3 = 351624, \\
&\ell_{c_1} = 53241, \quad \ell_{c_2} = 52341, \quad
\ell_{c_3} = 52431;
\end{align*}
and the {\bf Gorenstein patterns} as
\begin{align*}
g_1 &= \bivinc{ 5 }{ 1/3, 2/5, 3/1, 4/4, 5/2 }{ 2 }{ 3 }
     = 31524\, \brur{1}{5},\, \brur{2}{3}
     = \patternsbm{scale=1}{ 5 }{ 1/3, 2/5, 3/1, 4/4, 5/2 }{2/0,2/1,2/2,2/3,2/4,2/5,0/3,1/3,2/3,3/3,4/3,5/3}{},\\
g_2 &= \bivinc{ 5 }{ 1/4, 2/2, 3/5, 4/1, 5/3 }{ 3 }{ 2 }
     = 42513\, \brur{1}{5},\, \brur{3}{4}
     = \patternsbm{scale=1}{ 5 }{ 1/4, 2/2, 3/5, 4/1, 5/3 }{3/0,3/1,3/2,3/3,3/4,3/5,0/2,1/2,2/2,3/2,4/2,5/2}{}.
\end{align*}
Finally we define the two {\bf Gorenstein corner families}.  Let
\begin{align*}
 	\cG_1 
	&= \left(
	\bivinc{5}{1/5, 2/3, 3/2, 4/4, 5/1}{}{2,3,4},
	\bivinc{7}{1/7, 2/4, 3/3, 4/2, 5/6, 6/5, 7/1}{}{2,3,4,5,6},
	\bivinc{9}{1/9, 2/5, 3/4, 4/3, 5/2, 6/8, 7/7, 8/6, 9/1}{}{2,3,4,5,6,7,8}, \dotsc
	\right) \\
	&= 
	\left(
	\pattern{scale=1}{ 5 }{ 1/5, 2/3, 3/2, 4/4, 5/1 }{
	0/2,0/3,0/4,
	1/2,1/3,1/4,
	2/2,2/3,2/4,
	3/2,3/3,3/4,
	4/2,4/3,4/4,
	5/2,5/3,5/4}\,,
	\pattern{scale=1}{ 7 }{ 1/7, 2/4, 3/3, 4/2, 5/6, 6/5, 7/1 }{
	0/2,0/3,0/4,0/5,0/6,
	1/2,1/3,1/4,1/5,1/6,
	2/2,2/3,2/4,2/5,2/6,
	3/2,3/3,3/4,3/5,3/6,
	4/2,4/3,4/4,4/5,4/6,
	5/2,5/3,5/4,5/5,5/6,
	6/2,6/3,6/4,6/5,6/6,
	7/2,7/3,7/4,7/5,7/6}\,,
	\pattern{scale=1}{ 9 }{ 1/9, 2/5, 3/4, 4/3, 5/2, 6/8, 7/7, 8/6, 9/1 }{
	0/2,0/3,0/4,0/5,0/6,0/7,0/8,
	1/2,1/3,1/4,1/5,1/6,1/7,1/8,
	2/2,2/3,2/4,2/5,2/6,2/7,2/8,
	3/2,3/3,3/4,3/5,3/6,3/7,3/8,
	4/2,4/3,4/4,4/5,4/6,4/7,4/8,
	5/2,5/3,5/4,5/5,5/6,5/7,5/8,
	6/2,6/3,6/4,6/5,6/6,6/7,6/8,
	7/2,7/3,7/4,7/5,7/6,7/7,7/8,
	8/2,8/3,8/4,8/5,8/6,8/7,8/8,
	9/2,9/3,9/4,9/5,9/6,9/7,9/8}\,, \dotsc
	\right).
 \end{align*}
 The general member of this family is of the form
 \[
	\bivincs{11}{1/k, 2/\ell, 3/\cdot, 4/\cdot, 5/2,
	6/\cdot, 7/\cdot, 8/\ell, 9/{+}, 10/1, 11/1}
	{1/1, 2/2, 3/\cdot, 4/\cdot, 5/\cdot,
	6/\cdot, 7/\cdot, 8/\cdot, 9/\cdot, 10/\cdot, 11/k}
	{}{2,3,4,5,6,7,8,9,10},
 \]
 where $\ell = (k-3)/2$.  Also, let
\begin{align*}
 	\cG_2
	&= \left(
	\bivinc{5}{1/5, 2/2, 3/4, 4/3, 5/1}{}{1,2,3},
	\bivinc{7}{1/7, 2/3, 3/2, 4/6, 5/5, 6/4, 7/1}{}{1,2,3,4,5},
	\bivinc{9}{1/9, 2/4, 3/3, 4/2, 5/8, 6/7, 7/6, 8/5, 9/1}{}{1,2,3,4,5,6,7}, \dotsc
	\right) \\
	&= 
	\left(
	\pattern{scale=1}{ 5 }{ 1/5, 2/2, 3/4, 4/3, 5/1 }{
	0/1,0/2,0/3,
	1/1,1/2,1/3,
	2/1,2/2,2/3,
	3/1,3/2,3/3,
	4/1,4/2,4/3,
	5/1,5/2,5/3}\,,
	\pattern{scale=1}{ 7 }{ 1/7, 2/3, 3/2, 4/6, 5/5, 6/4, 7/1 }{
	0/1,0/2,0/3,0/4,0/5,
	1/1,1/2,1/3,1/4,1/5,
	2/1,2/2,2/3,2/4,2/5,
	3/1,3/2,3/3,3/4,3/5,
	4/1,4/2,4/3,4/4,4/5,
	5/1,5/2,5/3,5/4,5/5,
	6/1,6/2,6/3,6/4,6/5,
	7/1,7/2,7/3,7/4,7/5,}\,,
	\pattern{scale=1}{ 9 }{ 1/9, 2/4, 3/3, 4/2, 5/8, 6/7, 7/6, 8/5, 9/1 }{
	0/1,0/2,0/3,0/4,0/5,0/6,0/7,
	1/1,1/2,1/3,1/4,1/5,1/6,1/7,
	2/1,2/2,2/3,2/4,2/5,2/6,2/7,
	3/1,3/2,3/3,3/4,3/5,3/6,3/7,
	4/1,4/2,4/3,4/4,4/5,4/6,4/7,
	5/1,5/2,5/3,5/4,5/5,5/6,5/7,
	6/1,6/2,6/3,6/4,6/5,6/6,6/7,
	7/1,7/2,7/3,7/4,7/5,7/6,7/7,
	8/1,8/2,8/3,8/4,8/5,8/6,8/7,
	9/1,9/2,9/3,9/4,9/5,9/6,9/7,}\,, \dotsc
	\right).
 \end{align*}
  The general member of this family is of the form
 \[
 	\bivincs{11}{1/k, 2/\ell, 3/\cdot, 4/\cdot, 5/2,
	6/\cdot, 7/\cdot, 8/\ell, 9/{+}, 10/1, 11/1}
	{1/1, 2/2, 3/\cdot, 4/\cdot, 5/\cdot,
	6/\cdot, 7/\cdot, 8/\cdot, 9/\cdot, 10/\cdot, 11/k}
	{}{1,2,3,4,5,6,7,8,9},
 \]
  where $\ell = (k-1)/2$.

Then the Schubert variety $X_w$ is
\begin{enumerate}

\item smooth if and only if $w$ avoids the smooth patterns~\cite{LakSan},
\item factorial if and only if $w$ avoids the factorial patterns~\cite{BousBut},
\item defined by inclusions if and only if $w$ avoids the dbi patterns~\cite{GasRei},
\item a local complete intersection if and only if $w$ avoids the lci patterns (Theorem~\ref{thm:main}),
\item Gorenstein if and only if $w$ avoids the Gorenstein patterns,
and every associated Grassmannian permutation avoids every member of the corner families~\cite{WYGor}.

\end{enumerate}

It is clear that avoidance of the smooth patterns implies avoidance of the factorial
patterns.

\begin{lemma}
Avoidance of the factorial patterns implies avoidance of the dbi patterns.
\end{lemma}

\begin{proof}
Since $d_c = f_c$ it suffices to show that containment of $d_i$, for $i = 1,2,3$, implies
containment of $f$ or $f_c$. The cases $i = 1$ and $i = 2$ are similar so we only
do the former.

Assume we have an occurrence of $d_1 = 35142$ in a permutation
$w$. If the letters $5_w$ and $1_w$ are adjacent in the occurrence,
then we have an occurrence of $f$. (Here we write ``$5_w$'' instead of ``the letter in $w$ corresponding to $5$''.)
So assume they are not adjacent, and let the letter
to the right of $5_w$ be $a$. If $4_w < a$  or $a < 2_w$, we have an occurrence
of $f$. So assume that $2_w < a < 4_w$. Then the letters
$5_w$, $a$, $4_w$, $2_w$ form an occurrence of $f_c$.

Now assume we have an occurrence of $d_3 = 351624$ in a permutation
$w$. We prove this case by induction on the size of the gap between
$6_w$ and $2_w$. If these letters are adjacent we are done. If $5_w <
a$, we can move $6_w$ to $a$ and shorten the gap. If $a < 4_w$, then the
letters $3_w$, $5_w$, $6_w$, $a$, $4_w$ form an occurrence of $f$. Therefore
we can assume that $4_w < a < 5_w$. Then $3_w$, $5_w$, $1_w$, $a$,
$2_w$ form an occurrence of $d_1$, and we know that this implies
containment of either $f$ or $f_c$.
\end{proof}

Clearly we have that avoidance of the dbi patterns implies avoidance of the lci
patterns. So all that remains is:

\begin{lemma}
Avoidance of the lci patterns implies avoidance of the Gorenstein patterns
and the corner families for the associated Grassmannians.
\end{lemma}

\begin{proof}
The containment of $g_i$ implies the containment of $\ell_i$ for $i = 1,2$.
Also if a permutation contains one of the patterns in the Gorenstein corner
family $\cG_1$ then it contains $\ell_{c_1}$.
Finally if a permutation contains one of the patterns in the Gorenstein corner
family $\cG_2$, then it contains $\ell_{c_3}$.
\end{proof}

Figure~\ref{fig:hierarchy} gives another way of viewing some of the patterns above. Where possible, the classical patterns have been grouped together into a single marked mesh pattern to emphasize the underlying classical patterns $3421$ and $4231$ that characterize the smooth varieties.

\begin{figure}[ht]
\begin{center}
 \begin{tikzpicture}
 [scale = 1.25, place/.style = {circle,draw = green!50,fill = green!20,thick,minimum size = 5pt},auto]
 
 \node[] at (0,8) (sm)   {smooth};
 \node[] at (4,8) (sm1) {\patternsbm{scale=0.9}{ 4 }{ 1/3, 2/4, 3/1, 4/2 }{}{}};
 \node[] at (8,8) (sm2) {\patternsbm{scale=0.9}{ 4 }{ 1/4, 2/2, 3/3, 4/1 }{}{}};
 
 \node[] at (0,6) (fac) {factorial};
 \node[] at (4,6) (fac1) {\patternsbm{scale=0.9}{ 4 }{ 1/3, 2/4, 3/1, 4/2 }{2/0, 2/1, 2/2, 2/3, 2/4}{}};
 \node[] at (8,6) (fac2) {\patternsbm{scale=0.9}{ 4 }{ 1/4, 2/2, 3/3, 4/1 }{}{}};
 
 \node[] at (0,4) (dbi) {def.\ by incl.\ };
 \node[] at (3,4) (dbi1) {\patternsbm{scale=0.9}{ 4 }{ 1/3, 2/4, 3/1, 4/2 }{}{1/1/2/2/{},3/3/4/4/\scriptscriptstyle{1}}};
 \node[] at (5,4) (dbi1') {\patternsbm{scale=0.9}{ 4 }{ 1/3, 2/4, 3/1, 4/2 }{}{0/1/1/2/\scriptscriptstyle{1},1/0/2/1/\scriptscriptstyle{1}}};
 \node[] at (8,4) (dbi2) {\patternsbm{scale=0.9}{ 4 }{ 1/4, 2/2, 3/3, 4/1 }{}{}};
 
 \node[] at (0,2) (lci) {local compl.\ inters.\ };
 \node[] at (3,2) (lci1) {\patternsbm{scale=0.9}{ 4 }{ 1/3, 2/4, 3/1, 4/2 }{}{1/1/2/2/{},3/3/4/4/\scriptscriptstyle{1}}};
 \node[] at (5,2) (lci1') {\patternsbm{scale=0.9}{ 4 }{ 1/3, 2/4, 3/1, 4/2 }{}{0/1/1/2/\scriptscriptstyle{1},1/0/2/1/\scriptscriptstyle{1}}};
 \node[] at (8,2) (lci2) {\patternsbm{scale=0.9}{ 4 }{ 1/4, 2/2, 3/3, 4/1 }{}{1/2/4/3/1}};
 
 \node[] at (0,0) (gor) {Gorenstein};
 \node[] at (4,0) (gor1) {\patternsbm{scale=0.9}{ 4 }{ 1/3, 2/4, 3/1, 4/2 }{2/0, 2/1, 2/2, 2/3, 2/4, 0/2, 1/2, 3/2, 4/2}{1/1/2/2/{},3/3/4/4/\scriptscriptstyle{1}}};
 \node[] at (8,0) (gor2) {two inf.\ families of patterns};

 \draw[->] (sm) to (fac);
 \draw[->] (fac) to (dbi);
 \draw[->] (dbi) to (lci);
 \draw[->] (lci) to (gor);
 
 \draw[->] (sm1) to (fac1);
 \draw[->] (dbi1) to (lci1);
 \draw[->] (dbi1') to (lci1');
 \draw[->] (lci1) to (gor1);
 
 \draw[->] (sm2) to (fac2);
 \draw[->] (fac2) to (dbi2);
 \draw[->] (dbi2) to (lci2);
 \draw[->] (lci2) to (gor2);
 
 \end{tikzpicture}
  \caption{Properties of Schubert varieties described with pattern avoidance. An arrow between two patterns means that avoidance of the first pattern implies avoidance of the second}
 \label{fig:hierarchy}
 \end{center}
\end{figure}
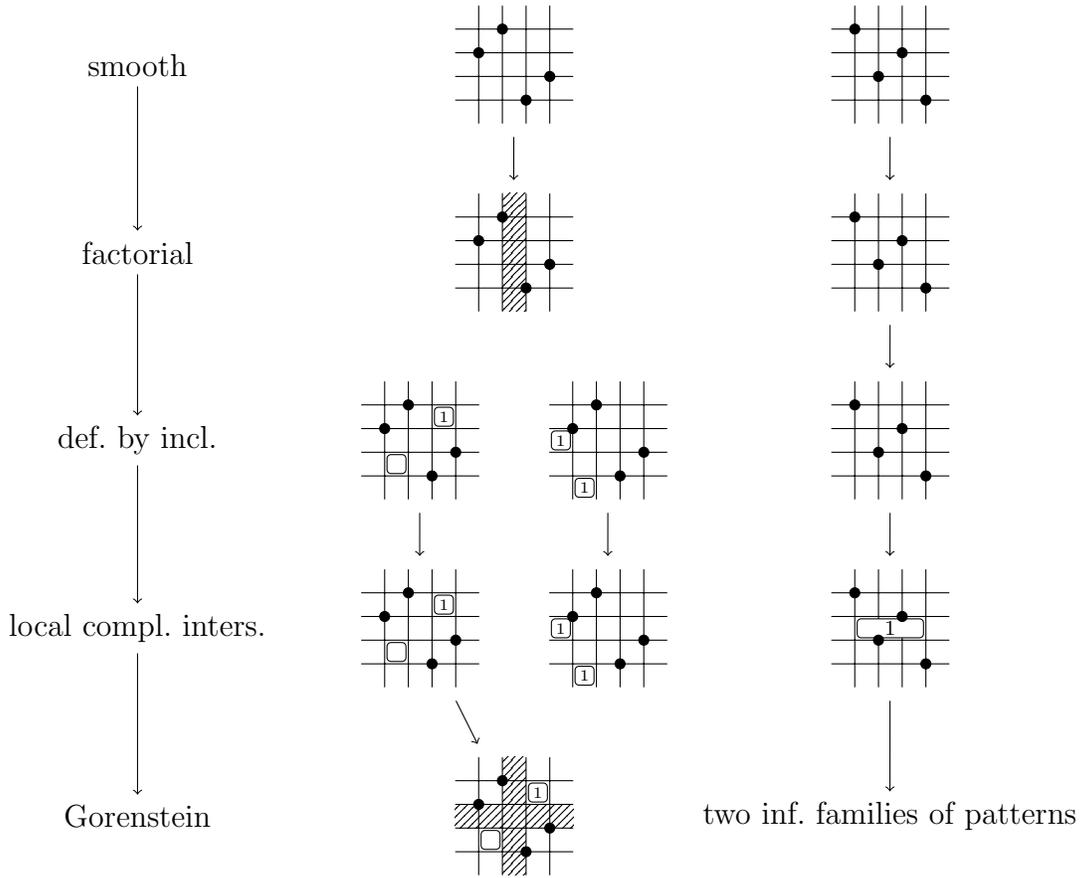

Since we are emphasizing the underlying patterns $3421$ and $4231$ it is worth noting that the first pattern is present in two other known properties of permutations. Billey and Warrington~\cite{BilWar01} introduced $321$-hexagon avoiding permutations as those permutations that avoid $321$ as well as four classical patterns from $S_8$. Alternatively, these can be characterized as avoiding $321$ and the pattern below.
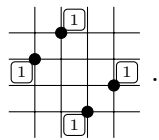
\[
\patternsbm{scale=1}{ 4 }{ 1/3, 2/4, 3/1, 4/2 }{}{2/4/3/5/\scriptscriptstyle{1}, 0/2/1/3/\scriptscriptstyle{1}, 4/2/5/3/\scriptscriptstyle{1}, 2/0/3/1/\scriptscriptstyle{1}}
\]
Tenner~\cite{T11} studies permutations with the property that the number of repeated letters in the reduced decomposition equals the number of occurrences of $321$ and $3412$. She shows that these are exactly the permutations avoiding $4321$ and $9$ classical patterns from $S_5$. It is easy to verify that these are the permutations that avoid $4321$ and the following pattern.
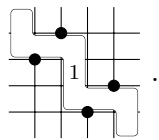
\[
\patternsbmTenner{scale=1}{ 4 }{1/3,2/4,3/1,4/2}{}{}{2/2/3/3/1}
\]

\subsection{Lci matrix Schubert varieties}
Let $\pi: \mathrm{GL}_n \rightarrow G/B$ be the natural quotient map, and let
$i: \mathrm{GL}_n \rightarrow M_n$ be the inclusion of $\mathrm{GL}_n$ into the affine
space of $n\times n$ matrices.  The {\bf matrix Schubert variety}
$Y_w$ is the closure
\[
Y_w:=\overline{i(\pi^{-1}(X_{w_0w}))};
\]
it was
introduced by Fulton in~\cite{FulMSV}.

As Fulton notes, $Y_w$ is also the Kazhdan--Lusztig variety
$\mathcal{N}_{v_n,\widetilde{w}}$, where $v_n\in S_{2n}$ is the
permutation defined by $v(i)=n+1-i$ if $i\leq n$ and $v(i)=n+(2n+1-i)$
if $i>n$ and $\widetilde{w}$ is defined by
$\widetilde{w}(i)=n+w_0w(i)$ if $i\leq n$ and
$\widetilde{w}(i)=2n+1-i$ if $i>n$.  Furthermore, looking at the
matrices, it is clear that
$\mathcal{N}_{v_n,\widetilde{w}}\times\mathbb{C}^{n(n-1)}\cong\mathcal{N}_{\id,\widetilde{w}},$
since the generalized Pl\"ucker coordinates defining $I_{\widetilde{w}}$
do not involve any of the additional variables found in $M^{(\id)}$ but not
in $M^{(v_n)}$.  Therefore, $Y_w$ is lci if and only if
$X_{\widetilde{w}}$ is.  As a consequence, we have the following corollary.

\begin{corollary}
The matrix Schubert variety $Y_w$ is a local complete intersection if
and only if $w$ avoids $1342$, $1432$, $1423$, $31524$, $24153$, and $351624$.
\end{corollary}

One can also reformulate this statement in terms of the diagram of
$w$.  Doing so recovers a theorem of Jen-Chieh Hsiao~\cite[Theorem 5.2]{Hsiao}.

\subsection{Local $K$-theory and cohomology classes at the identity}

Given a smooth variety $X$ with an action of an algebraic torus
$T=(\mathbb{C}^*)^n$ satisfying certain conditions (including that the
fixed points are isolated), Goresky, Kottwicz, and MacPherson~\cite{GKM} show
that the map on equivariant cohomology
\[
H^*_T(X)\rightarrow \bigoplus_{p\in X^T} H^*_T(p)
\]
induced by the
inclusion of the fixed points into $X$ is an injection and describe
its image.  This result was extended by Knutson and
Rosu~\cite[Thm.~A.5]{KnuRosu} to $K$-theory and the map
\[
K^*_T(X)\rightarrow \bigoplus_{p\in X^T} K^*_T(p).
\]
(All cohomology and $K$-theory will be with $\mathbb{Q}$
coefficients.)  Since $H^*_T(p)$ is isomorphic to
$\mathbb{Q}[t_1,\ldots,t_n]$ and $K^*_T(p)$ is isomorphic to
$\mathbb{Q}[t_1^\pm,\ldots,t_n^\pm]$, this theory provides in many
cases an easier method for calculating in the cohomology or $K$-theory
ring.  Indeed, for the case of the flag variety $G/B$, this theory was
implicitly anticipated by work of Kostant and Kumar~\cite{KosKumCoh,
  KosKumK}, who furthermore gave recursive formulas in
$\bigoplus_{p\in X^T} H^*_T(p)$ and $\bigoplus_{p\in X^T} K^*_T(p)$
for sets of elements of $H^*_T(G/B)$ and $K^*_T(G/B)$ which form bases
for these rings as free $H^*_T(p)$- and $K^*_T(p)$-modules.
Kumar~\cite{KumSmooth} later showed that these basis elements are
actually the classes of Schubert varieties.

The $T$-fixed points of $G/B$ are the Schubert points $e_u$ for all
$u\in S_n$.  Hence, the classes $[\mathcal{O}_{X_w}]\mid_{e_u}$ (both
in $H^*_T(e_u)$ and $K^*_T(e_u)$) determine the cohomology and
$K$-theory class of the Schubert variety $X_w$.  In addition, these
classes encode significant geometric information about $X_w$.  Our aim
in this section is to describe the consequences of our explicit
equations for lci Schubert varieties on the class
$[\mathcal{O}_{X_w}]\mid_{e_\id}$ in the case where $X_w$ is lci.  We
recover in part half of a theorem of Kumar in the case where $X_w$ is
smooth.  We will need to use the algebraic machinery of
$K$-polynomials and multidegrees, which we briefly describe here but
are described in greater detail in~\cite[Chap. 8]{MilSturm}
and~\cite[Section 2.3]{KnuMil}.

Let $\mathbf{a}_1,\ldots,\mathbf{a}_n$ denote an integral basis for the weight lattice
of $T=(\mathbb{C}^*)^n$, which we identify with $\mathbb{Z}^n$.  An
action of $T$ on a polynomial ring $S=\mathbb{C}[x_1,\ldots,x_k]$
(assuming the $x_i$ are eigenvectors for the action) corresponds to
the grading assigning each variable $x_i$ a degree
$\mathbf{\lambda}^{(i)}=\sum_{j=1}^n
\lambda^{(i)}_j\mathbf{a}_j\in\mathbb{Z}^n$, where $\mathbf{\lambda}^{(i)}$ is
the weight of the action of $T$ on $x_i$.  In the case where our
grading is positive, meaning that the $\mathbf{\lambda}^{(i)}$
generate a pointed cone in $\mathbb{Z}^n$, any finitely generated
graded $S$-module $M$ has a {\bf Hilbert series} (also known as a formal
$T$-character)
\[
\mathcal{H}(M)=\sum_{\mathbf{\lambda}\in\mathbb{Z}^n} \dim M_\lambda \mathbf{t}^{\mathbf{\lambda}},
\]
where the
sum is over all weights of $T$, $M_\lambda$ is the weight space for
$\lambda$, and $\mathbf{t}^{\mathbf{\lambda}}$ denotes $\prod_{j=1}^n t_j^{\lambda_j}$ where $\mathbf{\lambda}=\sum_{j=1}^n \lambda_j\mathbf{a}_j$.

The {\bf $K$-polynomial} of $M$ can be defined as
\[
\mathcal{K}(M)=\mathcal{H}(M)\prod_{i=1}^k
(1-\mathbf{t}^{\mathbf{\lambda}^{(i)}}).
\]
Given a finite
$\mathbb{Z}^n$ graded free resolution
\[
0\leftarrow M\leftarrow E_0\leftarrow \cdots\leftarrow E_L\leftarrow 0
\]
of $M$ with
\[
E_k=\bigoplus_j S(-\mathbf{\lambda}^{(j,k)}),
\]
the $K$-polynomial satisfies
\[
\mathcal{K}(M)=\sum_{k=1}^L (-1)^k \sum_j \mathbf{\lambda}^{(j,k)},
\]
so
$\mathcal{K}(M)$ is a representative for the class of $M$ in
$K^*_T(\Spec S)$.

Define the {\bf multidegree} of $M$, denoted $\mathcal{C}(M)$, as the
sum of the lowest degree terms of
$\mathcal{K}(M,\mathbf{1}-\mathbf{t})$.  (This means we substitute
$1-t_i$ for $t_i$ for each $i\in\dbrac{1,n}$.)  Taking the lowest
degree terms in this way is, up to a sign change which conveniently
agrees with a difference between the usual conventions for
Grothendieck and Schubert polynomials, the same as taking the Chern
map from $K$-theory to cohomology, so $\mathcal{C}(M)$ can be regarded
as a representative for the class of $M$ in $H^*_T(\Spec S)$.

Given a graded complete intersection $S/I$, where $I$ is generated by
$f_1,\ldots,f_L$,
the $K$-polynomial
$\mathcal{K}(S/I)$ and multidegree $\mathcal{C}(S/I)$ are easily seen
to be
\[
\mathcal{K}(S/I)=\prod_{i=1}^L (1-\mathbf{t}^{\deg f_i})
\]
and hence
\[
\mathcal{C}(S/I)=\prod_{i=1}^L \left(\sum_{j=1}^n\langle \mathbf{a}_j, \deg f_i\rangle t_j\right),
\]
since the Koszul resolution is a free resolution of $S/I$.

Now we consider specifically the case of $T$-invariant subvarieties of
$G/B$, and more specifically $\Omega_{\id}$.  In this case,
$S=S_{\id}$, and since $T$ under its usual action on $G/B$ acts on the
matrix entry at $(i,j)$ with weight $\mathbf{a}_i-\mathbf{a}_j$ (where $\mathbf{a}_i$ denotes as usual
the homomorphism from $T$ to $C^*$ picking
out the $i$-th diagonal entry), it acts on the variable $z_{i,j}$, which is the coordinate function for the matrix entry at $(i,j)$, with
weight $\mathbf{a}_j-\mathbf{a}_i$.  Moreover,
$\Spec S$ equivariantly retracts onto $e_\id$, so we can identify
classes in $H^*_T(\Spec S)$ and $K^*_T(\Spec S)$ respectively with
classes in $H^*_T(e_\id)$ and $K^*_T(e_\id)$.

Hence we can identify $[\mathcal{O}_{X_w}]\mid_{e_\id}$ with
$\mathcal{C}(S/I_w)$ (for cohomology) or $\mathcal{K}(S/I_w)$ (for
$K$-theory).  Furthermore, it is a well-known folklore theorem
(see~\cite{Goldin,WYGrob} for proofs for which none of the authors claim
originality) that, in the coordinates for the weight space we use,
\[
[\mathcal{O}_{X_w}]\mid_{e_v}=\Schub_{w_0w}(t_{v(1)},\ldots,t_{v(n)};t_n,\ldots,t_1)
\]
in cohomology, and
\[
[\mathcal{O}_{X_w}]\mid_{e_v}=\Groth_{w_0w}(t_{v(1)},\ldots,t_{v(n)};t_n,\ldots,t_1)
\]
in $K$-theory, where $\Schub_w$ and $\Groth_w$ are the double Schubert and
Grothendieck polynomials of Lascoux and Sch\"utzenberger~\cite{LSSchub,LSGroth}.  Hence we
can and will state our results purely as an identity between
polynomials.

Since we have explicit generators for the ideal defining
$\mathcal{N}_{\id,w}$ whenever $X_w$ is lci, we can calculate their
degrees to obtain explicit formulas for the local $K$-theory and
cohomology classes of $X_w$ at the identity.

\begin{proposition}
Suppose $X_v$ is defined by inclusions, $(x,y)\in D(v)$, $r=r_v(x,y)$,
and $f_{(x,y)}$ is the generator of $I_v$ defined in
Section~\ref{sect:dbisufficiency}.  Then $\deg
f_{(x,y)}=\mathbf{a}_{x+r}-\mathbf{a}_{y-r}$ (regardless of the choices made in defining
$f_{(x,y)}$).
\end{proposition}

\begin{proof}
In our grading, $z_{i,j}$ has degree $\mathbf{a}_j-\mathbf{a}_i$.  Hence the degree of
$d_{A,B}$ is $\sum_{j\in B} \mathbf{a}_j - \sum_{i\in A} \mathbf{a}_i$.

Given $(x,y)\in D(v)$, $f_{(x,y)}$ is defined to be $d_{A(x,y),B(x,y)}$ where
\[
A(x,y)=\dbrac{p,p+r-1}\cup \{x+r\},
\]
\[
B(x,y)=\{y-r\}\cup\dbrac{q-r+1,q},
\]
and $(p,q)$ is some essential set
box NE of $(x,y)$ and in the same connected component of the diagram.
At first glance it appears that the degree of $f_{(x,y)}$ depends on
the choice of $(p,q)$, but Lemma~\ref{lem:dbilocation} tells us that
(except in the case where $r=0$ and this is irrelevant), $r=q-p+1$, so
$p=q-r+1$ and $p+r-1=q$.  Hence,
\[
\deg f_{(x,y)}=\mathbf{a}_{y-r}+\sum_{i=0}^{r-1} \mathbf{a}_{q-i}-\mathbf{a}_{x+r}-\sum_{i=0}^{r-1} \mathbf{a}_{p+i} = \mathbf{a}_{y-r}-\mathbf{a}_{x+r}. \qedhere
\]
\end{proof}

\begin{proposition}
Suppose $X_w$ is lci, $(p,q)\in E^{\prime\prime}(w)$, $r=r_w(p,q)$,
and $f_{(p,q)}$ is as defined in Section~\ref{sect:nondbisufficiency}.
Then
\[
\deg f_{(p,q)}=\sum_{i=0}^{r} (\mathbf{a}_{q-r}-\mathbf{a}_{p+r}).
\]
\end{proposition}

\begin{proof}
The polynomial $f_{(p,q)}$ is defined as $d_{A^\prime(p,q), B^\prime(p,q)}$ where
\[
A^\prime(p,q)=\dbrac{p,p+r}
\]
and
\[
B^\prime(p,q)=\dbrac{q-r,q}.
\]
The proposition follows immediately.
\end{proof}

The corollaries follow immediately from the propositions and the
discussion above.

\begin{corollary}
\label{cor:dbilocalschub}
Suppose $X_v$ is defined by inclusions.  Then
\[
\mathcal{K}(S/I_v)=\Groth_{w_0v}(t_1,\ldots,t_n;t_n,\ldots,t_1)=\prod_{(x,y)\in D(v)} (1-t_{y-r_v(x,y)}/t_{x+r_v(x,y)}),
\]
and
\[
\mathcal{C}(S/I_v)=\Schub_{w_0v}(t_1,\ldots,t_n;t_n,\ldots,t_1)=\prod_{(x,y)\in D(v)} (t_{y-r_v(x,y)}-t_{x+r_v(x,y)}).
\]
\end{corollary}

\begin{corollary}
\label{cor:lcilocalschub}
Suppose $X_w$ is lci, and let $v$ be the permutation defined by
inclusions associated to $w$ by Theorem~\ref{thm:dbifromadbi}.  Then
\begin{align*}
\mathcal{K}(S/I_w)&=\Groth_{w_0w}(t_1,\ldots,t_n;t_n,\ldots,t_1)\\
&=\Groth_{w_0v}(t_1,\ldots,t_n;t_n,\ldots,t_1)\prod_{(p,q)\in E^{\prime\prime}(w)} \left(1-\prod_{i=0}^{r_w(p,q)} t_{q-i}/t_{p+i}\right),
\end{align*}
and
\begin{align*}
\mathcal{C}(S/I_w)&=\Schub_{w_0w}(t_1,\ldots,t_n;t_n,\ldots,t_1)\\
&=\Schub_{w_0v}(t_1,\ldots,t_n;t_n,\ldots,t_1)\prod_{(p,q)\in E^{\prime\prime}(w)} \left(\sum_{i=0}^{r_w(p,q)} t_{q-i}-t_{p+i}\right).
\end{align*}
\end{corollary}

For $j$ and $i$ with $1\leq j<i\leq n$, let $s_{ji}\in S_n$ be the
transposition switching $j$ and $i$.  For the case where $X_v$ is
smooth, the following is a theorem of Kumar~\cite{KumSmooth}, restated
in our language. (The equivalence of our statement with the original
is unfortunately folklore; parts of the connection can be found
in~\cite{BilLak,Goldin}.)
\begin{theorem}
\label{thm:kumar}
The following are equivalent:
\begin{enumerate}
\item $X_v$ is smooth.
\item
\[
\mathcal{K}(S/I_v)=\Groth_{w_0v}(t_1,\ldots,t_n;t_n,\ldots,t_1)=\prod_{(i,j): s_{ji}\not\leq v} (1-t_j/t_i).
\]
\item
\[
\mathcal{C}(S/I_v)=\Schub_{w_0v}(t_1,\ldots,t_n;t_n,\ldots,t_1)=\prod_{(i,j): s_{ji}\not\leq v} (t_j-t_i).
\]
\end{enumerate}
\end{theorem}

Comparing Theorem~\ref{thm:kumar} and
Corollary~\ref{cor:dbilocalschub} tells us (because
$\mathbb{Q}[t_1,\ldots,t_n]$ is a unique factorization domain) that,
in the case where $X_v$ is smooth, the map
\begin{align*}
\phi: D(v) \rightarrow \{(i,j)\mid s_{ji}\not\leq v\} \qquad
(x,y) \mapsto (x+r_v(x,y),y-r_v(x,y))
\end{align*}
is a bijection.  Indeed, $D(v)$ always has $\binom{n}{2}-\ell(v)$
elements, and a theorem of Carrell~\cite{CarRatSm} states that
$\{(i,j)\mid s_{ji}\not\leq v\}$ has $\binom{n}{2}-\ell(v)$ elements
whenever $X_v$ is smooth.  We believe a purely combinatorial proof can
be given that $\phi$ is a bijection whenever $v$ avoids $4231$ and
$3412$.

In addition, the map $\phi$ can be defined for any $v$ if we allow the
codomain to include all transpositions (or equivalently all positive
roots).  We conjecture the image of $\phi$ always contains
$\{(i,j)\mid s_{ji}\not\leq v\}$ and equals this set precisely when
$v$ is defined by inclusions.

These considerations may help in finding a connection between the
results of Gasharov and Reiner and those of Hultman, Linusson,
Shareshian, and Sj\"ostrand.

\subsection{Cohomology rings of lci Schubert varieties}

Let $Z\subseteq\Omega^\circ_\id$ be the scheme theoretic vanishing
locus of a principal nilpotent vector field on $G/B$.  In our
coordinates on $\Omega^\circ_\id$, the scheme $Z$ is defined by the ideal $\langle a_{i,j}\rangle_{1\leq j<i\leq n}\subseteq S$, where
\[
a_{i,j}=z_{i+1,j}-z_{i,j-1}+z_{i,j}(z_{j,j-1}-z_{j+1,j}).
\]
Aky{\i}ld{\i}z, Lascoux, and Pragacz~\cite{ALP} show that, for any Schubert variety
$X_w$, the cohomology ring $H^*(X_w,\mathbb{C})$ is actually the
coordinate ring of the scheme-theoretic intersection $X_w\cap Z.$

In earlier work leading to this theorem, Aky{\i}ld{\i}z and
Aky{\i}ldiz~\cite{Aky2} show that the isomorphism
$\phi:\mathbb{C}[Z]\rightarrow H^*(G/B)$ is explicitly given by
\[
\phi(z_{i,j}) = h_{i-j}(x_1,\ldots,x_j),
\]
where $h_{i,j}$ denotes
the complete homogeneous symmetric function in the given variables, and
$x_i$ is explicitly the first Chern class $c_1(L^\vee_i)$ of the
$i$-th dual tautological line bundle.  Moreover, Aky{\i}ld{\i}z, Lascoux,
and Pragacz~\cite{ALP} show that
\[
H^*(X_w)=H^*(G/B)/\phi(I_w)
\]
not only as abstract rings but explicitly as  $H^*(G/B)$-modules
under the interpretation $x_i=c_1(L^\vee_i)$.

Gasharov and Reiner~\cite{GasRei} gave a presentation for the cohomology
ring $H^*(X_v)$ (and indeed, its projection onto any partial flag
variety) in the case where $X_v$ is defined by inclusions.  Given
$(p,q)\in E(v)$ with $r_v(p,q)=0$, let
\[
K_{(p,q)}=\langle e_k(x_{q+1},\ldots,x_n)\rangle_{k=p-q}^{n-q},
\]
where $e_k$ denotes the $k$-th elementary symmetric function in the
given variables.  Given $(p,q)\in E^\prime(w)$, let
\[
K_{(p,q)}=\langle e_k(x_1,\ldots,x_q)\rangle_{k=q-p+2}^{n-q}.
\]
Then Gasharov and Reiner show the following~\cite[Thm. 3.1]{GasRei}.

\begin{theorem}
Let $v$ be defined by inclusions.  Then
\[
H^*(X_v)=H^*(G/B)/\sum_{(p,q)\in E(v)} K_{(p,q)}
\]
as an $H^*(G/B)$-module where $x_i=c_1(L^\vee_i)$.
\end{theorem}

Indeed, the original aim of Gasharov and Reiner was to investigate
$H^*(X_v)$ in the case where $X_v$ is smooth, and they defined the
class of Schubert varieties defined by inclusions because it is the
more general class for which their formula holds.  Using our explicit
generators for $I_v$, one can calculate explicit generators for
$\phi(I_v)$ and hence a presentation for $H^*(X_v)$.  The presentation
we obtain is different from that of Gasharov and Reiner; recovering
their result from ours via the theorem of Aky{\i}ld{\i}z, Lascoux, and
Pragacz requires some use of determinantal and symmetric function
identities.

In general, their presentation (which is not always minimal) can
require fewer than $\binom{n}{2}-\ell(w)$ generators, so this shows
$I_v$ can in some situations require more generators than $\phi(I_v)$.
(Indeed, this is obvious from considering the one dimensional Schubert
varieties $X_{s_i}$ where $s_i$ is the permutation switching $i$ and
$i+1$.)

In the more general case where $X_w$ is lci, one can obtain a
presentation for $H^*(X_w)$ by using the theorem of Gasharov and
Reiner along with the following proposition.

\begin{proposition}
Suppose $X_w$ is lci, $(p,q)\in E^{\prime\prime}(w)$, and
$r=r_w(p,q)$.  Then
\[
\phi(f_{(p,q)})=\phi(d_{A^\prime(p,q),B^\prime(p,q)})=s_{(p-q+r)^{r+1}}(x_1,\ldots,x_q),
\]
  where $s_{(p-q+r)^{r+1}}$ is the Schur function (in the given
  variables) corresponding to the rectangular partition with $r+1$
  parts of size $p-q+r$.
\end{proposition}

\begin{proof}
Note that $A^\prime(p,q)=\dbrac{p,p+r}$ and
$B^\prime(p,q)=\dbrac{q-r,q}$.
Hence,
\[
\phi(d_{A^\prime(p,q),B^\prime(p,q)})=
\begin{vmatrix}
h_{p-q+r}(x_1,\ldots,x_{q-r}) & \cdots & h_{p-q}(x_1,\ldots,x_q) \\
\vdots & & \vdots \\
h_{p+r-q+r}(x_1,\ldots,x_{q-r}) & \cdots & h_{p+r-q}(x_1,\ldots,x_q)
\end{vmatrix}.
\]
As noted by Aky{\i}ld{\i}z and Aky{\i}ldiz~\cite{Aky2}, one can repeatedly
apply column operations using the symmetric function identity
\[
h_a(x_1,\ldots,x_b)=h_a(x_1,\ldots,x_{b-1})+x_bh_{a-1}(x_1,\ldots,x_b)
\]
to show that
\[
\phi(d_{A^\prime(p,q),B^\prime(p,q)})=
\begin{vmatrix}
h_{p-q+r}(x_1,\ldots,x_q) & \cdots & h_{p-q}(x_1,\ldots,x_q) \\
\vdots & & \vdots \\
h_{p+r-q+r}(x_1,\ldots,x_q) & \cdots & h_{p+r-q}(x_1,\ldots,x_q)
\end{vmatrix}.
\]
The proposition then follows immediately from the Jacobi--Trudi
identity, a standard identity in the theory of symmetric functions.
\end{proof}

We obtain the following corollary from the proposition and the
discussion above.

\begin{corollary}
\label{cor:lcicohom}
Suppose $X_w$ is lci, and let $v$ be the permutation defined by inclusions
associated to $w$ by Theorem~\ref{thm:dbifromadbi}.  Then
\[
H^*(X_w)=H^*(X_v)/\langle s_{(p-q+r_w(p,q))^{r_w(p,q)+1}}(x_1,\ldots,x_q)\rangle_{(p,q)\in E^{\prime\prime}(w)}
\]
as an $H^*(G/B)$-module where $x_i=c_1(L^\vee_i)$.
\end{corollary}

Using Theorem~\ref{thm:dbidiagram} and Corollary~\ref{cor:lcicohom}, one
can obtain an explicit presentation of the cohomology ring $H^*(X_w)$
whenever $X_w$ is lci.

Reiner, Yong, and the second author~\cite{RWY} gave another
presentation of the cohomology ring of a general Schubert variety
$X_w$.  They associate an ideal $L_{(p,q)}$ to each element $(p,q)\in
E(w)$.  Each ideal $L_{(p,q)}$ is generated by Schur functions,
and $$H^*(X_w)\cong H^*(G/B)/\sum_{(p,q)\in E(w)} L_{(p,q)}.$$
Corollary~\ref{cor:lcicohom} shows that not all generators of each
$L_{(p,q)}$ are necessary; indeed only the unique generator of
smallest degree in $L_{(p,q)}$ suffices when $(p,q)\in
E^{\prime\prime}(w)$.

\section{Questions}
We conclude with a list of questions for future research.  We begin
with two purely combinatorial problems.

\begin{problem}
Enumerate the permutations $w\in S_n$ for which $X_w$ is lci.  An
ideal answer would provide an explicit generating function.
\end{problem}

For smooth Schubert varieties, the analogous question was answered in
unpublished work of Haiman~\cite{HaiSchubPre}.  (A proof of this
formula appears in~\cite{BousBut}.)  Bousquet-M\'elou and
Butler~\cite{BousBut} gave a generating function for the number of
factorial Schubert varieties.  For Schubert varieties defined by
inclusions, Albert and Brignall~\cite{AlbBrig} have given a generating
function.  On the other hand, the analogous question for Gorenstein
Schubert varieties is still open.

This question has motivation beyond mere curiosity.  The generating
function for the number of smooth Schubert varieties reflects the
earlier theorem of Ryan~\cite{Ryan} that all smooth Schubert varieties are
iterated Grassmannian bundles.  One might hope that a generating
function for the number of lci Schubert varieties could lead to a
similar structure theorem.

Answering the following more specific combinatorial question may help
in deriving the generating function enumerating lci Schubert varieties
from the generating function for Schubert varieties
defined by inclusions.

\begin{problem}
\label{prob:dbifromadbiconv}
Determine if the converse to Theorem~\ref{thm:dbifromadbi} is true.
More precisely, suppose $w$ is a permutation with essential set
$E(w)$, and suppose $E^{\prime\prime}(w)\subseteq E(w)$ is the set of
essential set boxes that are not defined by inclusions.  If
$E(w)\setminus E^{\prime\prime}(w)$ is the essential set for some
permutation $v$ (necessarily defined by inclusions) such that $r_v(p,q)=r_w(p,q)$ for all $(p,q)\in
E(w)\setminus E^{\prime\prime}(w)$ and
$\ell(v)-\ell(w)=\#E^{\prime\prime}(w)$, then is $w$ necessarily
almost defined by inclusions (or equivalently lci)?
\end{problem}

We now proceed to questions of a more geometric nature.

\begin{problem}
Find a geometric explanation for the appearance of ordinary pattern
avoidance in our theorem.
\end{problem}

A priori, one would expect that ordinary pattern avoidance is not
sufficient for characterizing the Schubert varieties that are lci.
Indeed, the weaker condition of being factorial~\cite{BousBut} and of
being Gorenstein~\cite{WYGor} cannot be characterized by classical pattern
avoidance since there exist examples where $X_v$ is not factorial
(respectively not Gorenstein), $w$ pattern contains $v$, and $X_w$ is
factorial (respectively Gorenstein).  Instead some more general form of pattern
avoidance is required in the statement of those theorems.

On the other hand, the pattern map of Billey and Braden~\cite{BilBra}
(which was also described by Bergeron and Sottile~\cite{BerSot}) gives
a geometric explanation of why the smooth Schubert varieties can be
characterized by pattern avoidance.  Their explanation relies on the
smoothness of the $T$-fixed locus (with the reduced scheme structure) of any smooth variety, which was originally shown (in fact for
any linearly reductive group $T$ over a field of characteristic 0) by
Fogarty and Norman~\cite{FogNorm}.  To use the explanation of Billey
and Braden, we would need that the $T$-fixed locus of any lci variety
is lci.  It seems to be not known if this is true in general.

\begin{problem}
Determine the lci locus for any Schubert variety.
\end{problem}

We conjecture that the intervals given in
Proposition~\ref{prop:nonlciinterv} determine precisely the non-lci
locus in the sense of~\cite[Thm. 2.6]{WYGov}.

\begin{conjecture}
The Schubert variety $X_w$ is not lci at $e_x$ if and only if there
is a permutation $x^\prime\geq x$ (in Bruhat order) such that $[u,v]$
interval embeds in $[x^\prime,w]$ for some $[u,v]$ in Family A, in
Family B, or listed in Figure~\ref{fig:except_ints}.
\end{conjecture}

Proving this conjecture would answer this question.

The smooth locus was characterized independently by Billey and
Warrington~\cite{BilWar}, Cortez~\cite{Cor}, Kassel, Lascoux, and
Reutenauer~\cite{KLR}, and Manivel~\cite{ManLocus} following earlier
work of Gasharov~\cite{Gash}.  The proofs in~\cite{BilWar, KLR, Man}
are similar and use a criterion on the Bruhat graph due to
Carrell~\cite{CarRatSm}, while~\cite{Cor} uses more geometric methods
depending on partial resolutions of singularities.
Perrin~\cite{PerGor} used similar geometric methods to characterize
the Gorenstein locus in the specific case of Schubert varieties coming
from a (co)minscule parabolic, partially resolving~\cite[Conj. 6.7]{WYGov}.

We have only been able to check our conjecture completely up to $n=6$,
but this conjecture is equivalent to a parsimony property satisfied by
the characterization of the singular locus and the conjectured
characterization of the non-Gorenstein locus.  To be precise, if $I$
is the minimal set of intervals specifying the singular locus (so that
no element of $S$ interval embeds in another) and $[u,v]\in I$, then
$X_v$ is singular only at $e_u$.  Our conjecture is equivalent to the
analogous statement where the word ``singular'' is replaced by
``non-lci'' throughout.

\begin{problem}
Characterize the lci Schubert varieties for the other simple Lie groups.
\end{problem}

In our view, the ideal answer to this question would be a
characterization in terms of the definition of pattern avoidance via
root subsystems given by Billey and Postnikov~\cite{BilPos}.  However,
a criterion in terms of the Bruhat graph would be of interest both for
this question and the previous one.  Hence we pose the following
question.

\begin{problem}
\label{q:lcibrugraph}
Determine if $X_w$ being lci depends solely on the Bruhat graph of
$w$.  If so, find reasonable properties of the Bruhat graph that
characterize when $X_w$ is lci.  Similarly, determine if $X_w$ being
lci at $e_v$ depends solely on the Bruhat graph between $v$ and $w$
and find properties of the graph that characterize when $X_w$ is lci
at $e_v$.
\end{problem}

Carrell~\cite{CarRatSm} proved that $X_w$ is smooth for simply-laced
types if and only if the Bruhat graph is regular, and Carrell and
Kuttler~\cite{CarKut} have given a similar but more complicated
characterization in general.

For the Schubert varieties defined by inclusions, the answer to the
analogous question is in a mysterious sense known.  Hultman, Linusson,
Shareshian, and Sj\"ostrand~\cite{HLSS} studied for any element $w$ of
any Coxeter group a hyperplane arrangement $\mathcal{A}_w$ known as
the inversion arrangement of $w$.  They showed that $re(w)$, the
number of chambers of $\mathcal{A}_w$, is always at most $br(w)$, the
number of elements of $G$ which are less than or equal to $w$ in
Bruhat order.  For the Coxeter group $G=S_n$, they also showed that
$re(w)=br(w)$ if and only if $X_w$ is defined by inclusions.  Their
proof only shows that the permutations satisfying $re(w)=br(w)$ are
given by the same pattern avoidance conditions, and no explanation for
this coincidence is known.  For other Coxeter groups (and in particular
the Weyl groups of other simple Lie groups), Hultman~\cite{Hul} has
shown that $re(w)=br(w)$ if and only if the Bruhat graph of $w$
satisfies a particular criterion.

Another possible further generalization of our work is the following.

\begin{problem}
For each value of $k\geq 1$, characterize the Schubert varieties $X_w$
for which $I_w$ can be generated by at most $\codim(X_w)+k$
generators.  In particular, is this set given by classical pattern avoidance, and, if so, by avoidance of finitely many patterns?
\end{problem}

The condition of failing to be lci by at most $k$ generators is,
like the property of being lci (which is the case $k=0$), an intrinsic
homological property of the ring which holds on Zariski open sets and
does not depend on the embedding.

One can also ask for characterizations of the lci locus both for $\mathrm{GL}_n$
and for other Lie groups in terms of the local $K$-polynomial or
multidegree.

\begin{problem}
\label{prob:lcilocalclass}
Determine if the converse to Corollary~\ref{cor:lcilocalschub} is
true, both at the identity and for a general $v$.  This means
determining if $X_w$ is automatically lci at $e_v$ whenever the
$\mathcal{K}(N_{v,w})$ is a product of $\ell(w)-\ell(v)$ terms of the
form $(1-\mathbf{t}^{\mathbf{\lambda}})$ for some weight $\mathbf{\lambda}$ or whenever
$\mathcal{C}(N_{v,w})$ is a product of $\ell(w)-\ell(v)$ terms of the form $\sum_{j=1}^n\langle \mathbf{a}_j,\mathbf{\lambda}\rangle t_j$.
\end{problem}

Note that Kostant and Kumar~\cite{KosKumK, KosKumCoh} give general
recursive algebraic formulas for $\mathcal{K}(N_{v,w})$ and
$\mathcal{C}(N_{v,w})$, so one can in principle apply this criterion
without knowing any geometry.

Theorem~\ref{thm:kumar} is an equivalence and extends (in the
original formulation by Kumar) to local neighborhoods, so $X_w$ is
smooth at $e_v$ whenever $\mathcal{K}(\mathcal{N}_{v,w})$ is a
particular specific product of $\ell(w)-\ell(v)$ terms of the form
$(1-\mathbf{t}^{\mathbf{\lambda}})$ or equivalently whenever
$\mathcal{C}(\mathcal{N}_{v,w})$ is a particular product of
$\ell(w)-\ell(v)$ linear terms (which in this case are all given by
roots).

Finally we focus on some questions related specifically to the
Schubert varieties defined by inclusions.

\begin{problem}
Is being defined by inclusions equivalent to some intrinsic geometric
property of $X_w$ which does not depend on its embedding into $G/B$?
\end{problem}

Note that $K$-theory and cohomology classes are not intrinsic, so
Corollary~\ref{cor:dbilocalschub} does not answer this question.  However,
Corollary~\ref{cor:dbilocalschub} does give a specific form for
the $K$-polynomial and multidegree when $X_w$ is defined by
inclusions.  Hence we can ask the following.

\begin{problem}
Determine if the converse to Corollary~\ref{cor:dbilocalschub} is
true.  In particular, determine if $X_w$ is automatically defined by
inclusions whenever $\mathcal{K}(S/I_w)$ is a product of $\ell(w)$
terms of the form $(1-\mathbf{t}^{\mathbf{\lambda}})$ for a {\em root} $\mathbf{\lambda}$, or
whenever $\mathcal{C}(S/I_w)$ is a product of $\ell(w)$ terms of the form $\sum_{j=1}^n \langle \mathbf{a}_j,\mathbf{\lambda}\rangle t_j$ where $\mathbf{\lambda}$ is a {\em root}.
\end{problem}

We can show combinatorially that a positive answer to
Question~\ref{prob:lcilocalclass} implies a positive answer to this question.

If the answer to this question is positive, it could possibly be used
to define what it means for $X_w$ to be locally defined by inclusions
at $e_v$.  Since the roots appearing in
Corollary~\ref{cor:dbilocalschub} satisfy the condition on the roots
appearing in Theorem~\ref{thm:kumar} (but not with multiplicity 1),
the conditions in Kumar's theorem specifying the roots that appear in
$\mathcal{K}(\mathcal{N}_{v,w})$ (which are the roots appearing in
$\mathcal{C}(\mathcal{N}_{v,w})$) when $\mathcal{N}_{v,w}$ is smooth
may be helpful.  One can then ask the following.

\begin{problem}
\label{q:localdbi}
Define a local notion of being defined by inclusions using
$K$-polynomials or multidegrees, extending to other Lie groups if
possible.  If possible, link this definition to conditions defining
Richardson varieties or to (relative) cohomology rings of Schubert or
Richardson varieties.
\end{problem}

We hope an answer to this question may help find the link between the
work of Gasharov and Reiner~\cite{GasRei} and that of Hultman,
Linusson, Shareshian, and Sj\"ostrand~\cite{HLSS}.

This hypothetical local notion of being defined by inclusions may also
be helpful in answering Question~\ref{q:lcibrugraph}.  Hence we also
ask for some sensible extension of Hultman's criterion~\cite{Hul}
characterizing elements $w$ in a general Coxeter group that satisfy
$re(w)=br(w)$ to intervals $[v,w]$ in Bruhat order.

\begin{problem}
Extend Hultman's criterion to intervals in Bruhat order, and link this
criterion to some hyperplane arrangement associated to inversions or
to a local notion of being defined by inclusions as in
Question~\ref{q:localdbi}.
\end{problem}

\bibliographystyle{alpha}
\bibliography{lcischub}

\end{document}